\documentclass{article}
\usepackage{amsmath}
\usepackage{amssymb}
\usepackage{amsthm}
\usepackage{cite}
\usepackage{caption}
\usepackage{euscript}
\usepackage[arrow,curve,matrix,arc,2cell]{xy}
\usepackage[utf8]{inputenc}
\usepackage[unicode]{hyperref}
\UseAllTwocells
\DeclareFontFamily{U}{rsfs}{} \DeclareFontShape{U}{rsfs}{n}{it}{<->
rsfs10}{} \DeclareSymbolFont{mscr}{U}{rsfs}{n}{it}
\DeclareSymbolFontAlphabet{\scr}{mscr}
\def\mathscr{\scr}
\begin{document}
\def\e#1\e{\begin{equation}#1\end{equation}}
\def\ea#1\ea{\begin{align}#1\end{align}}
\def\eq#1{{\rm(\ref{#1})}}
\theoremstyle{plain}
\newtheorem{thm}{Theorem}[section]
\newtheorem{prop}[thm]{Proposition}
\newtheorem{lem}[thm]{Lemma}
\newtheorem{cor}[thm]{Corollary}
\newtheorem{quest}[thm]{Question}
\newtheorem{conj}[thm]{Conjecture}
\theoremstyle{definition}
\newtheorem{dfn}[thm]{Definition}
\newtheorem{ex}[thm]{Example}
\newtheorem{rem}[thm]{Remark}
\numberwithin{equation}{section}
\def\dim{\mathop{\rm dim}\nolimits}
\def\Ker{\mathop{\rm Ker}}
\def\Coker{\mathop{\rm Coker}}
\def\sign{\mathop{\rm sign}\nolimits}
\def\id{\mathop{\rm id}\nolimits}
\def\SO{\mathop{\rm SO}\nolimits}
\def\SF{\mathop{\rm SF}\nolimits}
\def\Or{\mathop{\rm Or}\nolimits}
\def\ad{\mathop{\rm ad}\nolimits}
\def\Hom{\mathop{\rm Hom}\nolimits}
\def\Map{\mathop{\rm Map}\nolimits}
\def\Crit{\mathop{\rm Crit}\nolimits}
\def\Hol{\mathop{\rm Hol}\nolimits}
\def\Iso{\mathop{\rm Iso}\nolimits}
\def\Hess{\mathop{\rm Hess}\nolimits}
\def\Pd{\mathop{\rm Pd}\nolimits}
\def\Aut{\mathop{\rm Aut}\nolimits}
\def\Diff{\mathop{\rm Diff}\nolimits}
\def\Flag{\mathop{\rm Flag}\nolimits}
\def\FlagSt{\mathop{\rm FlagSt}\nolimits}
\def\dOrb{{\mathop{\bf dOrb}}}
\def\dMan{{\mathop{\bf dMan}}}
\def\mKur{{\mathop{\bf mKur}}}
\def\Kur{{\mathop{\bf Kur}}}
\def\Re{\mathop{\rm Re}}
\def\Im{\mathop{\rm Im}}
\def\SU{\mathop{\rm SU}}
\def\Sp{\mathop{\rm Sp}}
\def\Spin{\mathop{\rm Spin}}
\def\GL{\mathop{\rm GL}}
\def\ind{\mathop{\rm ind}}
\def\area{\mathop{\rm area}}
\def\U{{\rm U}}
\def\vol{\mathop{\rm vol}\nolimits}
\def\virt{{\rm virt}}
\def\emb{{\rm emb}}
\def\bs{\boldsymbol}
\def\ge{\geqslant}
\def\le{\leqslant\nobreak}
\def\O{{\mathbin{\mathcal O}}}
\def\cA{{\mathbin{\mathcal A}}}
\def\cB{{\mathbin{\mathcal B}}}
\def\cC{{\mathbin{\mathcal C}}}
\def\cD{{\mathbin{\scr D}}}
\def\cDHS{{\mathbin{\scr D}_{\Q HS}}}
\def\cE{{\mathbin{\mathcal E}}}
\def\cF{{\mathbin{\mathcal F}}}
\def\cG{{\mathbin{\mathcal G}}}
\def\cH{{\mathbin{\mathcal H}}}
\def\cI{{\mathbin{\mathcal I}}}
\def\cJ{{\mathbin{\mathcal J}}}
\def\cK{{\mathbin{\mathcal K}}}
\def\cL{{\mathbin{\mathcal L}}}
\def\cM{{\mathbin{\mathcal M}}}
\def\bcM{{\mathbin{\bs{\mathcal M}}}}
\def\cN{{\mathbin{\mathcal N}}}
\def\cO{{\mathbin{\mathcal O}}}
\def\cP{{\mathbin{\mathcal P}}}
\def\cS{{\mathbin{\mathcal S}}}
\def\cT{{\mathbin{\mathcal T}}}
\def\cU{{\mathbin{\mathcal U}}}
\def\cQ{{\mathbin{\mathcal Q}}}
\def\cW{{\mathbin{\mathcal W}}}
\def\C{{\mathbin{\mathbb C}}}
\def\bD{{\mathbin{\mathbb D}}}
\def\F{{\mathbin{\mathbb F}}}
\def\H{{\mathbin{\mathbb H}}}
\def\N{{\mathbin{\mathbb N}}}
\def\Q{{\mathbin{\mathbb Q}}}
\def\R{{\mathbin{\mathbb R}}}
\def\bS{{\mathbin{\mathbb S}}}
\def\Z{{\mathbin{\mathbb Z}}}
\def\sF{{\mathbin{\mathscr F}}}
\def\al{\alpha}
\def\be{\beta}
\def\ga{\gamma}
\def\de{\delta}
\def\io{\iota}
\def\ep{\epsilon}
\def\la{\lambda}
\def\ka{\kappa}
\def\th{\theta}
\def\ze{\zeta}
\def\up{\upsilon}
\def\vp{\varphi}
\def\si{\sigma}
\def\om{\omega}
\def\De{\Delta}
\def\La{\Lambda}
\def\Si{\Sigma}
\def\Th{\Theta}
\def\Om{\Omega}
\def\Ga{\Gamma}
\def\Up{\Upsilon}
\def\pd{\partial}
\def\ts{\textstyle}
\def\st{\scriptstyle}
\def\sst{\scriptscriptstyle}
\def\w{\wedge}
\def\sm{\setminus}
\def\bu{\bullet}
\def\op{\oplus}
\def\ot{\otimes}
\def\ov{\overline}
\def\ul{\underline}
\def\bigop{\bigoplus}
\def\bigot{\bigotimes}
\def\iy{\infty}
\def\es{\emptyset}
\def\ra{\rightarrow}
\def\Ra{\Rightarrow}
\def\Longra{\Longrightarrow}
\def\ab{\allowbreak}
\def\longra{\longrightarrow}
\def\hookra{\hookrightarrow}
\def\dashra{\dashrightarrow}
\def\t{\times}
\def\ci{\circ}
\def\ti{\tilde}
\def\d{{\rm d}}
\def\ha{{\ts\frac{1}{2}}}
\def\md#1{\vert #1 \vert}
\def\bmd#1{\big\vert #1 \big\vert}
\def\ms#1{\vert #1 \vert^2}
\def\nm#1{\Vert #1 \Vert}
\title{Conjectures on counting associative \\ 3-folds in $G_2$-manifolds}
\author{Dominic Joyce}
\date{}
\maketitle

\begin{abstract} There is a strong analogy between compact, torsion-free $G_2$-manifolds $(X,\vp,*\vp)$ and Calabi--Yau 3-folds $(Y,J,g,\om)$. We can also generalize $(X,\vp,*\vp)$ to `tamed almost $G_2$-manifolds' $(X,\vp,\psi)$, where we compare $\vp$ with $\om$ and $\psi$ with $J$. Associative 3-folds in $X$, a special kind of minimal submanifold, are analogous to $J$-holomorphic curves in~$Y$. 

Several areas of Symplectic Geometry -- Gromov--Witten theory, Quantum Cohomology, Lagrangian Floer cohomology, Fukaya categories -- are built using `counts' of moduli spaces of $J$-holomorphic curves in $Y$, but give an answer depending only on the symplectic manifold $(Y,\om)$, not on the (almost) complex structure $J$.

We investigate whether it may be possible to define interesting invariants of tamed almost $G_2$-manifolds $(X,\vp,\psi)$ by `counting' compact associative 3-folds $N\subset X$, such that the invariants depend only on $\vp$, and are independent of the 4-form $\psi$ used to define associative 3-folds.

We conjecture that one can define a superpotential $\Phi_\psi:\cU\ra\La_{>0}$ `counting' associative $\Q$-homology 3-spheres $N\subset X$ which is deformation-invariant in $\psi$ for $\vp$ fixed, up to certain reparametrizations $\Up:\cU\ra\cU$ of the base $\cU=\Hom(H_3(X;\Z),1+\La_{>0})$, where $\La_{>0}$ is a Novikov ring. Using this we define a notion of `$G_2$ quantum cohomology'.

We also discuss Donaldson and Segal's proposal \cite[\S 6.2]{DoSe} to define invariants `counting' $G_2$-instantons on tamed almost $G_2$-manifolds, with `compensation terms' counting weighted pairs of a $G_2$-instanton and an associative 3-fold, and suggest some modifications to it.
\end{abstract}

\setcounter{tocdepth}{2}
\tableofcontents
\section{Introduction}
\label{ca1}

Let $(Y,\om)$ be a compact symplectic manifold. Several areas of Symplectic Geometry --- Gromov--Witten invariants \cite{FuOn,HWZ,McSa}, Quantum Cohomology \cite{McSa}, Lagrangian Floer cohomology \cite{Fuka2,FOOO}, Fukaya categories \cite{Seid}, and so on --- involve choosing an almost complex structure $J$ on $Y$ compatible with $\om$, `counting' moduli spaces $\cM$ of $J$-holomorphic curves in $Y$ satisfying some conditions, and using the `numbers' $[\cM]_\virt$ and homological algebra to define the theory.

A remarkable feature of these theories is that although the family $\cJ$ of possible choices of $J$ is infinite-dimensional, and two $J_1,J_2$ in $\cJ$ may be very far apart, nonetheless the theory is independent of choice of $J$ (up to a suitable notion of equivalence), and so depends only on~$(Y,\om)$.

These areas are related to String Theory, and are driven by conjectures made by physicists. Oversimplifying rather, String Theorists tell us that if $(Y,J,g,\om)$ is a Calabi--Yau 3-fold, then the String Theory of $Y$ (a huge structure) has a `topological twisting', the `A model', a smaller and simpler theory. The A model depends only on the symplectic manifold $(Y,\om)$, not on the other geometric structures $J,g,\Om$, and encodes data including the Gromov--Witten invariants, Quantum Cohomology, and Fukaya category of $(Y,\om)$.

We wish to explore the possibility that an analogue of these ideas may work for compact $G_2$-manifolds. As in \S\ref{ca2}, if $(X,g)$ is a Riemannian 7-manifold with holonomy group $G_2$ then $X$ has a natural closed 3-form $\vp$ and Hodge dual closed 4-form $*\vp$, in a local normal form that we call `positive' 3- and 4-forms. There are two classes of special submanifolds in $X$, `associative 3-folds' $N\subset X$ calibrated by $\vp$, and `coassociative 4-folds' $C\subset X$ calibrated by $*\vp$. 

There is a well known analogy:
\ea
&\text{Calabi--Yau 3-folds $(Y,J,h)$} &&\!\!\!\!\!\!\leftrightarrow\!\!\! &&\text{Torsion-free $G_2$-manifolds $(X,\vp,*\vp)$} 
\nonumber\\
&\text{$J$-holomorphic curves in $Y$} &&\!\!\!\!\!\!\leftrightarrow\!\!\! &&\text{associative 3-folds in $X$} 
\label{ca1eq1}\\
&\text{(Special) Lagrangian 3-folds in $Y$} &&\!\!\!\!\!\!\leftrightarrow\!\!\! &&\text{coassociative 4-folds in $X$.} 
\nonumber
\ea

Torsion-free $G_2$-manifolds $(X,\vp,*\vp)$ are a rather restrictive class. Following Donaldson and Segal \cite[\S 3--\S 4]{DoSe}, we will work with the much larger class of {\it tamed almost-$G_2$-manifolds}, or {\it TA-$G_2$-manifolds}, $(X,\vp,\psi)$, which have a closed $G_2$ 3-form $\vp$ and a compatible closed $G_2$ 4-form $\psi$ on $X$, but need not have $\psi=*\vp$. We call $\vp,\psi$ {\it good\/} if they extend to a TA-$G_2$-manifold $(X,\vp,\psi)$. Then we can extend the analogy \eq{ca1eq1}, adding the lines:
\ea
&\text{Symplectic form $\om$ on $Y$} &&\!\!\!\!\leftrightarrow\!\!\! &&\text{Good 3-form $\vp$ on $X$} \nonumber\\
&\text{(Almost) complex structure $J$ on $Y$} &&\!\!\!\!\leftrightarrow\!\!\! &&\text{Good 4-form $\psi$ on $X$} 
\label{ca1eq2}\\
&\begin{subarray}{l}\ts\text{Symplectic manifold $(Y,\om)$ with} \\
\ts\text{compatible almost complex structure $J$}\end{subarray} &&\!\!\!\!\leftrightarrow\!\!\! &&\text{TA-$G_2$-manifold $(X,\vp,\psi)$.}\!\!\!\!{} 
\nonumber
\ea

Here we compare $\vp$ with $\om$ and $\psi$ with $J$ because the notion of associative 3-fold $N$ in $(X,\vp,\psi)$ depends only on $X,\psi$, not on $\vp$, but $N$ has volume $[\vp]\cdot[N]$ for $[\vp]\in H^3_{\rm dR}(X;\R)$ and~$[N]\in H_3(X;\Z)$. Following analogy \eq{ca1eq1}--\eq{ca1eq2}, and being very optimistic, one might hope to construct:
\begin{itemize}
\setlength{\itemsep}{0pt}
\setlength{\parsep}{0pt}
\item[(a)] Gromov--Witten type invariants $GW_{\psi,\al}\in\Q$ counting associative 3-folds $N$ in a TA-$G_2$-manifold $(X,\vp,\psi)$ in homology class $[N]=\al\in H_3(X;\Z)$.
\item[(b)] A `quantum cohomology algebra' $QH^*(X;\La_{\ge 0})$ for a TA-$G_2$-manifold $(X,\vp,\psi)$, defined by modifying usual cohomology $H^*(X;\La_{\ge 0})$ by terms involving counting associative 3-folds in $X$ passing through given cycles. 
\item[(c)] `Floer cohomology groups' or `Fukaya categories' for coassociative 4-folds $C$ in $(X,\vp,\psi)$, defined by counting associative 3-folds $N$ in $X$ with boundary $\pd N\subset C$, as discussed by Leung, Wang and Zhu~\cite{LWZ1,LWZ2}.
\end{itemize}

We particularly want anything we define to be {\it unchanged by continuous deformations of\/ $(\vp,\psi)$ which fix the cohomology class $[\vp]=\ga$ in\/} $H^3_{\rm dR}(X;\R)$, as this is our analogue of symplectic theories being independent of choice of almost complex structure $J$, and is our criterion for having found an interesting, `topological' theory, in the style of invariant theories in Symplectic Geometry.

Our conjectural answers to these are:
\begin{itemize}
\setlength{\itemsep}{0pt}
\setlength{\parsep}{0pt}
\item[(a$)'$] We outline how to define numbers $GW_{\psi,\al}\in\Q$ `counting' associative $\Q$-homology 3-spheres $N$ in $(X,\vp,\psi)$ with $[N]=\al\in H_3(X;\Z)$ and $\psi$ generic. These $GW_{\psi,\al}$ {\it depend on arbitrary choices}, and {\it are not invariant under deformations of $(\vp,\psi)$ fixing\/} $[\vp]\in H^3_{\rm dR}(X;\R)$. 

However, we expect the {\it family\/} of $GW_{\psi,\al}$ for all $\al\in H_3(X;\Z)$ to have some interesting deformation-invariant features, as in Conjecture \ref{ca1conj1}. In particular, the $GW_{\psi,\al}$ should be combined in a superpotential $\Phi_\psi:\cU\ra \La_{>0}$ as in \eq{ca1eq3} which is independent of choices up to reparametrization by a class of automorphisms of the base~$\cU$.
\item[(b$)'$] If this superpotential $\Phi_\psi$ has a critical point $\th\in\cU$, we explain how to define `$G_2$ quantum cohomology' $QH^*_\th(X;\La_{\ge 0})$, a supercommutative algebra over the Novikov ring $\La_{\ge 0},$ which is a deformation of $H^*(X;\La_{\ge 0})$, expected to be deformation-invariant up to isomorphism.

If a critical point $\th$ exists, we say that $(X,\vp,\psi)$ is {\it unobstructed}. This is a condition similar to Lagrangian Floer cohomology of a Lagrangian being unobstructed in Fukaya, Oh, Ohta and Ono~\cite{Fuka2,FOOO}.
\item[(c$)'$] We expect that it is {\it not possible\/} to construct a deformation-invariant version of Lagrangian Floer theory for coassociatives $C$ in $X$, based on counting associatives $N$ in $X$ with $\pd N\subset C$, for reasons explained in~\S\ref{ca62}.
\end{itemize}

The next conjecture explains (a$)'$ in more detail. We need the following notation. Let $\F$ be the field $\Q,\R$ or $\C$. Write $\La$ for the Novikov ring over $\F$:
\begin{equation*}
\La=\bigl\{\ts\sum_{i=1}^\iy c_iq^{\al_i}: \text{$c_i\in\F,$ $\al_i\in\R,$ $\al_i\ra \iy$ as $i\ra\iy$}\bigr\}, 
\end{equation*}
with $q$ a formal variable. Then $\La$ is a commutative $\F$-algebra. Define $v:\La\ra\R\amalg\{\iy\}$ by $v(\la)$ is the least $\al\in\R$ with the coefficient of $q^\al$ in $\la$ nonzero for $\la\in\La\sm\{0\}$, and $v(0)=\iy$. Write $\La_{\ge 0}\subset\La$ for the subalgebra of $\la\in\La$ with $v(\la)\ge 0$, and $\La_{>0}\subset\La_{\ge 0}$ for the ideal of $\la\in\La$ with $v(\la)>0$. 

Then $\La$ is a {\it complete non-Archimedean field\/} in the sense of Bosch, G\"untzer and Remmert \cite[\S A]{BGR}, with valuation $\nm{\la}=2^{-v(\la)}$, so we can consider {\it rigid analytic spaces\/} over $\La$ as in \cite[\S C]{BGR}. These are like schemes over $\La$, except that polynomial functions on schemes are replaced by convergent power series.

\begin{conj}[see Conjecture \ref{ca7conj1}] Let $X$ be a compact, oriented\/ $7$-manifold. Consider $1+\La_{>0}\subset\La$ as a group under multiplication in $\La$. Write
\begin{equation*}
\cU=\Hom\bigl(H_3(X;\Z),1+\La_{>0}\bigr)
\end{equation*}
for the set of group morphisms $\th:H_3(X;\Z)\ra 1+\La_{>0}$. By choosing a basis $e_1,\ldots,e_n$ for $H_3(X;\Z)/$torsion, where $n=b_3(X),$ we can identify $\cU\cong\La_{>0}^n$ by $\th\cong (\la_1,\cdots,\la_n)$ if\/ $\th(e_i)=1+\la_i$ for $i=1,\ldots,n,$ where $\La_{>0}$ is the open unit ball in $\La$ in the norm $\nm{\,.\,}$. We regard\/ $\cU$ as a \begin{bfseries}smooth rigid analytic space\end{bfseries} over $\La,$ as in Bosch, G\"untzer and Remmert\/~{\rm\cite{BGR}}.

Let\/ $\ga\in H^3_{\rm dR}(X;\R),$ and write $\cF_\ga$ for the set of closed\/ $4$-forms $\psi$ on $X$ such that there exists a closed\/ $3$-form $\vp$ on $X$ with\/ $[\vp]=\ga\in H^3_{\rm dR}(X;\R),$ for which $(X,\vp,\psi)$ is a TA-$G_2$-manifold, with the given orientation on $X$. Let\/ $\psi\in\cF_\ga$ be generic. Then we can define a superpotential\/ $\Phi_\psi:\cU\ra \La_{>0},$ of the form
\e
\Phi_\psi(\th)=\sum_{\al\in H_3(X;\Z):\ga\cdot\al>0}GW_{\psi,\al}\, q^{\ga\cdot\al}\,\th(\al), 
\label{ca1eq3}
\e
where $GW_{\psi,\al}\in\Q$ is a weighted count of associative $\Q$-homology $3$-spheres in $(X,\vp,\psi)$ with homology class $\al$. The $GW_{\psi,\al}$ are \begin{bfseries}not independent of choices\end{bfseries}, and are \begin{bfseries}not invariant under deformations of\/ $\psi$ in\end{bfseries} $\cF_\ga$. So \begin{bfseries}they are not enumerative invariants in the usual sense\end{bfseries}.

Nonetheless, the whole superpotential\/ $\Phi_\psi$ does have the following invariance property. If\/ $\psi_0,\psi_1$ are generic elements of the same connected component of\/ $\cF_\ga$ (we allow $\psi_0=\psi_1$), and\/ $\Phi_{\psi_0},\Phi_{\psi_1}$ are any choices for the superpotentials for\/ $\psi_0,\psi_1,$ then there is a \begin{bfseries}quasi-identity morphism\end{bfseries} $\Up:\cU\ra\cU,$ a special kind of isomorphism of rigid analytic spaces defined in\/ {\rm\S\ref{ca71},} with\/~$\Phi_{\psi_1}=\Phi_{\psi_0}\ci\Up$. 
\label{ca1conj1}
\end{conj}

Here we work over the Novikov ring $\La_{>0}$, as in \cite{Fuka2,FOOO}, as our theory involves infinite sums such as \eq{ca1eq3}, but we do not know these sums converge in the usual sense, so we have to use formal power series. If we knew all our formal sums converged, we could work over $\R$ or $\C$ instead, with $q\in\R,\C$ small.

Conjecture \ref{ca1conj1} implies that any information which depends on $\Phi_\psi$ only up to reparametrizations by quasi-identity morphisms $\Up:\cU\ra\cU$ is deformation-invariant. For example, the least $A>0$ such that $GW_{\psi,\al}\ne 0$ for $\al\in H_3(X;\Z)$ with $\ga\cdot\al=A$ should be deformation-invariant, and the values of $GW_{\psi,\al}$ for all $\al\in H_3(X;\Z)$ with $\ga\cdot\al=A$ should also be deformation-invariant. Section \ref{ca76} outlines how to define a `$G_2$ quantum cohomology algebra' $QH^*_\th(X;\La_{\ge 0})$ depending on a  critical point $\th$ of $\Phi_\psi$ in $\cU$, which should be deformation-invariant.

The message of this paper is both positive and negative. On the positive side, there is (the author believes) some nontrivial deformation-invariant information from counting associatives. On the negative side, not that much information survives -- much less than for $J$-holomorphic curves in Symplectic Geometry -- and conjectures more optimistic than Conjecture \ref{ca1conj1} are likely to be false.

The reasoning behind Conjecture \ref{ca1conj1} is as follows. Let $(X,\vp_t,\psi_t),$ $t\in[0,1]$ be a smooth 1-parameter family of TA-$G_2$-manifolds. Then as in \S\ref{ca27} we can form moduli spaces $\cM(\cN,\al,\psi_t)$ of compact associative 3-folds $N$ in $(X,\vp_t,\psi_t)$ of diffeomorphism type $\cN$ and homology class $[N]=\al\in H_3(X;\Z)$. To define enumerative invariants for associative 3-folds which are the same for $(X,\vp_0,\psi_0)$ and $(X,\vp_1,\psi_1)$, we need to understand how the moduli spaces $\cM(\cN,\al,\psi_t)$ can change as $t$ increases through~$[0,1]$.

The typical reason why moduli spaces change is that for some $t_0\in(0,1)$ there exists a family $N_t$ for $t\in[0,t_0]$, where $N_t$ for $t\in[0,t_0)$ is a compact associative 3-fold in $(X,\vp_t,\psi_t)$ in homology class $\al$ depending smoothly on $t$, and $N_{t_0}=\lim_{t\ra t_{0-}}N_t$ is a {\it singular\/} associative 3-fold, and no $N_t$ for $t\in (t_0,1]$ exist, so that a point in $\cM(\cN,\al,\psi_t)$ disappears as $t$ crosses $t_0$ in~$[0,1]$.

Let us suppose that $(X,\vp_t,\psi_t),$ $t\in[0,1]$ is a {\it generic\/} 1-parameter family. Then the singularities of $N_{t_0}$ are not arbitrary. To each singularity type $\cS$ of associative 3-folds we can assign an {\it index\/} $\ind\cS$, which is the codimension in which singularities of type $\cS$ occur in families of associative 3-folds over generic families of $G_2$-structures. In our problem $N_{t_0}$ can only have index 1, so it is enough for us to understand index 1 singularities of associative 3-folds. 

Sections \ref{ca4}--\ref{ca5} and \ref{ca72} describe several kinds of index 1 singularity of associative 3-folds. These are the only kinds the author knows, and may perhaps be the only kinds there are. They all definitely change the number of associative 3-folds, and so mean that na\"\i ve counts of associative 3-folds cannot be deformation-invariant.

In \S\ref{ca7} we assume that moduli spaces of compact associatives in $(X,\vp,\psi)$ have good compactness, smoothness, and orientation properties, and that their only boundary behaviour comes from the six kinds of index 1 singularity described in \S\ref{ca72}. Under these very strong assumptions, we explain how by counting associative 3-folds in cunning ways, we can still extract deformation-invariant information from the numbers of associative 3-folds as in Conjecture \ref{ca1conj1}, as we arrange that the changes under index 1 singularities cancel out.

As in \cite{DoSe}, $G_2$-{\it instantons\/} on a TA-$G_2$-manifold $(X,\vp,\psi)$ are connections $A$ on principal $G$-bundles $P\ra X$ whose curvature $F_A$ satisfies $F_A\w\psi=0$. In our analogy \eq{ca1eq1}--\eq{ca1eq2}, we can add the line:
\begin{equation*}
\text{Hermitian--Yang--Mills vector bundles on $Y$} \;\> \leftrightarrow \;\>\text{$G_2$-instantons on $(X,\vp,*\vp)$.}
\end{equation*}
Donaldson and Segal \cite[\S 6.2]{DoSe} proposed a programme to define invariants counting $G_2$-instantons, which would hopefully be unchanged under deformations of $(\vp,\psi)$, and would be analogues of Donaldson--Thomas invariants of Calabi--Yau 3-folds \cite{Joyc21,JoSo}. It is currently under investigation by Menet, Nordstr\"om, S\'a Earp, Walpuski, and others \cite{MNS,SaEa,SEWa,Walp1,Walp2,Walp3,Walp4}. As in \cite[\S 6.2]{DoSe}, to define invariants of $(X,\vp,\psi)$ unchanged under deformations of $\psi$ will require the inclusion of `compensation terms' counting solutions of some equation on associative 3-folds $N$ in $X$, to compensate for bubbling of $G_2$-instantons on associative 3-folds. 

Section \ref{ca8} discusses several aspects of this programme. We make a proposal for how to define canonical orientations for $G_2$-instanton moduli spaces, based on the ideas in \S\ref{ca3} on orienting associative moduli spaces. And we argue that counting $G_2$-instantons on $(X,\vp,\psi)$ in a deformation-invariant way should only be possible if counting associative 3-folds in $(X,\vp,\psi)$ is `unobstructed' -- the superpotential $\Phi_\psi$ has a critical point $\th$, as in (b$)'$ above -- and we choose some such $\th$, similar to choosing a `bounding cochain' for a Lagrangian in the Lagrangian Floer theory of Fukaya, Oh, Ohta and Ono~\cite{Fuka2,FOOO}.

On the relation with String Theory and M-theory, we can ask:

\begin{quest} Is there some good notion of `topological twisting' for M-theory or String Theory on TA-$G_2$-manifolds $(X,\vp,\psi),$ which includes the superpotential\/ $\Phi_\psi,$ and\/ $G_2$ quantum cohomology $QH^*_\th(X;\La_{\ge 0}),$ and modified Donaldson--Segal invariants, proposed above?
\label{ca1quest2}
\end{quest}

See de Boer et al.\ \cite{BMSS,BNS1,BNS2} for a discussion of topological $G_2$-strings. Superpotentials $\Phi$ counting associative 3-folds similar to those in Conjecture \ref{ca1conj1} were discussed in M-theory by Acharya \cite{Acha1,Acha2} and Harvey and Moore~\cite{HaMo}. 

Throughout \S\ref{ca2}--\S\ref{ca7} we state conjectures on how the author expects the mathematics to work. These are not of uniform difficulty. For some of them, the author or one of his friends could easily write down a proof, if we were not too busy writing grant proposals. However, our main conjecture includes some aspects which are seriously difficult, and the author has no idea how to prove:

\begin{itemize}
\setlength{\itemsep}{0pt}
\setlength{\parsep}{0pt}
\item Implicit in Conjecture \ref{ca1conj1} is the idea that the only {\it index one\/} singularities of associative 3-folds (i.e.\ the only singularities that can occur in associatives in generic 1-parameter families of TA-$G_2$-manifolds $(X,\vp_t,\psi_t)$, $t\in[0,1]$) are those described in \S\ref{ca72}. This is difficult because it requires some measure of control over {\it all possible\/} singularities of associative 3-folds, as described using Geometric Measure Theory, for instance.
\item A proper understanding of the multiple cover phenomena for associatives in \S\ref{ca72}(F) also looks rather difficult, but is essential for Conjecture~\ref{ca1conj1}.
\end{itemize}

We emphasize that this paper is very speculative, and little in it is actually proved. There are a few bits which are both new and more-or-less rigorous, in particular, some ideas on TA-$G_2$-manifolds in \S\ref{ca25}, and on canonical flags, flag structures, and orientations for associative moduli spaces $\cM(\cN,\al,\psi)$ in~\S\ref{ca3}.

This paper is similar to the author's paper \cite{Joyc4}, which made conjectures on invariants counting special Lagrangian 3-folds in Calabi--Yau 3-folds. 
\smallskip

\noindent{\it Acknowledgements.} This research was partly funded by a Simons Collaboration Grant on `Special Holonomy in Geometry, Analysis and Physics'. I would like to thank Bobby Acharya, Robert Bryant, Alexsander Doan, Simon Donaldson, Mark Haskins, Andriy Haydys, Johannes Nordstr\"om, Matthias Ohst, and Thomas Walpuski for helpful conversations, and a referee for useful comments.

\section{\texorpdfstring{Geometry of $G_2$-manifolds}{Geometry of G₂-manifolds}}
\label{ca2}

We begin by introducing $G_2$-manifolds and associative and coassociative submanifolds. Some references for \S\ref{ca21}--\S\ref{ca23} are the author's books \cite{Joyc3,Joyc20}. Parts of \S\ref{ca25}--\S\ref{ca26} on TA-$G_2$-manifolds and on associative moduli spaces are new.

\subsection{\texorpdfstring{$G_2$-manifolds}{G₂-manifolds}}
\label{ca21}

Let $(X,g)$ be a connected Riemannian $n$-manifold, and fix a basepoint $x\in X$. The {\it holonomy group\/} $\Hol(g)$ of $g$ is the group of isometries of $T_xX$
generated by parallel transport around smooth loops $\ga:[0,1]\ra X$ with $\ga(0)=\ga(1)=x$. We consider $\Hol(g)$ to be a (Lie) subgroup of ${\rm O}(n)$, defined up to conjugation by elements of ${\rm O}(n)$. Then $\Hol(g)$ is independent of the choice of base point $x$.

The possible holonomy groups were classified by Berger \cite{Berg} in 1955. If $X$ is simply-connected and $g$ is irreducible and nonsymmetric, then $\Hol(g)$ is one of
\begin{gather*}
\SO(n), \quad \U(m),\SU(m)\; (n=2m,\; m\ge 2),\\ \Sp(m),\Sp(m)\Sp(1)\; (n=4m,\; m\ge 2),\quad G_2\;(n=7),\quad\text{or}\quad \Spin(7)\; (n=8).
\end{gather*}

We are concerned with the exceptional holonomy group $G_2$ in 7 dimensions. In 1987, Bryant \cite{Brya} first used the theory of exterior differential systems to show that locally there exist many metrics with holonomy $G_2$. In 1989, Bryant and Salamon \cite{BrSa} found explicit examples of complete metrics with holonomy $G_2$ on noncompact manifolds. Constructions of compact 7-manifolds with holonomy $G_2$ were given by the author \cite{Joyc1,Joyc2,Joyc3} in 1993 and 2000, by Kovalev \cite{Kova} in 2000, and by Corti, Haskins, Nordstr\"om and Pacini \cite{CHNP} in 2012. 

Let $(x_1,\ldots,x_7)$ be coordinates on $\R^7$. Write $\d{\bf x}_{ij\ldots l}$ for the exterior form $\d x_i\w\d x_j\w\cdots\w\d x_l$ on $\R^7$. Define a 3-form $\vp_0$ on $\R^7$ by
\e
\vp_0=\d{\bf x}_{123}+\d{\bf x}_{145} +\d{\bf x}_{167}+\d{\bf
x}_{246} -\d{\bf x}_{257}-\d{\bf x}_{347}-\d{\bf x}_{356}.
\label{ca2eq1}
\e
The subgroup of $\GL(7,\R)$ preserving $\vp_0$ is the exceptional
Lie group $G_2$. It is compact, connected, simply-connected,
semisimple and 14-dimensional, and it also preserves the Hodge dual 4-form
\e
*\vp_0=\d{\bf x}_{4567}+\d{\bf x}_{2367}+\d{\bf x}_{2345}+
\d{\bf x}_{1357}-\d{\bf x}_{1346}-\d{\bf x}_{1256}-\d{\bf x}_{1247},
\label{ca2eq2}
\e
the Euclidean metric $g_0=\d x_1^2+\cdots+\d x_7^2$,
and the orientation on $\R^7$. The subgroup of $\GL(7,\R)$ preserving $*\vp_0$ is $\{\pm 1\}\t G_2$, but the subgroup preserving $*\vp_0$ and the orientation on $\R^7$ is~$G_2$.

Let $X$ be a 7-manifold, and $\vp\in\Ga^\iy(\La^3T^*X)$ a smooth 3-form on $X$. We call $\vp$ {\it positive\/} if for each $x\in X$ there exists an isomorphism $T_xX\cong\R^7$
identifying $\vp\vert_x$ with $\vp_0$ in \eq{ca2eq1}. This is an open condition on $\vp$. If $\vp$ is positive then the set of isomorphisms $T_xX\cong\R^7$ identifying $\vp\vert_x\cong\vp_0$ for all $x\in X$ is a principal subbundle $P_\vp$ of the frame bundle $F\ra X$ of $X$ with structure group $G_2$. That is, $P_\vp$ is a $G_2$-{\it structure\/} on $X$. This gives a 1-1 correspondence between positive 3-forms and $G_2$-structures on a 7-manifold~$X$. 

Similarly, we call a 4-form $\psi\in\Ga^\iy(\La^4T^*X)$ {\it positive\/} if for each $x\in X$ there exists an isomorphism $T_xX\cong\R^7$ identifying $\psi\vert_x$ with $*\vp_0$ in \eq{ca2eq2}. If we fix an orientation on $X$, the set of oriented isomorphisms $T_xX\cong\R^7$ identifying $\psi\vert_x\cong *\vp_0$ for all $x\in X$ is a $G_2$-structure $P_\psi$ on $X$. This gives a 1-1 correspondence between positive 4-forms and $G_2$-structures on an oriented 7-manifold~$X$. 

A $G_2$-{\it manifold\/} is a 7-manifold $X$ with a $G_2$-structure $P$. As above $P$ corresponds to positive 3- and 4-forms $\vp,*\vp$, and by an abuse of notation we call $(X,\vp,*\vp)$ a $G_2$-manifold. A $G_2$-manifold $(X,\vp,*\vp)$ has an associated Riemannian metric $g$ and orientation.

\begin{prop} Let\/ $(X,\vp,*\vp)$ be a\/ $G_2$-manifold, with associated metric $g$. Then the following are equivalent:
\begin{itemize}
\setlength{\itemsep}{0pt}
\setlength{\parsep}{0pt}
\item[{\rm(i)}] $\Hol(g)\subseteq G_2,$ and\/ $\vp$ is the induced\/ $3$-form,
\item[{\rm(ii)}] $\nabla\vp=0$ on $X,$ where $\nabla$ is the Levi-Civita connection of\/ $g,$ and
\item[{\rm(iii)}] $\d\vp=\d(*\vp)=0$ on $X$.
\end{itemize}
\label{ca2prop1}
\end{prop}

We call $\nabla\vp$ the {\it torsion} of the $G_2$-structure $P_\vp$. If $\nabla\vp=0$ then $(X,\vp,*\vp)$ is called {\it torsion-free}. If $g$ has holonomy $\Hol(g)\subseteq G_2$, then $g$ is Ricci-flat.

\begin{thm} Let\/ $(X,g)$ be a compact Riemannian\/ $7$-manifold with\/ $\Hol(g)\subseteq G_2$. Then $\Hol(g)=G_2$ if and only if\/ $\pi_1(X)$ is finite. In this case the moduli space of metrics with holonomy $G_2$ on $X,$ up to diffeomorphisms isotopic to the identity, is a smooth manifold of dimension~$b^3(X)$.
\label{ca2thm1}
\end{thm}

\subsection{\texorpdfstring{Calabi--Yau 3-folds and $G_2$-manifolds}{Calabi--Yau 3-folds and G₂-manifolds}}
\label{ca22}

Let $(Y,J)$ be a compact complex 3-manifold admitting K\"ahler metrics, with trivial canonical bundle $K_Y\cong\O_Y$. Yau's proof of the Calabi Conjecture implies that each K\"ahler class on $Y$ contains a unique Ricci-flat K\"ahler metric $h$. Then $h$ has holonomy group $\Hol(h)\subseteq\SU(3)$. We call $(Y,J,h)$ a {\it Calabi--Yau\/ $3$-fold}.

The Levi-Civita connection $\nabla$ of $h$ preserves $J,h$, the K\"ahler form $\om$ of $h$, and a holomorphic volume form $\Om$ in $H^0(K_Y)$, which we can scale to have length $\md{\Om}=2^{3/2}$. Then at each point $y\in Y$, there is an isomorphism of complex vector spaces $T_yY\cong\C^3$ identifying $h\vert_y,\om\vert_y,\Om\vert_y$ with $h_0,\om_0,\Om_0$, where
\begin{gather}
h_0=\ms{\d z_1}+\ms{\d z_2}+\ms{\d z_3},\quad \om_0=\ts\frac{i}{2}(\d z_1\w\d\bar z_1+\d z_2\w\d\bar z_2+\d z_3\w\d\bar z_3),
\nonumber\\
\text{and}\quad\Om_0=\d z_1\w\d z_2\w\d z_3,
\label{ca2eq3}
\end{gather}
with $(z_1,z_2,z_3)$ the complex coordinates on $\C^3$.

Calabi--Yau 3-folds and $G_2$-manifolds are connected in the following way. Identify $\R^7\cong\R\t\C^3$ by $(x_1,\ldots,x_7)\cong(x_1,x_2+ix_3,x_4+ix_5,x_6+ix_7)$. Then $g_0,\vp_0,*\vp_0$ in \S\ref{ca21} are related to $h_0,\om_0,\Om_0$ in \eq{ca2eq3} by
\e
g_0=\d x_1^2+h_0,\;\> \vp_0=\d x_1\w\om_0+\Re\Om_0,\;\> *\vp_0=\ha\om_0\w\om_0-\d x_1\w\Im\Om_0.
\label{ca2eq4}
\e

Therefore, if $(Y,J,h)$ is a Calabi--Yau 3-fold with K\"ahler form $\om$ and holomorphic volume form $\Om$, if we define $X=\R\t Y$ or $X=\cS^1\t Y$, with $x$ the coordinate on $\R$ or $\cS^1=\R/\Z$, and set
\begin{equation*}
g=\d x^2+h,\;\> \vp=\d x\w\om+\Re\Om,\;\> *\vp=\ha\om\w\om-\d x\w\Im\Om,
\end{equation*}
then $(X,\vp,*\vp)$ is a torsion-free $G_2$-manifold with metric $g$. There is a strong analogy between torsion-free $G_2$-manifolds and Calabi--Yau 3-folds.

\subsection{Calibrated submanifolds}
\label{ca23}

The next definition is due to Harvey and Lawson~\cite{HaLa}.

\begin{dfn} Let $(X,g)$ be a Riemannian manifold, and $\vp$ a closed $k$-form on $X$. We call $\vp$ a {\it calibration\/} if for every $x\in X$ and $k$-dimensional subspace $V\subseteq T_xX$ we have $\bmd{\vp\vert_V}\le 1$. If $\vp$ is a calibration, we say that an oriented $k$-dimensional submanifold $N$ in $X$ is {\it calibrated with respect to\/} $\vp$ if $\vp\vert_{T_xN}=\vol_{T_xN}$ for all $x\in N$. Calibrated submanifolds are naturally oriented, and a compact calibrated submanifold $N$ is volume-minimizing in its homology class, with volume $[\vp]\cdot[N]$, so calibrated submanifolds are minimal submanifolds.
\label{ca2def1}
\end{dfn}

Calibrated geometry is a natural companion to the theory of holonomy groups. If $(X,g)$ is a Riemannian manifold with special holonomy $H\subset{\rm O}(n)$, it will have natural constant $k$-forms $\vp$ corresponding to $H$-invariant $k$-forms in $\La^k(\R^n)^*$, and if $\vp$ is rescaled appropriately it is a calibration. Thus, manifolds with special holonomy have interesting special classes of minimal submanifolds.

Let $(X,\vp,*\vp)$ be a torsion-free $G_2$-manifold, with metric $g$ and 4-form $*\vp$. Then as in Harvey and Lawson~\cite[\S IV]{HaLa}:
\begin{itemize}
\setlength{\itemsep}{0pt}
\setlength{\parsep}{0pt}
\item[(a)] $\vp$ is a calibration on $(X,g)$, and its calibrated submanifolds are called {\it associative\/ $3$-folds}.
\item[(b)] $*\vp$ is a calibration on $(X,g)$, and its calibrated submanifolds are called {\it coassociative\/ $4$-folds}. If $C$ is a 4-dimensional submanifold of $X$ then $C$ is coassociative (with some unique orientation) if and only if~$\vp\vert_C=0$.
\end{itemize}
Examples of compact associative 3-folds and coassociative 4-folds in compact 7-manifolds with holonomy $G_2$ can be found in the author~\cite[\S 12.6]{Joyc3}.

Similarly, there are three kinds of calibrated submanifolds in a Calabi--Yau 3-fold $(Y,J,h)$ with K\"ahler form $\om$ and holomorphic volume form~$\Om$:
\begin{itemize}
\setlength{\itemsep}{0pt}
\setlength{\parsep}{0pt}
\item[(A)] $J$-{\it holomorphic curves}, that is, 2-submanifolds $\Si\subset Y$ calibrated w.r.t.~$\om$.
\item[(B)] {\it Special Lagrangian\/ $3$-folds}, or {\it SL\/ $3$-folds}, with {\it phase\/} $e^{i\th}$, for $\th\in\R$, that is, 3-submanifolds $L\subset Y$ calibrated w.r.t.~$\cos\th\,\Re\Om+\sin\th\,\Im\Om$. \\
In particular, SL 3-folds with phase 1 are calibrated w.r.t.\ $\Re\Om$, and SL 3-folds with phase $i$ are calibrated w.r.t.\ $\Im\Om$. \\
When we do not specify a phase, we mean phase 1.
\item[(C)] {\it Complex surfaces}, that is, 4-submanifolds $S\subset Y$ calibrated w.r.t.~$\ha\om\w\om$.
\end{itemize}

\begin{rem} From \eq{ca2eq4}, we deduce the following relation between calibrated submanifolds in a Calabi--Yau 3-fold $Y$ (or in $Y=\C^3$), and calibrated submanifolds in the $G_2$-manifold $\R\t Y$ (or in $\R^7=\R\t\C^3$):
\begin{itemize}
\setlength{\itemsep}{0pt}
\setlength{\parsep}{0pt}
\item[(i)] If $\Si$ is a $J$-holomorphic curve in $Y$ then $\R\t\Si$ is associative 3-fold in~$\R\t Y$.
\item[(ii)] If $L$ is an SL 3-fold in $Y$ with phase 1 then $\{x\}\t L$ is an associative 3-fold in $\R\t Y$ for each~$x\in\R$.
\item[(iii)] If $L$ is an SL 3-fold in $Y$ with phase $i$ then $\R\t L$ is a coassociative 4-fold in~$\R\t Y$.
\item[(iv)] If $S$ is a complex surface in $Y$ then $\{x\}\t S$ is a coassociative 4-fold in $\R\t Y$ for each~$x\in\R$.
\end{itemize}
\label{ca2rem1}
\end{rem}

This will be important to us because a great deal is known about examples and properties of singularities of SL 3-folds, as in \cite{Joyc3,Joyc4,Joyc5,Joyc6,Joyc7,Joyc8,Joyc9,Joyc10,Joyc11,Joyc12,Joyc13,Joyc14,Joyc15,Joyc16,Joyc17,Joyc18,Joyc19,Joyc20}, and from Remark \ref{ca2rem1}(ii) we can deduce many examples of singularities of associative 3-folds. Examples of singular associative 3-folds in $\R^7$ which do not come from special Lagrangians in $\C^3$ can be found in Lotay~\cite{Lota1,Lota2,Lota3}.

\subsection{\texorpdfstring{$G_2$-instantons}{G₂-instantons}}
\label{ca24}

Let $(X,\vp,*\vp)$ be a compact, torsion-free $G_2$-manifold. As in \cite[\S 10.3]{Joyc3}, the 2-forms $\La^2T^*X$ on $X$ have a natural splitting $\La^2T^*X=\La^2_7\op\La^2_{14}$, where $\La^2_7,\La^2_{14}$ are vector subbundles of $\La^2T^*X$ with ranks 7,14, and $\La^2_{14}$ is the kernel of the vector bundle morphism $\La^2T^*X\ra \La^6T^*X$ mapping $\al\mapsto \al\w(*\vp)$. Let $G$ be a compact Lie group (we usually take $G=\SU(2)$), $\pi:P\ra X$ a principal $G$-bundle, and $A$ a connection on $P$, with curvature $F_A$. Following Donaldson and Segal \cite{DoSe}, we call $(P,A)$ a $G_2$-{\it instanton}, with {\it structure group\/} $G$, if the component of $F_A$ in $\mathop{\rm ad}(P)\ot\La^2_7$ is zero, or equivalently, if~$F_A\w(*\vp)=0$.

Write $\cM(P,*\vp)$ for the moduli space of gauge equivalence classes $[A]$ of $G_2$-instanton connections $A$ on $P$. The deformation theory of $A$, and hence the local description of $\cM(P,*\vp)$ near $[A]$, is controlled by the elliptic complex
\e
\begin{gathered}
\xymatrix@C=40pt@R=15pt{ 0 \ar[r] & \Ga^\iy(\ad{P}) \ar[rr]_(0.4){\d_A} && *+[l]{\Ga^\iy(\ad{P}\ot T^*X)} \ar[d]_{(-\w *\vp)\ci\d_A}
\\
0 & \Ga^\iy(\ad{P}\ot \La^7T^*X) \ar[l]
&& *+[l]{\Ga^\iy(\ad{P}\ot \La^6T^*X).\!\!} \ar[ll]_(0.6){\d_A}}
\end{gathered}
\label{ca2eq5}
\e
Here infinitesimal gauge transformations live in $\Ga^\iy(\ad{P})$, infinitesimal changes $\de A$ to $A$ live in $\Ga^\iy(\ad{P}\ot T^*X)$, and $F_{A+\de A}\w*\vp$ lives in $\Ga^\iy(\ad{P}\ot\La^6T^*X)$. For any connection $A'$ on $P$, as $\d_{A'}F_{A'}=0$ and $\d(*\vp)=0$ we have $\d_{A'}(F_{A'}\w*\vp)=0$, and the linearization of this equation at $A,\de A$ lies in~$\Ga^\iy(\ad{P}\ot\La^7T^*X)$.

Because the deformation theory of $G_2$-instantons comes from an elliptic complex \eq{ca2eq5}, which has index 0, the moduli spaces $\cM(P,*\vp)$ are well behaved, in the same way that moduli spaces of associative 3-folds in \S\ref{ca26} are well behaved: except at points $[A]$ with nontrivial stabilizer groups, $\cM(P,*\vp)$ should be a {\it derived manifold\/} of virtual dimension 0 in the sense of \cite{Bori,BoNo,Joyc22,Joyc23,Joyc24,Joyc25,Joyc26,Joyc27,Spiv}, and if $*\vp$ is suitably generic then $\cM(P,*\vp)$ should be a manifold of dimension 0. 

There is a topological formula for the $L^2$-norm $\nm{F_A}_{L^2}$ of the curvature of a $G_2$-instanton. When $G=\SU(2)$ this is
\e
\nm{F_A}^2_{L^2}=-4\pi^2([\vp]\cup c_2(P))\cdot[X],
\label{ca2eq6}
\e
where $c_2(P)$ is the second Chern class of $P$. We will discuss $G_2$-instantons and the Donaldson--Segal programme \cite{DoSe} further in~\S\ref{ca8}.

\subsection{\texorpdfstring{Tamed almost-$G_2$-manifolds}{Tamed almost-G₂-manifolds}}
\label{ca25}

So far we have focused on {\it torsion-free\/} $G_2$-manifolds $(X,\vp,*\vp)$, with $\d\vp=\d(*\vp)=0$. But for our purposes, these are too restrictive, for two reasons:
\begin{itemize}
\setlength{\itemsep}{0pt}
\setlength{\parsep}{0pt}
\item We want to discuss structures invariant under deformations of $\vp,*\vp$. On a compact 7-manifold $X$, torsion-free $G_2$-structures $(\vp,*\vp)$ come in finite-dimensional families as in Theorem \ref{ca2thm1}, so deformation-invariance amongst torsion-free $G_2$-structures is not a powerful statement. 

Even worse, we will want to fix the cohomology class $[\vp]\in H^3_{\rm dR}(X;\R)$, and then there are no torsion-free deformations at all. 
\item We hope that choosing $(\vp,*\vp)$ {\it generic\/} will simplify the problem (e.g. ensure all associative 3-folds $N\subset X$ are unobstructed). But this is only plausible if we choose $(\vp,*\vp)$ from an infinite-dimensional family.
\end{itemize}

The obvious answer is to relax the condition $\d\vp=0$ or $\d(*\vp)=0$ on $(X,\vp,*\vp)$, but there would be a cost to this, as the next remark explains.

\begin{rem} Here are the important consequences of allowing $\d\vp\ne 0$ or $\d(*\vp)\ne 0$ for the theories of associative 3-folds and coassociative 4-folds discussed in \S\ref{ca23}, and for $G_2$-instantons in~\S\ref{ca24}:
\begin{itemize}
\setlength{\itemsep}{0pt}
\setlength{\parsep}{0pt}
\item[(a)] If $\d\vp=0$ then a compact associative 3-fold $N\subset X$ has volume given by the topological formula, for $[\vp]\in H^3_{\rm dR}(X;\R)$ and $[N]\in H_3(X;\Z)$:
\e
\vol(N)=[\vp]\cdot[N].
\label{ca2eq7}
\e
If $\d\vp\ne 0$ then $[\vp]$ no longer makes sense.

This matters to us for two reasons. Firstly, if $\d\vp\ne 0$ then in a moduli space $\cM(\cN,\al,*\vp)$ of associative 3-folds $N$ in class $\al\in H_3(X;\Z)$, we might have a sequence $[N_i]_{i=1}^\iy$ in $\cM(\cN,\al,*\vp)$ with $\vol(N_i)\ra\iy$ as $i\ra\iy$, and then there could be no limit point $\lim_{i\ra\iy}[N_i]$ in $\cM(\cN,\al,*\vp)$. Thus, the lack of a volume bound may cause moduli spaces $\cM(\cN,\al,*\vp)$ to be noncompact (though they could also be noncompact for other reasons). Secondly, as in \eq{ca1eq3} we hope to combine invariants $GW_{\psi,\al}$ counting associatives $N$ in a formal power series weighted by $q^{\vol(N)}=q^{\ga\cdot\al}$, and this is only sensible with a topological formula for~$\vol(N)$.
\item[(b)] McLean's moduli theory for compact associative 3-folds $N$ in \S\ref{ca26} works fine if $\d\vp\ne 0\ne\d(*\vp)$. However, the linear elliptic operator $\bD:\Ga^\iy(\nu)\ra\Ga^\iy(\nu)$ need only be self-adjoint if $\d(*\vp)=0$. As in Remark \ref{ca3rem2} below, we need $\bD$ to be self-adjoint for the `canonical flag' of $N$ defined in \S\ref{ca3} to be well behaved, and this is important for our proposal in Conjecture~\ref{ca1conj1}.
\item[(c)] As in (a), if $\d(*\vp)=0$ then as in \eq{ca2eq7} a compact coassociative 4-fold $C\subset X$ has volume given by the topological formula
\e
\vol(C)=[*\vp]\cdot[C].
\label{ca2eq8}
\e 
If $\d(*\vp)\ne 0$ then $[*\vp]$ no longer makes sense, and the lack of a  volume bound could cause moduli spaces of coassociatives to become noncompact.
\item[(d)] McLean's moduli theory for compact coassociative 4-folds $C$ in \S\ref{ca26} relies on the alternative definition that $C$ is coassociative if $\vp\vert_C=0$. If $\d\vp\ne 0$ then the deformation theory of $C$ is no longer part of an elliptic complex, so coassociatives will not form well behaved moduli spaces.
\item[(e)] If $\d\vp=0$ then as in (a),(c) a $G_2$-instanton $(P,A)$ has a topological formula \eq{ca2eq6} for the $L^2$-norm of its curvature. This may be important in proving compactness of moduli spaces $\cM(P,*\vp)$.
\item[(f)] Moduli theory for $G_2$-instantons $A$ uses $F_A\w(*\vp)=0$. If $\d(*\vp)\ne 0$ then the deformation theory of $A$ is no longer part of an elliptic complex \eq{ca2eq5}, so as in (d), $G_2$-instantons will not form well behaved moduli spaces.
\end{itemize}
\label{ca2rem2}
\end{rem}

Therefore we do not want to sacrifice either condition $\d\vp=0$ or $\d(*\vp)=0$. Instead we will do something more complicated: we will work with a version of the `tamed almost-$G_2$-manifolds' introduced by Donaldson and Segal \cite[\S 3--\S 4]{DoSe}, for the same reasons as us. Our treatment using (i)--(iii) is new. 

\begin{dfn} A {\it tamed almost-$G_2$-manifold\/} or {\it TA-$G_2$-manifold\/} $(X,\vp,\psi)$ is a 7-manifold $X$ equipped with a closed positive 3-form $\vp$ and a closed positive 4-form $\psi$ satisfying a compatibility condition. As in \S\ref{ca22}, $\vp$ corresponds to a $G_2$-structure $P_\vp$ on $X$, and this induces an orientation on $X$. Using this orientation, $\psi$ corresponds to a $G_2$-structure $P_\psi$ on $X$. Write $g_\vp,g_\psi$ for the metrics induced by $P_\vp,P_\psi$. We require that the following equivalent conditions should hold:
\begin{itemize}
\setlength{\itemsep}{0pt}
\setlength{\parsep}{0pt}
\item[(i)] For all $x\in X$ and all oriented 3-planes $V\subset T_xX$ which are associative with respect to the $G_2$-structure $P_\psi$, we have $\vp\vert_V>0$.
\item[(ii)] For all $x\in X$ and all oriented 4-planes $W\subset T_xX$ which are coassociative with respect to the $G_2$-structure $P_\vp$, we have $\psi\vert_W>0$.
\item[(iii)] There do not exist $x\in X$, an oriented 3-plane $V\subset T_xX$ which is associative with respect to $P_\psi$, and an oriented 4-plane $W\subset T_xX$ which is coassociative with respect to $P_\vp$, such that $V\subset W\subset T_xX$.
\end{itemize}

To show that (i)--(iii) are equivalent, suppose (iii) does not hold, so there exist $V\subset W\subset T_xX$ as in (iii). Then $\vp\vert_W=0$ as $W$ is $\vp$-coassociative, so $\vp\vert_V=0$, and (i) does not hold. Also $V\subset W$ with $V$ $\psi$-associative and $W$ a 4-plane imply that $\psi\vert_W=0$, so (ii) does not hold. Hence (i),(ii) both imply~(iii). 

Suppose (i) does not hold. Then by connectedness, either (a) $\vp\vert_V<0$ for all $x\in X$ and $\psi$-associative $V\subset T_xX$, or (b) there exist $x\in X$ and $\psi$-associative $V\subset T_xX$ with $\vp\vert_V=0$. But for (a), by reversing the orientation used to define $P_\psi$ we would get $\vp\vert_V>0$ for all $x,V$, so that (i) holds after all. In fact (a) is impossible, as we chose $P_\vp,P_\psi$ to have the same orientation. Thus there exists a $\psi$-associative $V\subset T_xX$ with $\vp\vert_V=0$. By \cite[Th.~IV.4.6]{HaLa} there is a unique $\vp$-coassociative $W\subset T_xX$ with $V\subset W$, so (iii) does not hold. Thus (iii) implies (i). A similar argument shows that (iii) implies (ii), so (i)--(iii) are equivalent.

Observe that if $X$ is compact, then (i),(ii) are equivalent to:
\begin{itemize}
\setlength{\itemsep}{0pt}
\setlength{\parsep}{0pt}
\item[(i$)'$] There exists a constant $K>0$ such that for all $x\in X$ and all oriented 3-planes $V\subset T_xX$ which are associative with respect to $P_\psi$, we have $\vol^{g_\psi}_V\le K\vp_x\vert_V$, where $\vol^{g_\psi}_V\in\La^3V^*$ is the volume form defined using the metric $g_\psi\vert_x$ on $T_xX$ and the orientation on~$V$.
\item[(ii$)'$] There exists a constant $K'>0$ such that for all $x\in X$ and all oriented 4-planes $W\subset T_xX$ which are coassociative with respect to $P_\vp$, we have $\vol^{g_\vp}_W\le K'\psi_x\vert_W$, where $\vol^{g_\vp}_W\in\La^4W^*$ is the volume form defined using the metric $g_\vp\vert_x$ on $T_xX$ and the orientation on $W$.
\end{itemize}

Note that we can have $P_\vp\ne P_\psi$, and $P_\vp=P_\psi$ if and only if $(X,\vp,\psi)$ is a torsion-free $G_2$-manifold $(X,\vp,*\vp)$.

For $(X,\vp,\psi)$ to be a TA-$G_2$-manifold is an open condition on pairs $(\vp,\psi)$ of a closed 3-form $\vp$ and a closed 4-form $\psi$ on $X$. Thus the family of TA-$G_2$-structures on $X$ is infinite-dimensional, if it is nonempty.
\label{ca2def2}
\end{dfn}

Following \cite{DoSe}, we extend the definitions of associative 3-folds, coassociative 4-folds and $G_2$-instantons to TA-$G_2$-manifolds:

\begin{dfn} Let $(X,\vp,\psi)$ be a TA-$G_2$-manifold. Then:
\begin{itemize}
\setlength{\itemsep}{0pt}
\setlength{\parsep}{0pt}
\item[(i)] An {\it associative\/ $3$-fold\/} $N\subset X$ is a 3-submanifold $N$ in $X$ which is associative with respect to the $G_2$-structure $P_\psi$. 
\item[(ii)] A {\it coassociative\/ $3$-fold\/} $C\subset X$ is a 4-submanifold $C$ in $X$ which is associative with respect to the $G_2$-structure $P_\vp$. 
\item[(iii)] A $G_2$-{\it instanton\/} $(P,A)$ on $X$, with {\it structure group\/} $G$ for $G$ a compact Lie group, is a principal $G$-bundle $\pi:P\ra X$ and a connection $A$ on $P$ whose curvature $F_A$ satisfies $F_A\w\psi=0$. 
\end{itemize}

All the issues in Remark \ref{ca2rem2}(a)--(f) work out nicely with these definitions. For (a), if $(X,\vp,\psi)$ is a compact TA-$G_2$-manifold, so that Definition \ref{ca2def2}(i$)'$ holds for some $K>0$, and $N\subset X$ is a compact associative 3-fold, then for each $x\in N$ we have $\vol^{g_\psi}_{T_xN}\le A\vp_x\vert_V$, so integrating over $N$ yields a topological volume bound generalizing~\eq{ca2eq7}:
\e
\vol^{g_\psi}(N)\le K[\vp]\cdot[N].
\label{ca2eq9}
\e
For (b), as $\d\psi=0$ the elliptic operator $\bD$ in \S\ref{ca26} is self-adjoint. For (c), if $X$ is compact and $C\subset X$ is compact coassociative then as for \eq{ca2eq9} we 
get a topological volume bound generalizing \eq{ca2eq8}, for $K'>0$ as in Definition \ref{ca2def2}(ii$)'$:
\begin{equation*}
\vol^{g_\vp}(C)\le K'[\psi]\cdot[C].
\end{equation*}
For (d), as $\d\vp=0$, moduli spaces of coassociatives are well behaved. For (e), if $(P,A)$ is a $G_2$-instanton with group $G=\SU(2)$, as in \cite{DoSe} we can show that
\begin{equation*}
\nm{F_A}^2_{L^2}\le -K''([\vp]\cup c_2(P))\cdot[X],
\end{equation*}
generalizing \eq{ca2eq6}, for $K''>0$ depending on $(X,\vp,\psi)$ similar to $K$ in Definition \ref{ca2def2}(i$)'$. For (f), as $\d\psi=0$, moduli spaces of $G_2$-instantons are well behaved.
\label{ca2def3}
\end{dfn}

\begin{prop}{\bf(a)} Let\/ $X$ be a compact oriented\/ $7$-manifold and\/ $\psi$ a closed positive\/ $4$-form on $X$. Then $\cC_{X,\psi}:=\bigl\{\vp\in\Ga^\iy(\La^3T^*X):(X,\vp,\psi)$ is a TA-$G_2$-manifold\/$\bigr\}$ is an open convex cone in\/~$\bigl\{\vp\in\Ga^\iy(\La^3T^*X):\d\vp=0\bigr\}$. 

Hence $\cK_{X,\psi}:=\bigl\{[\vp]:\vp\in\cC_{X,\psi}\bigr\}$ is an open convex cone in~$H^3_{\rm dR}(X;\R)$.
\smallskip

\noindent{\bf(b)} Let\/ $X$ be a compact\/ $7$-manifold and\/ $\vp$ a closed positive\/ $3$-form on $X$. Then $\cC'_{X,\vp}:=\bigl\{\psi\in\Ga^\iy(\La^4T^*X):(X,\vp,\psi)$ is a TA-$G_2$-manifold\/$\bigr\}$ is an open convex cone in\/ $\bigl\{\psi\in\Ga^\iy(\La^4T^*X):\d\psi=0\bigr\}$. 

Hence $\cK'_{X,\vp}:=\bigl\{[\psi]:\psi\in\cC'_{X,\vp}\bigr\}$ is an open convex cone in $H^4_{\rm dR}(X;\R)$.
\label{ca2prop2}
\end{prop}

\begin{proof} Suppose $\vp_1,\vp_2\in\cC_{X,\psi}$, and let $t_1,t_2\ge 0$ with $(t_1,t_2)\ne(0,0)$. Consider the 3-form $\vp=t_1\vp_1+t_2\vp_2$ on $X$. It is closed as $\vp_1,\vp_2$ are, and it satisfies Definition \ref{ca2def2}(i) as $\vp_1,\vp_2$ do, and from this we can deduce that $\vp$ is positive. Therefore $(X,\vp,\psi)$ is also a TA-$G_2$-manifold, so $\vp\in\cC_{X,\psi}$, and $\cC_{X,\psi}$ is a convex cone in $\bigl\{\vp\in\Ga^\iy(\La^3T^*X):\d\vp=0\bigr\}$. Openness holds as Definition \ref{ca2def2}(i) is an open condition on $\vp$, proving (a). Part (b) is similar.
\end{proof}

\begin{dfn} Let $X$ be a 7-manifold. A closed positive 3-form $\vp$ on $X$ will be called {\it good\/} if there exists a 4-form $\psi$ on $X$ with $(X,\vp,\psi)$ a TA-$G_2$-manifold. 

Similarly, a closed positive 4-form $\psi$ on $X$ will be called {\it good\/} if there exists a 3-form $\vp$ on $X$ with $(X,\vp,\psi)$ a TA-$G_2$-manifold. For compact $X$, to be good is an open condition on closed 3- and 4-forms~$\vp,\psi$.
\label{ca2def4}
\end{dfn}

\begin{rem} We can now extend our analogy between Calabi--Yau 3-folds $(Y,J,h)$ and $G_2$-manifolds $(X,\vp,*\vp)$, adding the lines:
\begin{align*}
&\text{Symplectic form $\om$ on $Y$} &&\!\!\!\!\leftrightarrow\!\!\! &&\text{Good 3-form $\vp$ on $X$} \\
&\text{(Almost) complex structure $J$ on $Y$} &&\!\!\!\!\leftrightarrow\!\!\! &&\text{Good 4-form $\psi$ on $X$} \\
&\begin{subarray}{l}\ts\text{Symplectic manifold $(Y,\om)$ with} \\
\ts\text{compatible almost complex structure $J$}\end{subarray} &&\!\!\!\!\leftrightarrow\!\!\! &&\text{TA-$G_2$-manifold $(X,\vp,\psi)$.}\!\!\!\!{} 
\end{align*}

Then Proposition \ref{ca2prop2}(a) is an analogue of the fact that K\"ahler forms $\om$ on a fixed complex manifold $(Y,J)$ form an open convex cone in the closed real (1,1)-forms on $Y$, and $\cK_{X,\psi}$ is an analogue of the K\"ahler cone of $(Y,J)$. Also Proposition \ref{ca2prop2}(b) is analogous to the fact that the family of almost complex structures $J$ compatible with a fixed symplectic form $\om$ on $Y$ form an infinite-dimensional contractible space.

Suppose we can show that some structure we define for TA-$G_2$-manifolds $(X,\vp,\psi)$, e.g. $G_2$ quantum cohomology in \S\ref{ca76}, is unchanged under deformations of $(X,\vp,\psi)$ fixing $\vp$. If so, this structure depends only on $X$ and the good 3-form $\vp$, as Proposition \ref{ca2prop2}(b) shows that the family of $\psi$ compatible with $\vp$ is connected. This is the analogue of the Gromov--Witten invariants, Lagrangian Floer cohomology, etc.\ of a symplectic manifold $(Y,\om)$ being independent of almost complex structure $J$. 

In fact our theories will manifestly depend only on $\psi$ and the cohomology class $[\vp]\in H^3(X;\R)$, so if they are independent of $\psi$ up to deformation, then they depend only on $(X,\vp)$ up to deformations fixing~$[\vp]$.
\label{ca2rem3}	
\end{rem}

\subsection{Moduli spaces of associative 3-folds}
\label{ca26}

Much of this paper concerns moduli spaces of associative 3-folds $\cM(\cN,\al,\psi)$ in a TA-$G_2$-manifold $(X,\vp,\psi)$. We will use the following notation.

\begin{dfn} Consider compact, oriented 3-manifolds $N$. Write $[N]_\cD$ or $\cN$ for the equivalence class of $N$ under the equivalence relation $N\sim N'$ if there exists an orientation-preserving diffeomorphism $\de:N\ra N'$. We call $[N]_\cD$ an {\it oriented diffeomorphism class}. Write $\cD$ for the set of all oriented diffeomorphism classes, and $\cDHS\subset\cD$ for the subset of $[N]_\cD$ with $N$ a $\Q$-{\it homology sphere}, that is, $b^1(N)=b^2(N)=0$, which is equivalent to $H_1(N;\Z)$ being finite.

Let $(X,\vp,\psi)$ be a TA-$G_2$-manifold. For each $\cN\in\cD$ and $\al\in H_3(X;\Z)$, we write $\cM(\cN,\al,\psi)$ for the {\it moduli space of immersed associative\/ $3$-folds\/} $i:N\ra X$ in $(X,\vp,\psi)$ which have oriented diffeomorphism type $\cN$ and homology class $\al$. In more detail, consider pairs $(N,i)$, where:
\begin{itemize}
\setlength{\itemsep}{0pt}
\setlength{\parsep}{0pt}
\item $N$ is a compact, oriented 3-manifold in oriented diffeomorphism class $\cN$;
\item $i:N\ra X$ is an immersed associative 3-fold in $(X,\vp,\psi)$;
\item $i^*(\vp)$ is a positive 3-form on $N$ with its given orientation; and 
\item $i_*([N])=\al\in H_3(X;\Z)$.
\end{itemize} 
Two such pairs $(N,i),(N',i')$ are equivalent, written $(N,i)\approx(N',i')$, if there exists an orientation-preserving diffeomorphism $\de:N\ra N'$ with $i=i'\ci\de$. We write $[N,i]$ for the $\approx$-equivalence class of~$(N,i)$.

Then just as a set, $\cM(\cN,\al,\psi)$ is the set of all such $[N,i]$. We make $\cM(\cN,\al,\psi)$ into a topological space by choosing $N\in\cN$, and writing
\begin{align*}
\cM(\cN,\al,\psi)\cong\bigl\{i\in \Map_{C^\iy}(N,X):\text{$i$ is an associative immersion,}\\ 
\text{$i^*(\vp)$ is positive, $i_*([N])=\al\in H_3(X;\Z)$}\bigr\}\big/\Diff_+(N),
\end{align*}
with $\Diff_+(N)$ the group of orientation-preserving diffeomorphisms $\de:N\ra N$ acting by $i\mapsto i\ci\de$. Then we give $\cM(\cN,\al,\psi)$ the quotient-subspace topology coming from the $C^\iy$-topology on $\Map_{C^\iy}(N,X)$. We write $\cM(\cN,\al,\psi)_\emb\subseteq\cM(\cN,\al,\psi)$ for the open subset of $[N,i]$ with $i:N\hookra X$ an embedding.

For each $[N,i]\in\cM(\cN,\al,\psi)$ we define the {\it isotropy group\/} $\Iso([N,i])$ to be the subgroup $\de\in \Diff_+(N)$ with $i\ci\de=i$. Then $\Iso([N,i])$ is finite, as $N$ is compact and $i$ an immersion, and $\Iso([N,i])=\{1\}$ if $[N,i]\in \cM(\cN,\al,\psi)_\emb$.

We use the notation $\cM(\cN,\al,\psi)$, omitting $\vp$, since as in Definition \ref{ca2def3} the notion of associative 3-fold in $(X,\vp,\psi)$ depends only on $X,\psi$, not on~$\vp$.

Now suppose $(X,\vp_t,\psi_t):t\in\cF$ is a smooth family of TA-$G_2$-manifolds over a base $\cF$ which is a finite-dimensional manifold, or manifold with boundary. Then we write $\cM(\cN,\al,\psi_t:t\in\cF)$ for the moduli space of pairs
\begin{equation*}
\cM(\cN,\al,\psi_t:t\in\cF)=\bigl\{(t,[N,i]):t\in\cF,\;\> [N,i]\in \cM(\cN,\al,\psi_t)\bigr\},
\end{equation*}
with topology induced from that on $\cF\t\Map_{C^\iy}(N,X)$ as above.
\label{ca2def5}
\end{dfn}

We want the moduli spaces $\cM(\cN,\al,\psi)$, $\cM(\cN,\al,\psi_t:t\in\cF)$ to be not just topological spaces, but (in good cases) manifolds or orbifolds, preferably compact and oriented, and (in general) derived manifold or derived orbifolds. The deformation theory of compact associative 3-folds was studied by McLean \cite[\S 5]{McLe}. He considered compact, embedded associative 3-folds in torsion-free $G_2$-manifolds, and showed that their moduli space is locally the solutions of a nonlinear elliptic p.d.e.\ with linearization the twisted Dirac operator $\bD$ below. Our theorem follows from and extends McLean's work using standard techniques.

\begin{thm}[McLean {\cite[\S 5]{McLe},} extended] Suppose\/ $(X,\vp,\psi)$ is a TA-$G_2$-manifold, and\/ $i:N\ra X$ be a compact, immersed associative $3$-fold, with $i_*([N])=\al\in H_3(X;\Z)$ and\/ $[N]=\cN\in\cD,$ so that\/~$[N,i]\in \cM(\cN,\al,\psi)$.

Write $g$ for the Riemannian metric on $X$ from the $G_2$-structure associated to $\psi,$ and\/ $\nu\ra N$ for the normal bundle of\/ $N$ in $X,$ a rank\/ $4$ vector bundle, and\/ $\nabla^\nu$ for the connection on $\nu$ induced by the Levi-Civita connection of\/ $g$. Then there is a natural first-order linear elliptic operator $\bD:\Ga^\iy(\nu)\ra\Ga^\iy(\nu)$ of index $0,$ a twisted Dirac operator, which is characterized by the equation  
\e
\bigl\langle\bD v,w\bigr\rangle_{L^2}=\int_N\psi_{a_1a_2[b_1b_2}(\nabla^\nu_{b_3]}v^{a_1})w^{a_2}
\label{ca2eq10}
\e
for all\/ $v,w\in\Ga^\iy(\nu)$. Here the $L^2$-inner product on $\Ga^\iy(\nu)$ is defined using $g,$ and we use the index notation for tensors, contracting together $\psi,v,\nabla^\nu w$ to get a $3$-form, which we integrate over the oriented\/ $3$-manifold\/~$N$.

Write $\cT_N=\Ker\bD$ and\/ $\O_N=\Coker\bD,$ as finite-dimensional real vector spaces with\/ $\dim\cT_N=\dim\O_N$. Then the finite group\/ $\Ga:=\Iso([N,i])$ from Definition\/ {\rm\ref{ca2def5}} acts on $\cT_N,\O_N$. There exist a $\Ga$-invariant open neighbourhood\/ $V$ of\/ $0$ in $\cT_N,$ a $\Ga$-equivariant smooth map $\Th:V\ra\O_N$ with\/ $\Th(0)=\d\Th(0)=0,$ an open neighbourhood\/ $W$ of\/ $[N,i]$ in $\cM(\cN,\al,\psi),$ and a homeomorphism $\Psi:\Th^{-1}(0)/\Ga\ra W$ with\/~$\Psi(0)=[N,i]$. 

We call\/ $\cT_N$ the \begin{bfseries}Zariski tangent space\end{bfseries} and\/ $\O_N$ the \begin{bfseries}obstruction space\end{bfseries} to $\cM(\cN,\al,\psi)$ at $[N,i]$. We call\/ $N$ \begin{bfseries}unobstructed\end{bfseries} if\/~$\O_N=0$.
\label{ca2thm2}
\end{thm}

The proof of Theorem \ref{ca2thm2} does not need $\psi$ closed, and does not use $\vp$ at all. However, if $v,w\in\Ga^\iy(\nu)$ then by Stokes' Theorem and \eq{ca2eq10} we have
\begin{align*}
0&=\int_N\d[\psi_{a_1a_2b_1b_2}v^{a_1}w^{a_2}\bigr]\\
&=\int_N\bigl[\d\psi_{a_1a_2b_1b_2b_3}v^{a_1}w^{a_2}
+\psi_{a_1a_2[b_1b_2}\nabla_{b_3]}^\nu v^{a_1}w^{a_2}+\psi_{a_1a_2[b_1b_2}v^{a_1}\nabla^\nu_{b_3]}w^{a_2}\bigr]\\
&=\int_N\bigl[\d\psi_{a_1a_2b_1b_2b_3}v^{a_1}w^{a_2}\bigr]+\bigl\langle\bD v,w\bigr\rangle_{L^2}-\bigl\langle v,\bD w\bigr\rangle_{L^2}.
\end{align*}
Hence if $\d\psi=0$ we have $\langle\bD v,w\rangle_{L^2}=\langle v,\bD w\rangle_{L^2}$, giving:

\begin{lem} In Theorem\/ {\rm\ref{ca2thm2},} if\/ $\d\psi=0$ (which is included in the definition of TA-$G_2$-manifold\/ $(X,\vp,\psi)$) then $\bD$ is a \begin{bfseries}self-adjoint\end{bfseries} linear operator.
\label{ca2lem1}	
\end{lem}

In \S\ref{ca3} we want $\bD$ to be self-adjoint to define `flags' of unobstructed associative 3-folds, and this is one reason we take $\psi$ closed in TA-$G_2$-manifolds~$(X,\vp,\psi)$.

Derived Differential Geometry is the study of `derived manifolds' and `derived orbifolds'. Different versions of derived manifolds are defined by Spivak \cite{Spiv}, Borisov--Noel \cite{Bori,BoNo} and the author \cite{Joyc22,Joyc23,Joyc24,Joyc25,Joyc26,Joyc27}. The author gives two equivalent notions of derived manifolds and orbifolds: {\it d-manifolds\/} and {\it d-orbifolds\/} \cite{Joyc22,Joyc23,Joyc24}, and {\it m-Kuranishi spaces\/} and {\it Kuranishi spaces\/} \cite{Joyc25,Joyc26,Joyc27}, which are an improved version of Fukaya--Oh--Ohta--Ono's Kuranishi spaces~\cite{FOOO,FuOn}.

Many moduli spaces in differential geometry are known to be derived manifolds or derived orbifolds \cite{Joyc24}. Theorem \ref{ca2thm2} implies that $\cM(\cN,\al,\psi)$ locally has the structure of a derived orbifold/Kuranishi space, since $(V,\O_N,\Ga,\Th,\Psi)$ is a Kuranishi neighbourhood on $\cM(\cN,\al,\psi)$. The author expects to prove the following conjecture in the next few years, as part of a larger project.

\begin{conj} In Definition\/ {\rm\ref{ca2def5}} we can give $\cM(\cN,\al,\psi)$  the structure of a \begin{bfseries}d-orbifold\end{bfseries} in the sense of\/ {\rm\cite{Joyc22,Joyc23,Joyc24},} or a \begin{bfseries}Kuranishi space\end{bfseries} in the sense of\/ {\rm\cite{Joyc25,Joyc26,Joyc27},} of virtual dimension\/ $0,$ canonical up to equivalence in the $2$-categories\/ $\dOrb,\Kur$. The open subset $\cM(\cN,\al,\psi)_\emb\subseteq\cM(\cN,\al,\psi)$ of embedded associatives becomes a \begin{bfseries}d-manifold\end{bfseries} or \begin{bfseries}m-Kuranishi space\end{bfseries}. Similarly, we can make $\cM(\cN,\al,\psi_t:t\in\cF)$ into a d-orbifold or Kuranishi space, with virtual dimension $\dim\cF,$ and with a $1$-morphism $\pi:\cM(\cN,\al,\psi_t:t\in\cF)\ra\cF$.
\label{ca2conj1}	
\end{conj}

Here is a class of immersed submanifolds that will be important to us:

\begin{dfn} Let $i:N\ra X$ be a compact, immersed submanifold. We call $N$ {\it finite-embedded\/} if either $i:N\ra X$ is an embedding, or else $i=\ti\imath\ci\pi$ for $\ti\imath:\ti N\ra X$ an embedded submanifold and $\pi:N\ra\ti N$ a finite cover.
\label{ca2def6}	
\end{dfn}

In several important moduli problems, by taking the geometric data generic, one can ensure that the moduli spaces are smooth. For example, Donaldson and Kronheimer \cite[\S 4.3]{DoKr} show that if $(M,g)$ is a compact oriented Riemannian 4-manifold with $b^2_+(M)>0$ then all moduli spaces of $\SU(2)$-instantons on $X$ are smooth, and McDuff and Salamon \cite[\S 3.4]{McSa} prove that if $(S,\om)$ is a symplectic manifold and $J$ is a generic almost structure on $S$ compatible with $\om$ then all moduli spaces of embedded $J$-holomorphic curves in $S$ are smooth.

\begin{conj} Suppose\/ $(X,\vp,\psi)$ is a compact TA-$G_2$-manifold, with\/ $\psi$ \begin{bfseries}generic\end{bfseries} amongst closed\/ $4$-forms on $X$. Then for all\/ $\cN\in\cD$ and\/ $\al\in H_3(X;\Z),$ the moduli space $\cM(\cN,\al,\psi)$ in Definition\/ {\rm\ref{ca2def5}} is a finite set. 

For each\/ $[N,i]\in\cM(\cN,\al,\psi),$ the associative $3$-fold\/ $N$ is unobstructed, and $N$ is finite-embedded, as in Definition\/ {\rm\ref{ca2def6}}. Furthermore, for any $A>0$ there are only finitely many pairs $(\cN,\al)$ with\/ $\cM(\cN,\al,\psi)\ne\es$ and\/~$[\vp]\cdot\al\le A$.
\label{ca2conj2}	
\end{conj}

Note here that $\cM(\cN,\al,\psi)$ has virtual dimension 0, and `compact smooth 0-manifold' is equivalent to `finite set'.

McLean \cite[\S 3--\S 4]{McLe} also studied moduli spaces of compact special Lagrangian submanifolds, and coassociative 4-folds. These are simpler than the associative case, as they are always smooth manifolds. 

\begin{thm}[McLean \cite{McLe}]{\bf(a)} Suppose $(Y,J,h)$ is a Calabi--Yau $m$-fold, and\/ $L\subset Y$ is a compact SL $m$-fold. Then the moduli space $\cM_L$ of special Lagrangian deformations of\/ $L$ is a smooth manifold of dimension\/ $b^1(L)$.
\smallskip

\noindent{\bf(b)} Suppose $(X,\vp,\psi)$ is a TA-$G_2$-manifold, and\/ $C$ is a compact coassociative $4$-fold in\/ $X$. Then the moduli space $\cM_C$ of coassociative deformations of\/ $C$ is a smooth manifold of dimension\/ $b^2_+(C)$.
\label{ca2thm3}	
\end{thm}

The proof of part (b) requires $\d\vp=0$.

\subsection{Associative 3-folds with boundary in coassociatives}
\label{ca27}

If $(X,\vp,\psi)$ is a TA-$G_2$-manifold and $C\subset X$ is coassociative, we can consider associative 3-folds $N\subset X$ with boundary $\pd N\subset C$. Note that associatives $N$ are defined using $\psi$, and coassociatives $C$ defined using $\vp$, but Definition \ref{ca2def2}(iii) ensures that $\pd N\subset C$ is a well behaved boundary condition for $N$. If $(X,\vp,\psi)$ is a compact TA-$G_2$-manifold, so that Definition \ref{ca2def2}(i$)'$ holds for some $K>0$, then as in \eq{ca2eq9} we have the topological volume bound
\begin{equation*}
\vol^{g_\psi}(N)\le K[\vp]\cdot[N].
\end{equation*}
where now we use relative (co)homology $[\vp]\in H^3_{\rm dR}(X,C;\R)$, $[N]
\in H_3(X,C;\Z)$.

Gayet and Witt \cite{GaWi} generalized Theorem \ref{ca2thm2} to the boundary case. The dimension of the moduli space is no longer automatically zero.

\begin{thm}[Gayet and Witt \cite{GaWi}, extended] Let\/ $(X,\vp,\psi)$ be a TA-$G_2$-manifold, and\/ $C\subset X$ a coassociative $4$-fold. Suppose\/ $N$ is a compact, immersed associative $3$-fold in $X$ with connected boundary $\pd N\subset C$ of genus $g$. Then the deformation theory of\/ $N$ for fixed\/ $(X,\vp,\psi),C$ is a nonlinear elliptic equation, of index\/ $d(N):=\int_{\pd N}c_1(\nu_{\pd N})+1-g,$ where $\nu_N$ is the normal bundle of\/ $\pd N$ in $C$ with its natural complex structure.
\label{ca2thm4}
\end{thm}

Thus as in Conjecture \ref{ca2conj1} we expect the moduli space $\cM_N$ of deformations of\/ $N$ to be a derived orbifold as in \cite{Joyc22,Joyc23,Joyc24,Joyc25,Joyc26,Joyc27}, of virtual dimension~$d(N)$.

Given two nearby coassociatives $C_1,C_2$ in $(X,\vp,*\vp)$ with $C_1\cap C_2=\es$, Leung, Wang and Zhu \cite{LWZ1,LWZ2} prove results on associative 3-folds $N$ in $(X,\vp,*\vp)$ with boundary $\pd N\subset C_1\amalg C_2$ and $\vol(N)$ small. This is intended as a first step towards constructing some kind of Floer theory for coassociative 4-folds $C$ by counting associative 3-folds $N$ with boundary $\pd N\subset C$. We discuss this in~\S\ref{ca62}.

\section{How to orient moduli spaces of associatives}
\label{ca3}

The material of this section is new. Our aim is to construct orientations on the moduli spaces $\cM(\cN,\al,\psi)$ of associatives in $(X,\vp,\psi)$ in \S\ref{ca26}, considered as derived orbifolds in the case of Conjecture \ref{ca2conj1}, or as orbifolds in the case of Conjecture \ref{ca2conj2}. For unobstructed associatives, our construction is rigorous.

We will show that any compact associative 3-fold $N\subset X$  has a natural {\it flag\/} $f_N$, a partial framing of the normal bundle $\nu\ra N$, defined in a subtle way using the operator $\bD:\Ga^\iy(\nu)\ra\Ga^\iy(\nu)$ from Theorem \ref{ca2thm2}. The set $\Flag(N)$ of flags on $N$ is a $\Z$-torsor. Roughly speaking we have $\Flag(N)\cong\Z$, and when $N$ is unobstructed we define $N$ to be positively (negatively) oriented if $f_N\in\Flag(N)$ corresponds to an even (odd) number in~$\Z$.

In fact things are more complicated, as the isomorphism $\Flag(N)\cong\Z$ is not canonical. We will define a new algebro-topological structure on $X$ called a {\it flag structure\/} $\cF$. The set of flag structures is a torsor over $\Hom_{\rm Grp}\bigl(H_3(X;\Z),\{\pm 1\}\bigr)$. Given a flag structure on $X$, the isomorphism $\Flag(N)\cong\Z$ is canonical mod $2\Z$, which is enough to define orientations.

Orienting moduli spaces $\cM(\cN,\al,\psi)$ is important for our programme, since it is essential to count associative 3-folds with signs to have any chance of getting a deformation-invariant answer, as we explain in~\S\ref{ca7}.

For comparison, Donaldson and Kronheimer \cite[\S 5.4 \& \S 7.1.6]{DoKr} construct orientations on moduli spaces of instantons on a 4-manifold $M$, and Fukaya--Oh--Ohta--Ono \cite[\S 8]{FOOO} define orientations on moduli spaces of $J$-holomorphic discs in a symplectic manifold $S$ with boundary in a Lagrangian $L$. In both cases some extra algebraic topological data is needed, namely an orientation on $H^1(M;\R)\op H^2_+(M;\R)$ in \cite{DoKr}, and a relative spin structure for $(S,L)$ in~\cite{FOOO}.

\subsection{Flags and flag structures}
\label{ca31}

Though we explain the material of this section for 3-submanifolds $N$ in a 7-manifold $X$, in fact it works in exactly the same way for $(2k+1)$-dimensional submanifolds $N$ of a $(4k+3)$-manifold $X$ for $k=0,1,\ldots.$

\begin{dfn} Let $X$ be an oriented 7-manifold, and $i:N\ra X$ a compact, oriented, immersed 3-manifold in $X$. Write $\nu\ra N$ for the normal bundle of $N$ in $X$. Then the orientation on $X$ induces an orientation on the total space of $\nu$.

Consider nonvanishing sections $s\in\Ga^\iy(\nu)$, so that $s(x)\ne 0$ for all $x\in N$. Let $s,s'$ be nonvanishing sections. Write
$0:N\ra \nu$ for the zero section, and $\ga:[0,1]\t N\ra\nu$ for the map $\ga:(t,x)\mapsto (1-t)s(x)+ts'(x)$. Then $0(N)$ is a 3-cycle in the homology of $\nu$ over $\Z$, and $\ga([0,1]\t N)$ is a 4-chain in the homology of $\nu$, where $\pd[\ga([0,1]\t N)]$ is disjoint from $0(N)$, and $\nu$ is an oriented 7-manifold. Define $d(s,s')\in\Z$ to be the intersection number $0(N)\bu\ga([0,1]\t N)$.

We have $d(s',s)=-d(s,s')$ and $d(s,s'')=d(s,s')+d(s',s'')$ for all nonvanishing sections $s,s',s''\in\Ga^\iy(\nu)$. Define a {\it flag on\/} $N$ to be an equivalence class $[s]$ of nonvanishing $s\in\Ga^\iy(\nu)$, where $s,s'$ are equivalent if $d(s,s')=0$. We call $(N,[s])$ a {\it flagged submanifold}. Write $\Flag(N)$ for the set of all flags $[s]$ on~$N$.

For $[s],[s']\in\Flag(N)$ we define $d([s],[s'])=d(s,s')\in\Z$ for any representatives $s,s'$ for $[s],[s']$. It is not difficult to show that for any $[s]\in\Flag(N)$ and any $k\in\Z$, there is a unique $[s']\in\Flag(N)$ with $d([s],[s'])=k$. We write $[s']=[s]+k$. This gives a natural action of $\Z$ on $\Flag(N)$ by addition, which makes $\Flag(N)$ into a $\Z$-torsor (that is, the $\Z$-action is free and transitive).
\label{ca3def1}	
\end{dfn}

For the next parts we restrict to $(N,[s])$ with $N$ {\it finite-embedded}, as in Definition \ref{ca2def6}. We compare flags for homologous 3-submanifolds $N_1,N_2$.

\begin{dfn} Let $X$ be an oriented 7-manifold, and suppose $N_1,N_2$ are compact, oriented, finite-embedded 3-submanifolds in $X$ with $[N_1]=[N_2]$ in $H_3(X;\Z)$ and $N_1\cap N_2=\es$, and $[s_1],[s_2]$ are flags on $N_1,N_2$. Choose a 4-chain $C_{12}$ in the homology of $X$ over $\Z$ with $\pd C_{12}=N_2-N_1$. Let $s_1,s_2$ be representatives for $N_1,N_2$, and let $N_1',N_2'$ be small perturbations of $N_1,N_2$ in the normal directions $s_1,s_2$. Then $N_1'\cap N_1=N_2'\cap N_2=\es$ as $s_1,s_2$ are nonvanishing and $N_1,N_2$ are finite-embedded, and also $N_1'\cap N_2=N_2'\cap N_1=\es$ as $N_1,N_2$ are disjoint and $N_1',N_2'$ are close to~$N_1,N_2$.

Define $D((N_1,[s_1]),(N_2,[s_2]))$ to be the intersection number $(N_2'-N_1')\bu C_{12}$ in homology over $\Z$. This is well defined as $\pd C_{12}=N_2-N_1$, so the 3-cycles $N_2'-N_1'$ and $\pd C_{12}$ are disjoint. It is also independent of the choices of $C_{12}$ and $N_1',N_2'$. We can show that for $k_1,k_2\in\Z$ we have
\e
D((N_1,[s_1]+k_1),(N_2,[s_2]+k_2))=D((N_1,[s_1]),(N_2,[s_2]))-k_1+k_2.
\label{ca3eq1}
\e

\label{ca3def2}	
\end{dfn}

\begin{prop} Let\/ $X$ be an oriented\/ $7$-manifold, and\/ $(N_1,[s_1]),(N_2,[s_2]),\ab(N_3,[s_3])$ be disjoint finite-embedded flagged submanifolds in $X$. Then 
\e
\begin{split}
D((N_1,[s_1]),(N_3,[s_3]))&=D((N_1,[s_1]),(N_2,[s_2]))\\
&\quad +D((N_2,[s_2]),(N_3,[s_3]))\mod 2.
\end{split}
\label{ca3eq2}
\e

\label{ca3prop1}	
\end{prop}

\begin{proof} Let $s_1,s_2,s_3$ be representatives for $[s_1],[s_2],[s_3]$, and $N_1',N_2',N_3'$ be small perturbations of $N_1,N_2,N_3$ in directions $s_1,s_2,s_3$. Choose 4-chains $C_{12},C_{23}$ over $\Z$ in $X$ with $\pd C_{12}=N_2-N_1$ and $\pd C_{23}=N_3-N_2$. Then $C_{13}=C_{12}+C_{23}$ is a 4-chain with $\pd C_{13}=N_3-N_1$. Also choose a 4-chain $C_{12}'$ with $\pd C_{12}'=N_2'-N_1'$. Then we have
\begin{align*}
&D((N_1,[s_1]),(N_3,[s_3]))-D((N_1,[s_1]),(N_2,[s_2]))-
D((N_2,[s_2]),(N_3,[s_3]))	\\
&=(N'_3-N'_1)\bu(C_{12}+C_{23})-(N'_2-N'_1)\bu C_{12}-(N'_3-N'_2)\bu C_{23}\\
&=(N'_3-N'_2)\bu C_{12}+(N'_2-N'_1)\bu C_{23}=(N_3-N_2)\bu C'_{12}+(N'_2-N'_1)\bu C_{23}\\
&=\pd C_{23}\bu C'_{12}+\pd C_{12}'\bu C_{23}=\pd (C_{23}\bu C'_{12})+2\pd C_{12}'\bu C_{23}=0+2\pd C_{12}'\bu C_{23},
\end{align*}
using the definition of $D((N_i,[s_i]),(N_j,[s_j]))$ in the first step, the easy identity $(N'_3-N'_2)\bu C_{12}=(N_3-N_2)\bu C'_{12}$ in the third, and that a boundary is zero in homology in the sixth. Equation \eq{ca3eq2} follows.
\end{proof}

\begin{prop} Let\/ $X$ be an oriented\/ $7$-manifold, and\/ $(N,[s])$ be an immersed flagged submanifold in $X,$ and\/ $(N',[s']),(N'',[s''])$ be any two small perturbations of\/ $(N,[s])$ with $N',N''$ embedded in $X$. Then
\e
\begin{split}
D((N',[s']),(N'',[s'']))=0\mod 2.
\end{split}
\label{ca3eq3}
\e

\label{ca3prop2}	
\end{prop}

\begin{proof} Given $(N',[s']),(N'',[s''])$ as in the proposition, choose a generic smooth 1-parameter family $(\hat N_t:[\hat s_t])$ of small perturbations of $(N,[s])$ for $t\in[0,1]$ with $(\hat N_0,[\hat s_0])=(N',[s'])$ and $(\hat N_1,[\hat s_1])=(N'',[s''])$. Then by genericness we can suppose that there exist $0<t_1<t_2<\cdots<t_k<1$ such that $\hat N_t$ is embedded for $t\in[0,1]\sm\{t_1,\ldots,t_k\}$, and $\hat N_{t_i}$ is immersed with a single self-intersection point $x_i\in X$ for $i=1,\ldots,k$, such that the family $\hat N_t,$ $t\in[0,1]$ crosses itself transversely at $x_i$ as $t$ increases through~$t_i$.

Choose another compact embedded flagged submanifold $(\check N,[\check s])$ in $X$ with $[\check N]=[N]\in H_3(X;\Z)$ which is disjoint from $N$, and hence also disjoint from $N',N'',\hat N_t$ as these are small perturbations of $N$. Consider the function
\e
t\longmapsto D((\check N,[\check s]),(\hat N_t,[\hat s_t]))\quad\text{for $t\in [0,1]\sm\{t_1,\ldots,t_k\}$.}
\label{ca3eq4}
\e
Since $\check N$ is disjoint from $\hat N_t$, and $(\hat N_t,[\hat s_t])$ deforms continuously in $t$, this function is constant in each connected component of $[0,1]\sm\{t_1,\ldots,t_k\}$. Define $D((\check N,[\check s]),(\hat N_t,[\hat s_t]))$ using a 4-chain $C_t$ with $\pd C_t=\hat N_t-\check N$, where $C_t$ depends continuously on $t\in[0,1]$. For $t$ close to $t_i$, near $x_i$ in $X$ there are two sheets $\hat N_t^+,\hat N_t^-$ of $\hat N_t$, and hence two sheets $\pd C_t^+,\pd C_t^-$ of $\pd C_t$. As $t$ crosses $t_i$, we see that $\hat N_t^+$ crosses $\pd C_t^-$ transversely, and $\hat N_t^-$ crosses $\pd C_t^+$ transversely with the same orientation, so that $D((\check N,[\check s]),(\hat N_t,[\hat s_t]))$ changes by $\pm 2$. Therefore the total change in \eq{ca3eq4} between $t=0$ and $t=1$ is even, giving
\begin{equation*}
D((\check N,[\check s]),(N',[s']))=D((\check N,[\check s]),(N'',[s'']))\mod 2.
\end{equation*}
Equation \eq{ca3eq3} now follows from Proposition~\ref{ca3prop1}.
\end{proof}

Flag structures are the algebro-topological data we will need in \S\ref{ca32} to orient moduli spaces of associative 3-folds in~$(X,\vp,\psi)$.

\begin{dfn} Let $X$ be an oriented 7-manifold. A {\it flag structure\/} is a map
\e
F:\bigl\{\text{immersed flagged submanifolds $(N,[s])$ in $X$}\bigr\}\longra\{\pm 1\},
\label{ca3eq5}
\e
satisfying:
\begin{itemize}
\setlength{\itemsep}{0pt}
\setlength{\parsep}{0pt}
\item[(i)] If $(N,[s])$ is an immersed flagged submanifold and $(N',[s'])$ is any small perturbation of $(N,[s])$ then $F(N,[s])=F(N',[s'])$.
\item[(ii)] $F(N,[s]+k)=(-1)^k\cdot F(N,[s])$ for all $(N,[s])$ and~$k\in\Z$.
\item[(iii)] If $(N_1,[s_1]),(N_2,[s_2])$ are disjoint finite-embedded flagged submanifolds in $X$ with $[N_1]=[N_2]$ in $H_3(X;\Z)$ then
\e
F(N_2,[s_2])=F(N_1,[s_1])\cdot (-1)^{D((N_1,[s_1]),(N_2,[s_2]))}.
\label{ca3eq6}
\e
\item[(iv)] If $(N_1,[s_1]),(N_2,[s_2])$ are disjoint immersed flagged submanifolds then
\e
F(N_1\amalg N_2,[s_1\amalg s_2])=F(N_1,[s_1])\cdot F(N_2,[s_2]).
\label{ca3eq7}
\e
\end{itemize}
\label{ca3def3}	
\end{dfn}

\begin{prop} Let\/ $X$ be an oriented\/ $7$-manifold. Then:
\begin{itemize}
\setlength{\itemsep}{0pt}
\setlength{\parsep}{0pt}
\item[{\bf(a)}] There exists a flag structure $F$ on $X$.
\item[{\bf(b)}] If\/ $F,F'$ are flag structures on $X$ then there exists a unique group morphism $\ep:H_3(X;\Z)\ra\{\pm 1\}$ such that
\e
F'(N,[s])=F(N,[s])\cdot \ep([N]) \qquad\text{for all\/ $(N,[s])$.}
\label{ca3eq8}
\e
\item[{\bf(c)}] Let\/ $F$ be a flag structure on $X$ and\/ $\ep:H_3(X;\Z)\ra\{\pm 1\}$ a group morphism, and define $F'$ in \eq{ca3eq5} by \eq{ca3eq8}. Then $F'$ is a flag structure on~$X$.
\end{itemize}
Parts\/ {\bf(a)}--{\bf(c)} imply that the set\/ $\FlagSt(X)$ of flag structures on $X$ is a torsor over $\Hom_{\rm Grp}\bigl(H_3(X;\Z),\{\pm 1\}\bigr)$.
\label{ca3prop3}	
\end{prop}

\begin{proof} For (a), let $V$ be the image of the projection $H_3(X;\Z)\ra H_3(X;\Z_2)$. It is a $\Z_2$-vector space, as $\Z_2$ is a field. Choose a basis $e_i:i\in I$ for $V$. The indexing set $I$ is countable, and finite if $X$ is compact. For each $i\in I$, choose an embedded flagged submanifold $(N_i,[s_i])$ in $X$ with $[N_i]=e_i$ in $H_3(X;\Z_2)$. As there are at most countably many $N_i$, we can choose them to be disjoint. For each $i\in I$, choose~$\de_i=\pm 1$.

We will construct a flag structure $F$ with $F(N_i,[s_i])=\de_i$. Let $(N,[s])$ be an immersed flagged submanifold in $X$. Then $[N]\in V\subseteq H_3(X;\Z_2),$ so as the $e_i$ are a basis for $V$ there is a unique finite subset $J\subseteq I$ with $[N]=\sum_{j\in J}e_j$ in $H_3(X;\Z_2)$. Choose a small perturbation $(N',[s'])$ of $(N,[s])$ such that $N'$ is embedded in $X$ and disjoint from $N_j$ for all $j\in J$. Observe that Definition \ref{ca3def2} and Propositions \ref{ca3prop1}--\ref{ca3prop2} make sense in homology over $\Z_2$ as well as over $\Z$, so we can define $D_{\Z_2}((N_1,[s_1]),(N_2,[s_2]))\in\Z_2$ if $(N_1,[s_1]),(N_2,[s_2])$ are embedded submanifolds with $[N_1]=[N_2]\in H_3(X;\Z_2)$. Thus we may set
\begin{equation*}
F(N,[s])=(-1)^{D_{\Z_2}((N',[s']),(\coprod_{j\in J}N_j,\coprod_{j\in J}[s_j]))}\cdot \ts\prod_{j\in J}\de_j,	
\end{equation*}
since $[N']=[N]=[\coprod_{j\in J}N_j]$ in $H_3(X;\Z_2)$. Propositions \ref{ca3prop1} and \ref{ca3prop2} imply that this is independent of the choice of perturbation $(N',[s'])$, so $F(N,[s])$ is well defined. From \eq{ca3eq1}--\eq{ca3eq3} and by construction it is not difficult to show that $F$ is a flag structure, proving~(a).

For (b), suppose $(N_1,[s_1]),(N_2,[s_2])$ are immersed flagged submanifolds with $[N_1]=[N_2]=\al\in H_3(X;\Z)$. Choose another immersed flagged submanifold $(N_3,[s_3])$ with $[N]=\al$ and $N$ disjoint from both $N_1,N_2$. Then
\begin{align*}
&F'(N_1,[s_1])F(N_1,[s_1])^{-1}\\
&\!=\!\bigl[F'(N_3,[s_3])\!\cdot \!(-1)^{D((N_3,[s_3]),(N_1,[s_1]))}\bigr]\!\cdot\!\bigl[F(N_3,[s_3])\!\cdot\! (-1)^{D((N_3,[s_3]),(N_1,[s_1]))}\bigr]{}^{-1}\\
&\!=\!F'(N_3,[s_3])F(N_3,[s_3])^{-1}\\
&\!=\!\bigl[F'(N_2,[s_2])\!\cdot \!(-1)^{D((N_2,[s_2]),(N_3,[s_3]))}\bigr]\!\cdot\!\bigl[F(N_2,[s_2])\!\cdot\! (-1)^{D((N_2,[s_2]),(N_3,[s_3]))}\bigr]{}^{-1}\\
&\!=\!F'(N_2,[s_2])F(N_2,[s_2])^{-1},
\end{align*}
by Definition \ref{ca3def3}(iii) for $F,F'$. Thus $F'(N,[s])F(N,[s])^{-1}$ depends only on the homology class $[N]\in H_3(X;\Z)$. Hence there exists a unique map $\ep:H_3(X;\Z)\ra\{\pm 1\}$ with $F'(N,[s])F(N,[s])^{-1}=\ep([N])$, so that \eq{ca3eq8} holds. 

Dividing \eq{ca3eq7} for $F'$ by \eq{ca3eq7} for $F$ yields $\ep([N_1\amalg N_2])=\ep([N_1])\cdot \ep([N_2])$, so $\ep(\al+\be)=\ep(\al)\ep(\be)$ for $\al,\be\in H_3(X;\Z)$, and $\ep:H_3(X;\Z)\ra\{\pm 1\}$ is a group morphism. This proves (b). Part (c) is easy to check from Definition~\ref{ca3def3}.
\end{proof}

\subsection{Canonical flags of associatives, and orientations}
\label{ca32}

Given any compact, immersed associative $i:N\ra X$ in a TA-$G_2$-manifold $(X,\vp,\psi)$, we will define a flag $[s]$ for $N$. To do this we will need the notion of {\it spectral flow\/} introduced by Atiyah, Patodi and Singer~\cite[\S 7]{APS}.

\begin{dfn} Let $N$ be a compact manifold, and suppose that for all $t\in[0,1]$ we are given a vector bundle $E_t\ra N$ and a linear first-order elliptic operator $A_t:\Ga^\iy(E_t)\ra\Ga^\iy(E_t)$, which is self-adjoint with respect to some metrics $g_t$ on $N$ and $h_t$ on the fibres of $E_t$, where $E_t,A_t,g_t,h_t$ depend continuously on $t\in[0,1]$. Then Atiyah et al.\ \cite[\S 7]{APS} define the {\it spectral flow\/} $\SF(A_t:t\in[0,1])\in\Z$. 

Heuristically, $\SF(A_t:t\in[0,1])\in\Z$ is the number of eigenvalues $\la$ of $A_t$ which cross from $\la\in(-\iy,0)$ to $\la\in[0,\iy)$ as we deform $t$ from 0 to 1, counted with signs. We need the $A_t$ to be self-adjoint so that their eigenvalues are real.
\label{ca3def4}
\end{dfn}

If $E_0=E_1$, $A_0=A_1$ then (for simplicity assuming $E_t,A_t$ are smooth in $t\in\cS^1=\R/\Z$) we may define a vector bundle $E\ra N\t\cS^1$ by $E\vert_{N\t\{t\}}=E_t$ and an elliptic operator $A:\Ga^\iy(E)\ra\Ga^\iy(E)$ by $A\vert_{N\t\{t\}}=A_t+\frac{\pd}{\pd t}$, and then \cite[Th.~7.4]{APS} shows that $\SF(A_t:t\in[0,1])=\ind(A)$, which may be computed using the Atiyah--Singer Index Theorem.

\begin{dfn} Let $(X,\vp,\psi)$ be a TA-$G_2$-manifold, and $i:N\ra X$ be a compact, immersed associative 3-fold in $X$. Write $g$ for the Riemannian metric on $X$ from the $G_2$-structure associated to $\psi$, and $\nu\ra N$ for the normal bundle of $N$ in $X$. Then Theorem \ref{ca2thm2} defines a first-order linear elliptic operator $\bD:\Ga^\iy(\nu)\ra\Ga^\iy(\nu)$, which by Lemma \ref{ca2lem1} is self-adjoint with respect to the metrics induced by $g$, as we assume $\d\psi=0$ for TA-$G_2$-manifolds~$(X,\vp,\psi)$.

Choose a flag $[s]$ for $N$, and choose a representative $s$ for $[s]$ which is of constant length 1 for the metric on $\nu$ induced by $g$. Now $\bD$ is a twisted Dirac operator on $N$. Another example of a twisted Dirac operator on $N$ is
\e
\d*\!+\!*\d=\begin{pmatrix} 0 & *\d \\ *\d & \d * \end{pmatrix}:\Ga^\iy(\La^0T^*N\!\op\!\La^2 T^*N)\ra \Ga^\iy(\La^0T^*N\!\op\!\La^2 T^*N).
\label{ca3eq9}
\e
It is easy to see that there is a unique isomorphism $\nu\cong \La^0T^*N\op\La^2 T^*N$ which identifies $s$ with $1\op 0$ in $\Ga^\iy(\La^0T^*N\op\La^2 T^*N)$, and identifies the symbols of $\bD$ and $\d*+*\d$. Under this identification, $\bD$ and $\d*+*\d$ differ by an operator of order zero, since their symbols (first-order parts) agree. Thus we have 
\e
\bD\cong \d*+*\d+B:\Ga^\iy(\La^0T^*N\op\La^2 T^*N)\longra \Ga^\iy(\La^0T^*N\op\La^2 T^*N),
\label{ca3eq10}
\e
for some unique vector bundle morphism
\e
B=\bigl(\begin{smallmatrix} B_{00} & B_{02} \\ B_{20} & B_{22} \end{smallmatrix}\bigr):\La^0T^*N\op\La^2 T^*N\longra \La^0T^*N\op\La^2 T^*N.
\label{ca3eq11}
\e

Define a family of self-adjoint first order linear elliptic operators
\e
A_t:\Ga^\iy(\La^0T^*N\op\La^2 T^*N)\longra \Ga^\iy(\La^0T^*N\op\La^2 T^*N)
\label{ca3eq12}
\e
for $t\in[0,1]$ by $A_t=\d*+*\d+tB$. Then $A_0=\d*+*\d$ in \eq{ca3eq9}, and $A_1\cong\bD$ under our isomorphism $\La^0T^*N\op\La^2 T^*N\cong\nu$. Thus as in Definition \ref{ca3def4} we have the spectral flow $\SF(A_t:t\in[0,1])\in\Z$.

Suppose $s,s'$ are non-vanishing sections of $\nu\ra N$ yielding flags $[s],[s']$, and $A_t:t\in[0,1]$, $A'_t:t\in[0,1]$ the corresponding families of elliptic operators. Definition \ref{ca3def1} defines $d(s,s')\in\Z$. By using \cite[Th.~7.4]{APS} and computing the index of a Dirac-type operator on $N\t\cS^1$ by the Atiyah--Singer Index Theorem, we can show that (up to the sign of $d(s,s')$)
\e
\SF(A'_t:t\in[0,1])=\SF(A_t:t\in[0,1])+d(s',s).
\label{ca3eq13}
\e

This implies that $\SF(A_t:t\in[0,1])$ depends only on the flag $[s]$, not on the representative $s$. Also, since $\Flag(N)$ is a $\Z$-torsor as in \S\ref{ca31}, there is a unique flag $f_N$ on $N$, called the {\it canonical flag\/} of $N$, such that $\SF(A_t:t\in[0,1])=0$ for $A_t:t\in[0,1]$ constructed using $s\in f_N$. It has the property that for any flag $[s]$ for $N$ and family $A_t:t\in[0,1]$ constructed from $s\in[s]$ as above, we have
\e
f_N=[s]+\SF(A_t:t\in[0,1]).
\label{ca3eq14}
\e

\label{ca3def5}	
\end{dfn}

\begin{rem} Suppose $(X,\vp,*\vp)$ is a torsion-free compact $G_2$-manifold, and $N\subset X$ is a compact, unobstructed associative 3-fold in $X$, and $(W,\Om)$ is an Asymptotically Cylindrical $\Spin(7)$-manifold (not necessarily torsion-free) with $\Spin(7)$ 4-form $\Om$, with one end asymptotic to $(X\t(0,\iy),\d t\w\vp+*\vp)$, and $M\subset W$ is a closed, Asymptotically Cylindrical Cayley 4-fold in $W$, with one end asymptotic to $N\t(0,\iy)$ in~$X\t(0,\iy)$.

Ohst \cite{Ohst} studies the deformation theory of $M$ in $X$. We can interpret \cite[Prop.~19]{Ohst} in our language as saying that the moduli space $\cM_M$ of Asymptotically Cylindrical Cayley deformations of $M$ in $(W,\Om)$ has virtual dimension
\begin{equation*}
\mathop{\rm vdim}\cM_M=\ha\bigl(\chi(M)+\si(M)-b^0(N)-b^1(N)\bigr)-e(\nu_M,f_N),
\end{equation*}
where $\chi(M),\si(M)$ are the Euler characteristic and signature of $M$ (the sign of $\si(M)$ depends on the model for $\Spin(7)$ 4-forms $\Om$, we follow \cite{Joyc3,Joyc20}), and $\nu_M$ is the normal bundle of $M$ in $W$, and $e(\nu_M,f_N)$ is the Euler class of $\nu_M$ relative to the canonical flag $f_N$ at infinity in $M$. That is, $e(\nu_M,f_N)$ is the number of zeroes, counted with signs, of a generic section $s$ of $\nu_M\ra M$ asymptotic to a nonvanishing section $s'$ of the normal bundle $\nu_N$ of $N$ in $X$ with~$[s']=f_N$.

\label{ca3rem1}	
\end{rem}

Suppose that for $u\in(-\ep,\ep)$ we are given a TA-$G_2$-manifold $(X,\vp_u,\psi_u)$ and compact immersed associative $N_u$ in $(X,\vp_u,\psi_u)$, both varying smoothly with $u$. Consider how the canonical flag $f_{N_u}$ of $N_u$ varies with $u\in(-\ep,\ep)$. Choose $s_u\in\Ga^\iy(\nu_u)$ depending smoothly on $u\in(-\ep,\ep)$ and of constant length 1 in the metric $g_u$ associated to $\psi_u$, and let $A_{t,u}:t\in[0,1]$ be the family of operators associated to $(X,\vp_u,\psi_u),N_u,s_u$ in Definition \ref{ca3def5}. Then by \eq{ca3eq14} we have 
\begin{equation*}
f_{N_u}=[s_u]+\SF(A_{t,u}:t\in[0,1]).
\end{equation*}

Here the flag $[s_u]$ varies smoothly with $u\in(-\ep,\ep)$, so $f_{N_u}$ varies smoothly with $u$ if and only if $\SF(A_{t,u}:t\in[0,1])$ is constant in $u$. Since $A_{t,u}$ depends smoothly on $t,u$ the only way $\SF(A_{t,u}:t\in[0,1])$ could fail to be constant in $u$ is if either
\begin{itemize}
\setlength{\itemsep}{0pt}
\setlength{\parsep}{0pt}
\item[(a)] an eigenvalue of $A_{0,u}=\d*_u+*_u\d$ crosses 0 as $u$ varies; or
\item[(b)] an eigenvalue of $A_{1,u}\cong\bD_u$ crosses 0 as $u$ varies.
\end{itemize}
Now by Hodge theory, $\Ker A_{0,u}\cong H^0(N;\R)\op H^2(N;\R)$, which is of constant dimension. Thus (a) is impossible. Hence $f_{N_u}$ must vary smoothly with $u$ unless $\Ker\bD_u=\cT_{N_u}$ jumps as $u$ varies. In particular, if $N_u$ is unobstructed for all $u\in(-\ep,\ep)$ then $\Ker\bD_u=0$, so (b) does not happen. This proves:

\begin{prop} Suppose that for\/ $u\in(-\ep,\ep)$ we are given a TA-$G_2$-manifold\/ $(X,\vp_u,\psi_u)$ and a compact, immersed, unobstructed associative $3$-fold\/ $N_u$ in $(X,\vp_u,\psi_u),$ both varying smoothly with\/ $u$. Then the canonical flag $f_{N_u}$ of\/ $N_u$ varies continuously with $u$ in~$(-\ep,\ep)$.
\label{ca3prop4}	
\end{prop}

Now we explain how to orient moduli spaces of associatives.

\begin{dfn} Let $(X,\vp,\psi)$ be a TA-$G_2$-manifold. Choose a flag structure $F$ on $X$, which is possible by Proposition \ref{ca3prop3}(a). The orientations on moduli spaces we define will depend on this choice. Let $N$ be a compact, immersed, unobstructed associative 3-fold in $(X,\vp,\psi)$. Then Definition \ref{ca3def5} defines a canonical flag $f_N$ for $N$. Define $\Or(N)=F(N,f_N)$, so that $\Or(N)=\pm 1$. 

If we take $\psi$ to be generic, and assume Conjecture \ref{ca2conj2}, then all compact associatives are unobstructed, so this defines maps $\Or:\cM(\cN,\al,\psi)\ra\{\pm 1\}$ for all $\cN,\al$. We think of $\Or$ as an {\it orientation on the\/ $0$-manifold\/} $\cM(\cN,\al,\psi)$, since in dimension 0 an orientation is a choice of sign for each point. Note that $\Or(N)$ {\it is not an orientation on\/} $N$, which already has a natural orientation.
\label{ca3def6}	
\end{dfn}

Combining Proposition \ref{ca3prop4} and Definition \ref{ca3def3}(i) yields:

\begin{cor} Suppose that for\/ $u\in(-\ep,\ep)$ we are given a TA-$G_2$-manifold\/ $(X,\vp_u,\psi_u)$ and a compact, immersed, unobstructed associative $3$-fold\/ $N_u$ in $(X,\vp_u,\psi_u),$ both varying smoothly with\/ $u$. Fix a flag structure $F$ on $X$. Then the orientation $\Or(N_u)$ is constant in $u\in(-\ep,\ep)$.
\label{ca3cor1}	
\end{cor}

The next conjecture should be proved using similar methods to
Fukaya--Oh--Ohta--Ono's treatment \cite[\S 8]{FOOO} of orientations on Kuranishi space moduli spaces of $J$-holomorphic discs.

\begin{conj} Assume Conjecture\/ {\rm\ref{ca2conj1}}. Then for any TA-$G_2$-manifold $(X,\vp,\psi)$ we have Kuranishi spaces $\bcM(\cN,\al,\psi),$ the moduli spaces of associative $3$-folds in $(X,\vp,\psi),$ and for any smooth family of TA-$G_2$-manifolds $(X,\vp_t,\psi_t):t\in\cF,$ we have $1$-morphisms of Kuranishi spaces $\bs\pi:\bcM(\cN,\al,\psi_t:t\in\cF)\ra\cF,$ interpreted as families of moduli spaces $\bcM(\cN,\al,\psi_t)$ over the base $\cF$.

Choose a flag structure $F$ for $X$. Using the ideas on canonical flags above, we can construct orientations for the Kuranishi spaces $\bcM(\cN,\al,\psi)$ and coorientations for the $1$-morphisms $\bs\pi:\bcM(\cN,\al,\psi_t:t\in\cF)\ra\cF,$ for all\/ $\cN,\al$. These (co)orientations are compatible with pullbacks of families $(X,\vp_t,\psi_t):t\in\cF,$ and agree with those in Definition\/ {\rm\ref{ca3def6}} for unobstructed\/~$[N,i]\in\bcM(\cN,\al,\psi)$.

\label{ca3conj1}	
\end{conj}

The next example describes the typical way in which the author expects orientations of associatives to change discontinuously in a family.

\begin{ex} Let $(X,\vp_s,\psi_s)$ for $s\in(-\ep^2,\ep^2)$ be a smooth family of TA-$G_2$-manifolds, and $i_t:N\ra X$ for $t\in(-\ep,\ep)$ a family of compact, immersed 3-submanifolds, with $N_t:=i_t(N)$ associative in $(X,\vp_{t^2},\psi_{t^2})$ for $s=t^2$. Write $\bD_t$ for the operator $\bD$ in Theorem \ref{ca2thm2} for $N_t$. Suppose $N_t$ is unobstructed for $t\ne 0$, so that $\Ker\bD_t=0$ for~$t\ne 0$. 

As $t\mapsto s=t^2$ is stationary at $t=0$, we see that $\frac{\d}{\d t}i_t\vert_{t=0}$ is an infinitesimal deformation of $N_0$ as an associative 3-fold in $(X,\vp_0,\psi_0)$, and lies in $\Ker\bD_0$. We suppose that $\Ker\bD_0=\bigl\langle\frac{\d}{\d t}i_t\vert_{t=0}\bigr\rangle\cong\R$. Thus, $\Ker\bD_t$ is 0 for $t\ne 0$ and $\R$ for $t=0$. This happens because an eigenvalue $\la$ of $\bD_t$ crosses 0 as $t$ increases through zero, crossing either from $\la<0$ to $\la>0$, or from $\la>0$ to $\la<0$. 

Thus the canonical flag $f_{N_t}$ of $N_t$ changes discontinuously by $\pm 1$ as $t$ passes through zero. If we fix a flag structure $F$ on $X$, so that Definition \ref{ca3def6} defines orientations of compact, unobstructed associative 3-folds, then $\Or(N_t)$ changes sign as $t$ passes through zero. Thus we can suppose that
\begin{equation*}
\Or(N_t)=\begin{cases} -1, & t<0, \\ 1, & t>0. \end{cases}
\end{equation*}
This does not contradict Corollary \ref{ca3cor1}, as $N_0$ is obstructed.

When $s<0$ we have no associative 3-folds of interest in $(X,\vp_s,\psi_s)$, but when $s>0$ we have two compact, unobstructed associative 3-folds $N_t,N_{-t}$ for $t=\sqrt{s}$, with opposite orientations. Thus, if we count associative 3-folds $N$ weighted by orientations $\Or(N)$, the number will not change under this transition, making it plausible that we might get a deformation-invariant answer. Note that the use of spectral flow in defining orientations, so that $\Or(N_t)$ changes sign when eigenvalues of $\bD_t$ cross zero, is crucial here. If we counted associatives without orientations, the number would not be deformation-invariant. 
\label{ca3ex1}	
\end{ex}

\begin{rem} We have been discussing associative 3-folds $N$ in a TA-$G_2$-manifold $(X,\vp,\psi)$, which by definition has $\d\psi=0$. We now consider how the theory changes if we allow~$\d\psi\ne 0$.

In \S\ref{ca26}, the moduli spaces $\cM(\cN,\al,\psi)$, McLean's Theorem \ref{ca2thm2}, and Conjectures \ref{ca2conj1} and \ref{ca2conj2} remain unchanged when $\d\psi\ne 0$. However, as in Lemma \ref{ca2lem1} the twisted Dirac operator $\bD$ in Theorem \ref{ca2thm2} is no longer self-adjoint if $\d\psi\ne 0$, though it does have self-adjoint symbol. This affects the spectral flow term $\SF(A_t:t\in[0,1])$ in Definition~\ref{ca3def5}.

For non-self-adjoint operators $A_t$ of this type, eigenvalues $\la$ are either real, or occur in complex-conjugate pairs $\la,\bar\la$ in $\C\sm\R$. To define $\SF(A_t:t\in[0,1])$, we must count eigenvalues that cross the imaginary axis $i\R$ in $\C$ as $t$ increases from 0 to 1. So when $\d\psi\ne 0$ we have a new phenomenon, that a pair $\la,\bar\la$ in $\C\sm\R$ can cross $i\R$ at $t\in(0,1)$, changing $\SF(A_t:t\in[0,1])$ by $\pm 2$. For $\bD$ to have imaginary eigenvalues does not make $N$ unobstructed, and does not correspond to any qualitative change in the families of associative 3-folds in~$(X,\vp,\psi)$.

As a consequence, the analogue of Proposition \ref{ca3prop4} with $\d\psi_u\ne 0$ should be {\it false\/}: given families $(X,\vp_u,\psi_u)$ and compact, unobstructed associative $3$-folds $N_u$ in $(X,\vp_u,\psi_u)$ varying smoothly with $u\in(-\ep,\ep)$, but allowing $\d\psi_u\ne 0$, the canonical flag $f_{N_u}$ of $N_u$ need not vary continuously with $u$ in $(-\ep,\ep)$, but can jump by $\pm 2$ when conjugate pairs of eigenvalues of $\bD_u$ cross $i\R$. However, because these jumps in canonical flags are even, the analogue of Corollary \ref{ca3cor1} with $\d\psi_u\ne 0$, and also Conjecture \ref{ca3conj1}, should remain true. 

In conclusion: for associative 3-folds in $(X,\vp,\psi)$ with $\d\psi\ne 0$, the author expects the theory of orientations on moduli spaces $\cM(\cN,\al,\psi)$ outlined above to continue to work nicely. But the canonical flags $f_N$ lose the continuity property in Proposition \ref{ca3prop4}, which is important for our proposal in Conjecture~\ref{ca1conj1}.

An aside: when a pair of eigenvalues $\la,\bar\la$ cross $i\R$ for $N_u$, say at $\la=is$ for $s>0$, the author expects a new $\cS^1$ family of Cayley 4-folds $N_u\t\cS^1_s$ to appear in the $\Spin(7)$-manifold $X\t\cS^1_s$, where $\cS^1_s=\R/2\pi s\Z$. So one might be able to compensate for this phenomenon by counting Cayley 4-folds in~$X\t\cS^1_s$. 
\label{ca3rem2}	
\end{rem}

\section{An index 1 singularity of associative 3-folds}
\label{ca4}

We now describe the first of two kinds of singularity of associative 3-folds that will be crucial to our discussion.

\subsection{\texorpdfstring{A family of SL 3-folds in $\C^3$}{A family of SL 3-folds in ℂ³}}
\label{ca41}

We describe a family of explicit SL 3-folds $K_{\bs{\phi},s}$
in $\C^3$. This family was first found by Lawlor \cite{Lawl}, was made
more explicit by Harvey \cite[p.~139--140]{Harv}, and was discussed from
a different point of view by the author in \cite[\S 5.4(b)]{Joyc6}. Our
treatment is based on that of Harvey.

Let $a_1,a_2,a_3>0$, and define polynomials $p(x)$, $P(x)$ by
\begin{equation*}
p(x)=(1+a_1x^2)(1+a_2x^2)(1+a_3x^2)-1
\quad\text{and}\quad P(x)=\frac{p(x)}{x^2}.
\end{equation*}
Define real numbers $\phi_1,\phi_2,\phi_3$ and $s$ by
\begin{equation*}
\phi_k=a_k\int_{-\iy}^\iy\frac{\d x}{(1+a_kx^2)\sqrt{P(x)}}
\quad\text{and}\quad s=\frac{1}{3}(a_1a_2a_3)^{-1/2}.
\end{equation*}
Clearly $\phi_k>0$ and $s>0$. But writing $\phi_1+\phi_2+\phi_3$
as one integral and rearranging gives
\begin{equation*}
\phi_1+\phi_2+\phi_3=\int_0^\iy\frac{p'(x)\d x}{(p(x)+1)\sqrt{p(x)}}
=2\int_0^\iy\frac{\d w}{w^2+1}=\pi,
\end{equation*}
making the substitution $w=\sqrt{p(x)}$. So $\phi_k\in(0,\pi)$
and $\phi_1+\phi_2+\phi_3=\pi$. This yields a 1-1 correspondence between triples $(a_1,a_2,a_3)$ with $a_k>0$, and quadruples $(\phi_1,\phi_2,\phi_3,s)$ with $\phi_k\in(0,\pi)$,
$\phi_1+\phi_2+\phi_3=\pi$ and~$s>0$.

For $k=1,2,3$ and $y\in\R$, define $z_k(y)$ by $z_k(y)=
{\rm e}^{i\psi_k(y)}\sqrt{a_k^{-1}+y^2}$, where 
\begin{equation*}
\psi_k(y)=a_k\int_{-\iy}^y\frac{\d x}{(1+a_kx^2)\sqrt{P(x)}}\,.
\end{equation*}
Now write $\bs{\phi}=(\phi_1,\phi_2,\phi_3)$, and define 
a submanifold $K_{\bs{\phi},s}$ in $\C^3$ by
\e
K_{\bs{\phi},s}=\bigl\{(z_1(y)x_1,z_2(y)x_2,z_3(y)x_3):
y\in\R,\; x_k\in\R,\; x_1^2+x_2^2+x_3^2=1\bigr\}.
\label{ca4eq1}
\e
Our next result comes from Harvey~\cite[Th.~7.78]{Harv}.

\begin{prop} The set\/ $K_{\bs{\phi},s}$ defined in \eq{ca4eq1} is an
embedded SL\/ $3$-fold in $\C^3$ diffeomorphic to ${\mathcal S}^2\t\R$. It
is asymptotically conical at rate $O(r^{-2})$ to the union $\Pi_0\cup\Pi_{\bs\phi}$ of two special Lagrangian $3$-planes $\Pi_0,\Pi_{\bs\phi}$ given by
\begin{equation*}
\Pi_0=\bigl\{(x_1,x_2,x_3):x_j\in\R\bigr\},\;\>
\Pi_{\bs\phi}=\bigl\{({\rm e}^{i\phi_1}x_1,{\rm e}^{i\phi_2}x_2,
{\rm e}^{i\phi_3}x_3):x_j\in\R\bigr\}.
\end{equation*}

\label{ca4prop1}
\end{prop}

An easy calculation shows that near $\Pi_0$ for small $s>0$ we have
\e
\begin{split}
K_{\bs{\phi},s}\approx\bigl\{(1+isr^{-3})(x_1,x_2,x_3)+O(s^{5/3}r^{-4}):(x_1,x_2,x_3)\in\R^3&,\\ 
r=(x_1^2+x_2^2+x_3^2)^{1/2}\gg 0&\bigr\}.
\end{split}
\label{ca4eq2}
\e

The next proposition can be proved from Proposition \ref{ca4prop1} and Remark~\ref{ca2rem1}(ii).

\begin{prop} Suppose $V,V'$ are $3$-dimensional vector subspaces of\/ $\R^7$ which are associative, with\/ $V\cap V'=\{0\}$. Then there exists an isomorphism $\R^7\cong\R\t\C^3$ such that\/ \eq{ca2eq4} holds, which identifies $V\subset\R^7$ with\/ $\{0\}\t\Pi_0\subset\R\t\C^3$ and\/ $V'\subset\R^7$ with\/ $\{0\}\t\Pi_{\bs\phi}\subset\R\t\C^3,$ for some unique $\bs\phi=(\phi_1,\phi_2,\phi_3)$ in $(0,\pi)^3$ with\/~$\phi_1+\phi_2+\phi_3=\pi$. 

Hence there is a family of associative\/ $3$-folds $K_s^{V,V'}\subset\R^7$ for $s>0$ identified with $\{0\}\t K_{\bs{\phi},s}\subset\R\t\C^3,$ such that\/ $K_s^{V,V'}$ is diffeomorphic to $\cS^2\t\R,$ and is Asymptotically Conical, with cone $V\cup V'$. This family is independent of the choice of isomorphism\/~$\R^7\cong\R\t\C^3$.
\label{ca4prop2}	
\end{prop}

We could think of $V\cup V'$ as a singular associative 3-fold in $\R^7$ with a singularity at 0, and $K_s^{V,V'}$ for $s>0$ as a family of associative smoothings of $V\cup V'$. However, it is more helpful to regard $V\cup V'$ as a nonsingular, {\it immersed\/} associative 3-fold with a self-intersection point at~0.

Let us describe $K_s^{V,V'}$ near $V\sm\{0\}$ for small $s>0$. From \eq{ca4eq2} we see that we may choose Euclidean coordinates $(x_1,x_2,x_3)$ on $V$ and $(x_4,x_5,x_6,x_7)$ on the orthogonal complement $V^\perp$ in $\R^7$, which we identify with the normal bundle $\nu_V$ of $V$ in $\R^7$, such that
\e
K_s^{V,V'}\approx \Ga_{s\ze}+O(s^{5/3}r^{-4})\quad\text{near $V\sm\{0\}$ for small $s>0$,}
\label{ca4eq3}
\e
with $\Ga_{s\ze}$ the graph of $s\ze$ in $\nu$, where $\ze\in\Ga^\iy(\nu_V\vert_{V\sm\{0\}})$ is given by
\begin{equation*}
\ze(x_1,x_2,x_3)=(r^{-3}x_1,r^{-3}x_2,r^{-3}x_3,0),\quad r=(x_1^2+x_2^2+x_3^2)^{1/2}.
\end{equation*}
Let $\bD_V:\Ga^\iy(\nu_V)\ra\Ga^\iy(\nu_V)$ be the operator of Theorem \ref{ca2thm2} for the associative $V$ in $\R^7$. Then $\bD(\ze)=0$ on $V\sm\{0\}$, since $\ze$ is an associative deformation of $V$. In fact we can regard $\ze$ as a section of $\nu$ on $V$ in {\it currents\/} (a kind of generalized section). Then calculation shows that in currents we have
\begin{equation*}
\bD_V(\ze)=4\pi\,\de_0\cdot (0,0,0,1),
\end{equation*}
with $\de_0$ the delta function on $V$ at 0, in the sense of currents.

\subsection{Desingularizing immersed associative 3-folds}
\label{ca42}

The next definition sets up notation for a conjecture on an index one singularity of associative 3-folds.

\begin{dfn} Suppose that for $t\in(-\ep,\ep)$ we are given a TA-$G_2$-manifold $(X,\vp_t,\psi_t)$ and a compact, immersed, unobstructed associative 3-fold $i_t:N\ra X$ in $(X,\vp_t,\psi_t)$, both varying smoothly with $t$. We write $N_t=i_t(N)$. Suppose there are distinct points $x^\pm$ in $N$ with $i_0(x^+)=i_0(x^-)=x$ in $X$, and these are the only immersed points in~$i_0:N\ra X$.

We will be interested in two separate cases:
\begin{itemize}
\setlength{\itemsep}{0pt}
\setlength{\parsep}{0pt}
\item[(a)] $N$ is a disjoint union $N=N^+\amalg N^-$, where $N^\pm$ are connected with $x^+\in N^+$ and $x^-\in N^-$, and $i_t\vert_{N^\pm}$ are embeddings.
\item[(b)] $N$ is connected.
\end{itemize}

Write $\Pi^+=\d_{x^+}i_0(T_{x^+}N)$ and $\Pi^-=\d_{x^-}i_0(T_{x^-}N)$, as associative 3-planes in $T_xX$, and suppose $\Pi^+\cap\Pi^-=\{0\}$, so that we have a splitting
\e
T_xX=\Pi^+\op \Pi^-\op\langle v\rangle_\R,
\label{ca4eq4}
\e
where $v\in T_xX$ is chosen uniquely to be of unit length, orthogonal to $\Pi^+\op \Pi^-$, with \eq{ca4eq4} compatible with the orientations of $\Pi^+,\Pi^-,\langle v\rangle_\R\cong\R$ and~$T_xX$.

Proposition \ref{ca4prop2} gives a unique family of associative 3-folds $K_s$, $s>0$ in $T_xX$ asymptotic at rate $O(r^{-2})$ to $\Pi^+\cup\Pi^-$. Conjecture \ref{ca4conj1} explains when we expect there to exist a compact associative 3-fold $\ti N_t$ in $(X,\vp_t,\psi_t)$ which is close to $i_0(N)$ away from $x$ in $X$, and close to $K_s$ near $x$, identifying $X\cong T_xX$ near $x$. To state the conjecture we first need to define two real numbers~$\ga,\de$.

Now $\frac{\d}{\d t}i_t(x^+)\vert_{t=0}$ and $\frac{\d}{\d t}i_t(x^-)\vert_{t=0}$ lie in $T_xX$. Define $\ga\in\R$ by
\begin{equation*}
\ga=v\cdot \ts\bigl(\frac{\d}{\d t}i_t(x^+)\vert_{t=0}-\frac{\d}{\d t}i_t(x^-)\vert_{t=0}\bigr).
\end{equation*}
Then $\ga$ measures the speed at which the two sheets of $N_t$ near $x$ in $X$ cross each other as $t$ increases through 0 in~$(-\ep,\ep)$.

The discussion at the end of \S\ref{ca41} gives $O(r^{-2})$ sections $\ze^+$ of $\nu_{\Pi^+}\vert_{\Pi^+\sm\{0\}}$ and $\ze^-$ of $\nu_{\Pi^-}\vert_{\Pi^-\sm\{0\}}$ such that 
\begin{align*}
K_s&\approx \Ga_{s\ze^+}+O(s^{5/3}r^{-4})\quad\text{near $\Pi^+\sm\{0\}$ for small $s>0$,}\\
K_s&\approx \Ga_{s\ze^-}+O(s^{5/3}r^{-4})\quad\text{near $\Pi^-\sm\{0\}$ for small $s>0$.}\\
\end{align*}
These $\ze^\pm$ make sense as currents on all of $\Pi^\pm$, and satisfy
\e
\bD_{\Pi^+}(\ze^+)=4\pi\,\de_0\cdot v,\quad
\bD_{\Pi^-}(\ze^-)=-4\pi\,\de_0\cdot v,
\label{ca4eq5}
\e 
where $v$ in \eq{ca4eq4} is a normal vector to both $\Pi^+$ and $\Pi^-$.

Now let $\bD_{N_0}:\Ga^\iy(\nu_{N_0})\ra\Ga^\iy(\nu_{N_0})$ be the operator from Theorem \ref{ca2thm2} for $N_0$ in $(X,\vp_0,\psi_0)$. It is an isomorphism, as $N_0$ is unobstructed. So its extension to currents is also an isomorphism. Thus there exists a unique current section $\chi$ of $\nu_{N_0}$ such that
\begin{equation*}
\bD_{N_0}(\chi)=4\pi\,\de_{x^+}\cdot v-4\pi\,\de_{x^-}\cdot v.
\end{equation*}
Then $\chi$ is smooth on $N_0\sm\{x^+,x^-\}$, and from \eq{ca4eq5} we see that $\chi-\ze^+$ is smooth near $x^+$, and $\chi-\ze^-$ is smooth near $x^-$. 

Near $x^+$ in $N$, under the splitting \eq{ca4eq4}, the section $\chi\approx\ze^+$ of $\nu_{N_0}$ has a pole in the $\Pi^-$ factor in \eq{ca4eq4}, but remains continuous in the $\langle v\rangle$-factor, so that $\lim_{x\ra x^+}v\cdot\chi(x)$ exists in $\R$, and similarly $\lim_{x\ra x^-}v\cdot\chi(x)$ exists. Define $\de=\lim_{x\ra x^+}v\cdot\chi(x)-\lim_{x\ra x^-}v\cdot\chi(x)$ in~$\R$.

The point of this is if we try to define an associative 3-fold $\ti N_s$ in $(X,\vp_0,\psi_0)$ by gluing $K_s$ for small $s$ into $N_0$ at $x$, then $\ti N$ should look like the graph of $s\chi$ near $i_0(N)\sm\{x\}$ to leading order in $s$. But the two ends of this graph only fit together to leading order in $s$ if $\de=0$, so $\de$ is the first-order obstruction to deforming $N_0$ to an associative 3-fold $\ti N_s$ in the fixed TA-$G_2$-manifold $(X,\vp_0,\psi_0)$, rather than in $(X,\vp_t,\psi_t)$ for some $t$. 

To make Conjecture \ref{ca4conj1} simpler, we suppose $\ga\ne 0\ne\de$. This should hold if $(X,\vp_t,\psi_t):t\in(-\ep,\ep)$ is a generic 1-parameter family of TA-$G_2$-manifolds.  
\label{ca4def1}	
\end{dfn}

\begin{conj} Work in the situation of Definition\/ {\rm\ref{ca4def1}}. Then for all sufficiently small\/ $t\in(-\ep,\ep)$ with\/ $\ga\de^{-1} t<0$ there exists a unique compact, embedded, unobstructed associative $3$-fold\/ $\ti N_t$ in $(X,\vp_t,\psi_t),$ such that\/ $\ti N_t$ is close to $N_0$ away from $x$ in $X$ and\/ $\ti N_t$ is close to $K_s$ near $x$ in $X,$ identifying $X$ near $x$ with\/ $T_xX\cong\R^7$ near $0,$ where $0<s\approx -\ga\de^{-1} t$ to leading order in~$t$.

Topologically, $\ti N_t$ is the connected sum of\/ $N$ with itself at $x^+,x^-,$ so that\/ $\ti N_t\cong N^+\# N^-$ in case {\rm(a),} and\/ $\ti N_t\cong N\#(\cS^1\t\cS^2)$ in case {\rm(b)}.

No such associative $3$-fold\/ $\ti N_t$ exists in $(X,\vp_t,\psi_t)$ if\/~$\ga\de^{-1} t\ge 0$.

We may determine the canonical flag $f_{\ti N_t}$ of\/ $\ti N_t$ in {\rm\S\ref{ca32}} from that of\/ $N_t$ as follows. Let\/ $(N',[s'])$ be a flagged submanifold in $X$ with\/ $[N']=[N_0]$ in $H_3(X;\Z),$ such that $N'$ is disjoint from $N_0,$ and hence from $N_t$ and\/ $\ti N_t$ for small\/ $t$. Then in the notation of\/ {\rm\S\ref{ca31}} we have
\e
D((N,[s']),(\ti N_t,f_{\ti N_t}))=D((N,[s']),(N_t,f_{N_t}))+\begin{cases} 0, & \de<0, \\
1, & \de>0.	\end{cases}
\label{ca4eq6}
\e

If we fix a flag structure on $X,$ so that\/ {\rm\S\ref{ca32}} defines orientations $\Or(N)$ of compact, unobstructed associative $3$-folds $N,$ then \eq{ca4eq6} implies that
\begin{equation*}
\Or(\ti N_t)=\begin{cases} \Or(N_0), & \de<0, \\ -\Or(N_0), & \de>0. \end{cases}
\end{equation*}
\label{ca4conj1}	
\end{conj}

\begin{rem}{\bf(a)} Here is why we require $s\approx -\ga\de^{-1} t$ in this conjecture. To define an associative 3-fold $\ti N_t$ in $(X,\vp_t,\psi_t)$ by gluing $K_s$ for small $s>0$ into $N_t$ for small $t$ near $x$, then $\ti N_t$ should look like the graph of $s\chi+t\frac{\d}{\d t}i_t\vert_{t=0}$ near $i_0(N)\sm\{x\}$ to leading order in $s,t$. The distance between the two ends of this graph in the $\R$-component in \eq{ca4eq4} is $s\de+t\ga$, by definition of $\ga,\de$ in Definition \ref{ca4def1}. As the two ends of the graph must match up, we require that $s\de+t\ga=0$, to leading order in $s,t$. Since $K_s$ is only defined if $s>0$, we expect that no such $\ti N_t$ exists if~$\ga\de^{-1} t\ge 0$.
\smallskip

\noindent{\bf(b)} Equation \eq{ca4eq6} is a guess, but here is some justification for it. The author expects that the eigenvalues (in any bounded region) and eigenvectors of $\bD_{\ti N_t}$ for small $t$ will be close to those of $\bD_{N_0}$, {\it except\/} that $\bD_{\ti N_t}$ should have one additional eigenvector $\xi_t$, with small eigenvalue $\la_t$, where we expect $\xi_t\cong \chi$ away from $x$, and $\xi_t\cong\frac{\d}{\d s}K_s$ near $K_s$, with $s\approx -\ga\de^{-1} t$.

We can estimate this eigenvalue $\la_t$ by
\begin{align*}
\la_t&=\nm{\xi_t}_{L^2}^{-2}\cdot\langle \xi_t,\bD_{\ti N_t}\xi_t\rangle_{L^2}\approx (C\md{\ga}^{-1/6}\md{\de}^{1/6}t^{-1/6})^{-2}\cdot\langle \chi,\bD_{N_0}\chi\rangle_{L^2}\\
&=C^{-2}\md{\ga}^{1/3}\md{\de}^{-1/3}t^{1/3}\cdot\langle \chi,4\pi\,\de_{x^+}\cdot v-4\pi\,\de_{x^-}\cdot v\rangle_{L^2}\\
&=4\pi C^{-2}\md{\ga}^{1/3}\md{\de}^{-1/3}t^{1/3}\cdot\bigl(\ts\lim_{x\ra x^+}v\cdot\chi(x)-\lim_{x\ra x^-}v\cdot\chi(x)\bigr)\\
&=4\pi C^{-2}\md{\ga}^{1/3}\md{\de}^{-1/3}t^{1/3}\de.
\end{align*}
Here in the first step we expect $\nm{\xi_t}_{L^2}$ to be dominated by $\nm{\frac{\d}{\d s}K_s}_{L^2}=Cs^{-1/6}$ for $C>0$ and $s\approx -\ga\de^{-1} t$, and $\langle \xi_t,\bD_{\ti N_t}\xi_t\rangle_{L^2}$ to be dominated by $\langle \chi,\bD_{N_0}\chi\rangle_{L^2}$. Hence we expect $\bD_{\ti N_t}$ to have one small eigenvalue $\la=O(t^{1/3})$, which is positive if $\de>0$ and negative if $\de<0$. So by properties of spectral flow, the canonical flag $f_{\ti N_t}$ of $\ti N_t$ should increase by 1 as $\de$ increases through 0, and this is the reason for the last term in \eq{ca4eq6}.
\smallskip

\noindent{\bf(c)} Motivated by a talk on earlier version of these conjectures given by the author in a conference in London in 2012, Nordstr\"om \cite{Nord} proved part of Conjecture \ref{ca4conj1}. He shows that for for small $s>0$ there exists a associative 3-fold $\ti N_t$ in $(X,\vp_t,\psi_t)$ by gluing $K_s$ into $N_0$ for some unique small $t\in(-\ep,\ep)$, but he does not prove that $s\approx -\ga\de^{-1} t$. A related conjecture for SL 3-folds was stated in \cite[\S 6]{Joyc4} and proved in \cite[\S 9]{Joyc18}, and also independently by Yng-Ing Lee \cite{LeeY}, and by Dan Lee~\cite{LeeD}. 
\label{ca4rem1}
\end{rem}

\section{Another index 1 associative singularity}
\label{ca5}

Next we describe a second kind of singularity of associative 3-folds.

\subsection{\texorpdfstring{Three families of SL 3-folds in $\C^3$}{Three families of SL 3-folds in ℂ³}}
\label{ca51}

Let $G$ be the group $\U(1)^2$, acting on $\C^3$ by
\e
({\rm e}^{i\th_1},{\rm e}^{i\th_2}):(z_1,z_2,z_3)\mapsto
({\rm e}^{i\th_1}z_1,{\rm e}^{i\th_2}z_2,{\rm e}^{-i\th_1-i\th_2}z_3)\quad\text{for $\th_1,\th_2\in\R$.}
\label{ca5eq1}
\e
All the $G$-invariant special Lagrangian 3-folds in $\C^3$ were
written down explicitly by Harvey and Lawson \cite[\S III.3.A]{HaLa}, and studied in more detail in \cite[Ex.~5.1]{Joyc5} and
\cite[\S 4]{Joyc8}. Here are some examples of $G$-invariant SL
3-folds which will be important in what follows.

\begin{dfn} Define a subset $L_0$ in $\C^3$ by
\e
\begin{split}
L_0=\bigl\{(z_1,z_2&,z_3)\in\C^3:\ms{z_1}=\ms{z_2}=\ms{z_3},\\
&\Im(z_1z_2z_3)=0,\quad \Re(z_1z_2z_3)\ge 0\bigr\}.
\end{split}
\label{ca5eq2}
\e
Then $L_0$ is a {\it special Lagrangian cone} on $T^2$, invariant
under the Lie subgroup $G$ of $\SU(3)$ given in \eq{ca5eq1}. Let
$s>0$, and define
\ea
\begin{split}
L^1_s=\bigl\{(z_1,z_2,z_3)\in\C^3:\,&
\ms{z_1}-s=\ms{z_2}=\ms{z_3},\\
&\Im(z_1z_2z_3)=0,\; \Re(z_1z_2z_3)\ge 0\bigr\},
\end{split}
\label{ca5eq3}\\
\begin{split}
L^2_s=\bigl\{(z_1,z_2,z_3)\in\C^3:\,&
\ms{z_1}=\ms{z_2}-s=\ms{z_3},\\
&\Im(z_1z_2z_3)=0,\; \Re(z_1z_2z_3)\ge 0\bigr\},
\end{split}
\label{ca5eq4}\\
\begin{split}
L^3_s=\bigl\{(z_1,z_2,z_3)\in\C^3:\,&
\ms{z_1}=\ms{z_2}=\ms{z_3}-s,\\
&\Im(z_1z_2z_3)=0,\; \Re(z_1z_2z_3)\ge 0\bigr\}.
\end{split}
\label{ca5eq5}
\ea
Then each $L^a_s$ is a $G$-invariant, nonsingular, embedded SL 3-fold in $\C^3$ diffeomorphic to ${\mathcal S}^1\t\R^2$, which is {\it Asymptotically Conical} ({\it AC\/}), with cone~$L_0$.
\label{ca5def1}
\end{dfn}

Thus the $L^a_s$ for $a=1,2,3$ are {\it three different\/} families
of AC SL 3-folds in $\C^3$ asymptotic to the same SL cone $L_0$,
each family depending on $s\in(0,\iy)$. Hence $\{0\}\t L^a_s$ is a nonsingular AC associative 3-fold in $\R^7=\R\t\C^3$ as in \S\ref{ca22}, diffeomorphic to ${\mathcal S}^1\t\R^2$ for $a=1,2,3$ and $s>0$, asymptotic to the singular associative $T^2$-cone $\{0\}\t L_0$. For brevity we write $L_0,L^a_s$ in place of~$\{0\}\t L_0,\{0\}\t L^a_s$.

Write $\nu_{L_0}$ for the normal bundle of $L_0$ in $\R^7$, and $\bD_{L_0}:\Ga^\iy(\nu_{L_0})\ra\Ga^\iy(\nu_{L_0})$ for the operator in Theorem \ref{ca2thm2}. Define sections $\ze_1,\ze_2$ of $\nu_{L_0}$ by
\e
\begin{split}
\ze_1&:(0,z_1,z_2,z_3)\longmapsto \ts(0,\frac{1}{3}\bar z^{-1}_1,-\frac{1}{6}\bar z^{-1}_2,-\frac{1}{6}\bar z^{-1}_3),\\	
\ze_2&:(0,z_1,z_2,z_3)\longmapsto \ts(0,-\frac{1}{6}\bar z^{-1}_1,\frac{1}{3}\bar z^{-1}_2,-\frac{1}{6}\bar z^{-1}_3).
\end{split}
\label{ca5eq6}
\e
Then $\ze_1,\ze_2$ are homogeneous $O(r^{-1})$ with $\bD_{L_0}(\ze_1)=\bD_{L_0}(\ze_2)=0$. A similar analysis to \eq{ca4eq3} shows that
\e
\begin{split}
L^1_s&\approx \Ga_{s\ze_1}+O(s^2r^{-2}),\qquad\quad L^2_s\approx \Ga_{s\ze_2}+O(s^2r^{-2}),\quad\text{and}\\
L^3_s&\approx \Ga_{-s\ze_1-s\ze_2}+O(s^2r^{-2})
\quad\text{near $L_0\sm\{0\}$ in $\R^7$ for small $s>0$.}
\end{split}
\label{ca5eq7}
\e

\subsection{\texorpdfstring{Associative 3-folds with singularities modelled on $L_0$}{Associative 3-folds with singularities modelled on L₀}}
\label{ca52}

The next definition sets up notation for our conjecture.

\begin{dfn} Let $(X,\vp_t,\psi_t)$ for $t\in(-\ep,\ep)$ be a smooth family of TA-$G_2$-manifolds, and that $N_0$ a compact associative 3-fold in $(X,\vp_0,\psi_0)$ with one singular point $x$, locally modelled on $L_0$ (or $\{0\}\t L_0$) in $\R^7=\R\t\C^3$, under an identification $T_xX\cong\R^7$. Write $\nu$ for the normal bundle of $N_0\sm\{x\}$ in $X$, and $\bD:\Ga^\iy(\nu)\ra\Ga^\iy(\nu)$ for the operator in Theorem~\ref{ca2thm2}.
  
The author \cite{Joyc14,Joyc15,Joyc16,Joyc17,Joyc18} studied SL $m$-folds with isolated conical singularities in (almost) Calabi--Yau $m$-folds, and very similar techniques should work to study singular associative 3-folds of this type. To do the analysis, we should work in {\it weighted Sobolev spaces\/} $L^2_{k,\la}(\nu)$ in the sense of Lockhart and McOwen \cite{Lock,LoMc}, where $\la\in\R$ is a growth rate, so that roughly $L^2_{k,\la}(\nu)$ contains sections of $\nu$ on $N_0\sm\{x\}$ which grow at rate $O(r^\la)$ near $x$ in $N_0$, where $r$ is the distance to $x$. 

Then $\bD$ extends to an operator on weighted Sobolev spaces
\e
\bD_{k,\la}:L^2_{k+1,\la}(\nu)\longra L^2_{k,\la-1}(\nu).
\label{ca5eq8}
\e
Write $\nu_{L_0}$ for the normal bundle of $L_0\sm\{0\}$ in $\R^7$, and $\bD_{L_0}:\Ga^\iy(\nu_{L_0})\ra\Ga^\iy(\nu_{L_0})$ for the corresponding twisted Dirac operator. For each $\la\in\R$, define
\begin{align*}
V_\la=\big\{s\in\Ga^\iy(\nu_{L_0}):\text{$\bD_{L_0}(s)=0$ and $s$ is homogeneous of order}&\\
\text{$O(r^\la)$ under dilations of $L_0$}&\bigr\}.
\end{align*}
Then $V_\la$ is finite-dimensional, isomorphic to the kernel of an elliptic operator on the link $T^2$ of $L_0$. Write $\cD_{L_0}\subset\R$ for the set of $\la$ with $V_\la\ne 0$. Then $\cD_{L_0}$ is discrete. The Lockhart--McOwen theory implies that $\bD_{k,\la}$ in \eq{ca5eq8} is Fredholm if and only if $\la\in\R\sm\cD_{L_0}$, where the index, kernel and cokernel of $\bD_{k,\la}$ are independent of $k\in\N,$ and if $\la_1,\la_2\in \R\sm\cD_{L_0}$ with $\la_1<\la_2$ then
\e
\ind(\bD_{k,\la_1})=\ind(\bD_{k,\la_2})+\sum_{\la\in \cD_{L_0}:\la_1<\mu<\la_2}\dim V_\mu.
\label{ca5eq9}
\e
 
Since $\bD$ is self-adjoint of order 1 and $\dim N_0=3$, it turns out that $\cD_{L_0}$ and the $V_\la$ are symmetric about $\la=-1$ in $\R$, with for all $k,l\in\N$ and $\la\in\R\sm\cD_{L_0}$
\e
\begin{gathered}
\Ker(\bD_{k,\la})\cong \Coker(\bD_{l,-2-\la}), \quad \Coker(\bD_{k,\la})\cong \Ker(\bD_{l,-2-\la}),\\ 
\ind(\bD_{k,\la})=-\ind(\bD_{l,-2-\la}).
\end{gathered}
\label{ca5eq10}
\e
Combining \eq{ca5eq9}--\eq{ca5eq10} yields for $\la\in\R\sm\cD_{L_0}$
\begin{equation*}
\ind(\bD_{k,\la})=\begin{cases} \phantom{-}\ha \dim V_{-1}+\sum_{\mu\in \cD_{L_0}:\la<\mu<-1}\dim V_\mu, & \la<-1, \\
-\ha \dim V_{-1}-\sum_{\mu\in\cD_{L_0}:-1<\mu<\la}\dim V_\mu, & \la>-1.	
\end{cases}
\end{equation*}

The cone $L_0$ was studied as a special Lagrangian cone in \cite[Ex.~3.5]{Joyc18}. A similar analysis should show that $V_{-1}=\langle\ze_1,\ze_2\rangle\cong\R^2$, for $\ze_1,\ze_2$ as in \eq{ca5eq6}, and that $V_\la=0$ for $\la\in(-1,0)$, and $V_0\cong\R^7$ is the normal projections of translation vector fields in $\R^7$. Therefore $\bD_{k,\la}$ is Fredholm with index $-1$ for $\la\in(-1,0)$, and Fredholm with index 1 for $\la\in(-2,-1)$.

Let us assume that $\Coker(\bD_{k,\la})=0$ for $\la\in(-2,-1)$, that is, $N_0$ is minimally obstructed. This should hold provided $(X,\vp_t,\psi_t)$, $t\in(-\ep,\ep)$ is a generic 1-parameter family of TA-$G_2$-manifolds. Then $\dim\Ker(\bD_{k,\la})=1$, so $\Ker(\bD_{k,\la})=\langle\chi\rangle$ for $\la\in(-2,-1)$, say. So \eq{ca5eq10} gives $\Coker(\bD_{k,\la})\cong\langle\chi\rangle$ for $\la\in(-1,0)$. From the Lockhart--McOwen theory with $V_{-1}=\langle\ze_1,\ze_2\rangle$ and $V_\la=0$ for $\la\in(-1,0)$ we can show that identifying $N_0,\nu$ near $x\in X$ with $L_0,\nu_{L_0}$ near $0\in\R^7$ we have
\begin{equation*}
\chi=\de_1\ze_1+\de_2\ze_2+O(1)
\end{equation*}
for some $\de_1,\de_2\in\R$, not both zero. 

To simplify Conjecture \ref{ca5conj1} we assume that $\de_2\ne 0$, $\de_1\ne 0$, and $\de_1\ne\de_2$. These are the respective first-order obstructions to gluing $L_s^1,L_s^2,L_s^3$ for small $s>0$ into $N_0$ to make an associative 3-fold in the fixed TA-$G_2$-manifold $(X,\vp_0,\psi_0)$, rather than in $(X,\vp_t,\psi_t)$ for some $t$. This should hold provided $(X,\vp_t,\psi_t)$, $t\in(-\ep,\ep)$ is a generic 1-parameter family of TA-$G_2$-manifolds.

Define a section $\xi\in\Ga^\iy(\nu)$ by, using the index notation for tensors,
\begin{equation*}
\xi^{a_4}=\ts g^{a_1b_1}g^{a_2b_2}g^{a_3b_3}g^{a_4b_4}(\vol_{N_0})_{a_1a_2a_3}\bigl(\frac{\d}{\d t}\psi_t\vert_{t=0}\bigr){}_{b_1b_2b_3b_4},
\end{equation*}
where $g$ is the Riemannian metric on $X$ associated to $\psi$, and $\vol_{N_0}$ the volume form on $N_0$ induced by $g$. Define $\ga\in\R$ by $\ga=\langle\chi,\xi\rangle_{L^2}$.

Now if $N_0$ extended to a smooth family $N_t$, $t\in(-\ep,\ep)$ of compact associative 3-folds in $(X,\vp_t,\psi_t)$ with conical singularities, then $\th=\frac{\d}{\d t}N_t\vert_{t=0}$ would be a section of $\nu$ with $\th=O(r^0)$ and $\bD(\th)=\xi$. As $\Coker(\bD_{k,\la})=\langle\chi\rangle$ for $\la\in(-1,0)$, this would give $\ga=\langle \chi,\xi\rangle_{L^2}=\langle\chi,\bD(\th)\rangle_{L^2}=0$. Thus $\ga$ is the first-order obstruction to extending $N_0$ to a family $N_t$, $t\in(-\ep,\ep)$ in $(X,\vp_t,\psi_t)$. 

To simplify Conjecture \ref{ca5conj1} we assume that $\ga\ne 0$. This should hold provided $(X,\vp_t,\psi_t)$, $t\in(-\ep,\ep)$ is a generic 1-parameter family of TA-$G_2$-manifolds.
\label{ca5def2}	
\end{dfn}

\begin{conj} Work in the situation of Definition\/ {\rm\ref{ca5def2}}. Then:
\begin{itemize}
\setlength{\itemsep}{0pt}
\setlength{\parsep}{0pt}
\item[{\bf(i)}] For all small\/ $t\in(-\ep,\ep)$ with\/ $\ga\de_2^{-1}t<0$ there exists a unique compact, embedded, unobstructed associative $3$-fold\/ $\ti N^1_t$ in $(X,\vp_t,\psi_t),$ such that\/ $\ti N^1_t$ is close to $N_0$ away from $x$ in $X$ and\/ $\ti N^1_t$ is close to $L^1_s$ near $x$ in $X,$ identifying $X$ near $x$ with\/ $T_xX\cong\R^7$ near $0,$ where $0<s\approx -\ga\de_2^{-1} t$ to leading order in $t$. No such\/ $\ti N^1_t$ exists in $(X,\vp_t,\psi_t)$ if\/~$\ga\de_2^{-1} t\ge 0$.

\item[{\bf(ii)}] For all small\/ $t\in(-\ep,\ep)$ with\/ $\ga\de_1^{-1} t>0$ there exists a unique compact, embedded, unobstructed associative $3$-fold\/ $\ti N^2_t$ in $(X,\vp_t,\psi_t),$ such that\/ $\ti N^2_t$ is close to $N_0$ away from $x$ in $X$ and\/ $\ti N^2_t$ is close to $L^2_s$ near $x$ in $X,$ identifying $X$ near $x$ with\/ $T_xX\cong\R^7$ near $0,$ where $0<s\approx \ga\de_1^{-1} t$ to leading order in $t$. No such\/ $\ti N^2_t$ exists in $(X,\vp_t,\psi_t)$ if\/~$\ga\de_1^{-1} t\le 0$.
\item[{\bf(iii)}] For all small\/ $t\in(-\ep,\ep)$ with\/ $\ga(\de_2-\de_1)^{-1}t>0$ there exists a unique compact, embedded, unobstructed associative $3$-fold\/ $\ti N^3_t$ in $(X,\vp_t,\psi_t),$ such that\/ $\ti N^3_t$ is close to $N_0$ away from $x$ in $X$ and\/ $\ti N^3_t$ is close to $L^2_s$ near $x$ in $X,$ identifying $X$ near $x$ with\/ $T_xX\cong\R^7$ near $0,$ where $0<s\approx \ga(\de_2-\de_1)^{-1}t$ to leading order in $t$. No such\/ $\ti N^3_t$ exists in $(X,\vp_t,\psi_t)$ if\/~$\ga(\de_2-\de_1)^{-1}t\le 0$.

\end{itemize}
\label{ca5conj1}	
\end{conj}

We discuss canonical flags and orientations of the $\ti N^a_t$ in Conjecture~\ref{ca5conj2}.

\begin{rem}{\bf(a)} Here is why we expect $s\approx \ga\de_1^{-1} t>0$ in part (i). Suppose we have an associative 3-fold $\ti N^1_t$ in $(X,\vp_t,\psi_t)$ modelled on $N_0$ away from $x$ in $X$ and on $L^1_s$ near $x$ in $X$, for small $s>0$ and $t\in(-\ep,\ep)$. Then near $N_0$ we can write $\ti N^1_t\approx\Ga_\th$ for $\th\in\Ga^\iy(\nu)$. As $\ti N^1_t$ is associative we must have $\bD(\th)=t\xi+O(t^2)$. Since $\ti N^1_t$ approximates $L_s^1$ near $x$, from \eq{ca5eq7} we see that $\th\approx s\ze_1+O(1)$. We now show that
\e
\begin{split}
t\ga&=\langle\chi,t\xi\rangle_{L^2}-0=\langle\chi,\bD(\th)\rangle_{L^2}-\langle \bD(\chi),\th\rangle_{L^2}+O(t^2)\\
&=(\de_1\ze_1+\de_2\ze_2)\w(s\ze_1)+O(t^2)=-\de_2s+O(t^2).
\end{split}
\label{ca5eq11}
\e

Here one might expect that $\langle \chi,\bD(\th)\rangle_{L^2}=\langle \bD(\chi),\th\rangle_{L^2}$, as $\bD$ is self-adjoint. However, as $\chi,\th=O(r^{-1})$ and $\nabla\chi,\nabla\th=O(r^{-2})$, so that the $L^2$-inner products between $\chi,\th$ and $\nabla\chi,\nabla\th$ are not defined, it turns out that 
\begin{equation*}
\langle\chi,\bD(\th)\rangle_{L^2}-\langle \bD(\chi),\th\rangle_{L^2}=\text{boundary term,}
\end{equation*}
where the boundary term is obtained by completing $N_0\sm\{x\}$ to a compact manifold $\bar N_0$ with boundary $\pd\bar N_0=T^2$, and using Stokes' Theorem. 

The boundary term depends only on the leading terms $\chi=\de_1\ze_1+\de_2\ze_2+\cdots$, $\th=s\ze_1+\cdots$ in $V_{-1}$, and may be written in terms of an antisymmetric bilinear form $\w:V_{-1}\t V_{-1}\ra\R$, as in the third step of \eq{ca5eq11}. Guessing (out of laziness) that this is normalized with $\ze_1\w\ze_2=1$ gives the final step of \eq{ca5eq11}. Thus $t\ga=-\de_2s+O(t^2)$, giving $s\approx -\ga\de_2^{-1} t$, and showing that $\ti N_t^1$ in (i) exists only when $\ga\de_2^{-1}t<0$, as $s>0$. Parts (ii),(iii) are similar, using \eq{ca5eq7} for~$L_s^2,L_s^3$.

\smallskip

\noindent{\bf(b)} A related conjecture for SL 3-folds with singularities modelled on $L_0\subset\C^3$ was stated in \cite[\S 3.2]{Joyc4}, and now follows from work of the author \cite{Joyc14,Joyc15,Joyc16,Joyc17,Joyc18} and Imagi \cite{Imag}. Proving Conjecture \ref{ca5conj1} should not be that difficult, by adapting known technology for special Lagrangians to the associative case.
\label{ca5rem1}
\end{rem}

\subsection{\texorpdfstring{Algebraic topology of desingularizations using $L^a_s$}{Algebraic topology of desingularizations using Lªˢ}}
\label{ca53}

In \cite[\S 4]{Joyc4} the author discussed starting with a compact SL 3-fold $N_0$ with one singular point locally modelled on $L_0\subset\C^3$ in \eq{ca5eq2} in an (almost) Calabi--Yau 3-fold $(Y,J,h)$, and desingularizing $N_0$ by gluing in $L^a_s\subset\C^3$ for $a=1,2,3$ and small $s>0$ from \eq{ca5eq3}--\eq{ca5eq5} to get compact nonsingular SL 3-folds $\ti N^a_s$ in $Y$. In \cite[\S 4.3]{Joyc4} we computed the integral homology groups $H_1(\ti N^a_s;\Z)$ from $H_1(N_0;\Z)$. This is a purely topological calculation, and so applies just as well to smoothing associative 3-folds with singularities modelled on $L_0\subset\R^7$ by gluing in $L^a_s\subset\R^7$, as in \S\ref{ca52}. Thus, from \cite[\S 4.2]{Joyc4} we deduce:

\begin{prop} Work in the situation of Conjecture\/ {\rm\ref{ca5conj1}}. Write\/ $P=N_0\sm B_\ep(x),$ for $B_\ep(x)$ a ball of radius $\ep$ about $x$ in $X$ for $\ep>0$ small. Then $P$ is a compact, nonsingular\/ $3$-manifold with boundary, where $\pd P$ may be identified with\/ $G=T^2$ in {\rm\eq{ca5eq1},} since $\pd(L_0\sm B_\ep(x))$ is a free $G$-orbit. Define $\rho:\Z^2\ra H_1(P;\Z)$ to be the composition of natural morphisms
\begin{equation*}
\xymatrix@C=40pt{ \Z^2 \ar@{=}[r] & H_1(G;\Z) \ar[r]^\cong & H_1(\pd P;\Z) \ar[r]^{{\rm inc}_*} & H_1(P;\Z). }	
\end{equation*}
Then $\Ker(\rho)\cong\Z,$ so $\Ker\rho=\langle(b_1,b_2)\rangle_\Z$ for $(b_1,b_2)\in\Z^2\sm\{0\}$ unique up to sign. Also $H_1(N_0;\Z)$ and\/ $H_1(\ti N^a_t;\Z)$ are determined by the exact sequences
\begin{align*}
\xymatrix@!0@C=50pt{ *+[r]{\Z^2} \ar[rrr]^(0.4)\rho &&& *+[l]{H_1(P;\Z)} \ar[r] & H_1(N_0;\Z) \ar[r] & 0, }\\
\xymatrix@!0@C=50pt{ *+[r]{\Z} \ar[rrr]^(0.4){n\mapsto\rho(n,0)} &&& *+[l]{H_1(P;\Z)} \ar[r] & H_1(\ti N_t^1;\Z) \ar[r] & 0, }\\
\xymatrix@!0@C=50pt{ *+[r]{\Z} \ar[rrr]^(0.4){n\mapsto\rho(0,n)} &&& *+[l]{H_1(P;\Z)} \ar[r] & H_1(\ti N_t^2;\Z) \ar[r] & 0, }\\
\xymatrix@!0@C=50pt{ *+[r]{\Z} \ar[rrr]^(0.4){n\mapsto\rho(-n,-n)} &&& *+[l]{H_1(P;\Z)} \ar[r] & H_1(\ti N_t^3;\Z) \ar[r] & 0. }
\end{align*}
If\/ $H_1(N_0;\Z)$ is infinite then so are $H_1(\ti N^1_t;\Z)$ for $a=1,2,3$.

Suppose now that\/ $H_1(N_0;\Z)$ is finite. Then we have
\begin{align*}
\bmd{H_1(\ti N^1_t;\Z)}&=\begin{cases}\md{b_1}\cdot \bmd{H_1(N_0;\Z)}, & b_1\ne 0, \\ \iy, & b_1=0, \end{cases}\\
\bmd{H_1(\ti N^2_t;\Z)}&=\begin{cases}\md{b_2}\cdot \bmd{H_1(N_0;\Z)}, & b_2\ne 0, \\ \iy, & b_2=0, \end{cases}\\
\bmd{H_1(\ti N^3_t;\Z)}&=\begin{cases}\md{-b_1-b_2}\cdot \bmd{H_1(N_0;\Z)}, & -b_1-b_2\ne 0, \\ \iy, & -b_1-b_2=0. \end{cases}
\end{align*}
Hence if we define an invariant\/ $I$ of compact\/ $3$-manifolds $N$ by
\e
I(N)=\begin{cases} \bmd{H_1(N;\Z)}, & \text{$H_1(N;\Z)$ is finite,} \\
0, & \text{otherwise,} \end{cases}
\label{ca5eq12}
\e
then in all cases in Conjecture\/ {\rm\ref{ca5conj1}} we have
\e
\sign(b_1)\cdot I(\ti N^1_t)+\sign(b_2)\cdot I(\ti N^2_t)+\sign(-b_1-b_2)\cdot I(\ti N^3_t)=0.
\label{ca5eq13}
\e

Note too that for all compact\/ $3$-manifolds $N_1,N_2$ we have
\e
I(N_1\# N_2)=I(N_1)\cdot I(N_2).
\label{ca5eq14}
\e

\label{ca5prop1}	
\end{prop}

\begin{conj} In the situation of Conjecture\/ {\rm\ref{ca5conj1},} there is some formula relating the canonical flags of\/ $\ti N^1_t,\ti N^2_t,\ti N^3_t,$ depending on $\ga,\de_1,\de_2,b_1,b_2$. If we choose a flag structure on $X$ then the corresponding orientations satisfy
\e
\sum_{\text{$a=1,2,3:$ $\ti N^a_t$ exists when $t<0$}\!\!\!\!\!\!\!\!\!\!\!\!\!\!\!\!\!\!\!\!\!\!\!\!\!\!\!\!\!\!\!\!\!\!\!\!\!\!\!\!\!\!\!\!\!} \Or(\ti N^a_t)\cdot I(\ti N^a_t)=\sum_{\text{$a=1,2,3:$ $\ti N^a_t$ exists when $t>0$}\!\!\!\!\!\!\!\!\!\!\!\!\!\!\!\!\!\!\!\!\!\!\!\!\!\!\!\!\!\!\!\!\!\!\!\!\!\!\!\!\!\!\!\!\!} \Or(\ti N^a_t)\cdot I(\ti N^a_t).
\label{ca5eq15}
\e

\label{ca5conj2}	
\end{conj}

Observe that Conjecture \ref{ca5conj2} is plausible by \eq{ca5eq13}, as there are always at least two choices of signs $\Or(\ti N^1_t),\Or(\ti N^2_t),\Or(\ti N^3_t)$ for which \eq{ca5eq15} holds. The point of \eq{ca5eq15} is that as we cross the `wall' $t=0$ in the family of TA-$G_2$-manifolds $(X,\vp_t,\psi_t)$, the signed weighted count of associative 3-folds does not change. In \cite{Joyc4} the author made a similar proposal to define invariants of (almost) Calabi--Yau 3-folds by counting SL 3-folds $N$ weighted by $I(N)$ in~\eq{ca5eq12}.

\begin{rem}{\bf(a)} Let $N$ be a compact oriented 3-manifold. If $b^1(N)=0$ then the moduli space $\cM_N^{\U(1)}$ of flat $\U(1)$-connections on $N$ is finite, and is $\md{H_1(N;\Z)}$ points. If $b^1(N)>0$ then $\cM_N^{\U(1)}$ is a finite number of copies of $T^{b^1(N)}$, so $\chi(\cM_N^{\U(1)})=0$. In both cases, $\chi(\cM_N^{\U(1)})=I(N)$ in~\eq{ca5eq12}.

In \S\ref{ca7} we propose counting associative 3-folds $N$ in $(X,\vp,\psi)$, with signs, weighted by $I(N)$. Thus, we can interpret this as counting associative 3-folds with flat $\U(1)$-connections. This may have an interpretation in String Theory or M-theory, as counting some kind of brane, such as D3-branes in Type IIB String Theory on the $G_2$-manifold, or M2-brane instantons in M-theory.
\smallskip

\noindent{\bf(b)} The programme of \S\ref{ca7} would work using any invariant $I$ of compact oriented 3-manifolds satisfying \eq{ca5eq13}--\eq{ca5eq14}, and such that $I(N)=0$ if $b^1(N)>0$. The author expects that $I$ in \eq{ca5eq12} is the unique such invariant.  
\label{ca5rem2}	
\end{rem}

\section{\texorpdfstring{$\U(1)$-invariant associative 3-folds in $\R^7$}{U(1)-invariant associative 3-folds in ℝ⁷}}
\label{ca6}

Next we discuss a class of $\U(1)$-invariant associative 3-folds in $\R^7$ which should be amenable to study using analytic techniques, and will provide a large class of examples of singularities of associative 3-folds. Understanding the behaviour of these singularities may help guide any programme for defining invariants by counting associative 3-folds. This class is closely related to the author's papers \cite{Joyc10,Joyc11,Joyc12,Joyc13} on $\U(1)$-invariant SL 3-folds in~$\C^3$.

\subsection{\texorpdfstring{Associative 3-folds and $J$-holomorphic curves}{Associative 3-folds and J-holomorphic curves}}
\label{ca61}

We will study associative 3-folds $N$ in $\R^7$ invariant under the $\U(1)$-action
\e
\begin{split}
e^{i\th}:(x_1,\ldots,x_7)\longmapsto(x_1,x_2,x_3,\cos\th\, x_4-\sin\th\, x_5,\sin\th\, x_4+\cos\th\, x_5&,\\
\cos\th\, x_6+\sin\th\, x_7,-\sin\th\, x_6+\cos\th\, x_7&).
\end{split}
\label{ca6eq1}
\e
This preserves $g_0,\vp_0,*\vp_0$ on $\R^7$ from \S\ref{ca21}. The $\U(1)$-action fixes the associative 3-plane $\R^3=\bigl\{(x_1,x_2,x_3,0,0,0,0):x_j\in\R\bigr\}$ in $\R^7$. 

Define $\U(1)$-invariant quadratic polynomials $y_1,y_2,y_3$ on $\R^7$ by
\begin{align*}
y_1(x_1,\ldots,x_7)&=x_4^2+x_5^2-x_6^2-x_7^2,\\
y_2(x_1,\ldots,x_7)&=2(x_4x_7+x_5x_6),\\
y_3(x_1,\ldots,x_7)&=2(x_4x_6-x_5x_7).
\end{align*}
Then $y_1^2+y_2^2+y_3^2=(x_4^2+x_5^2+x_6^2+x_7^2)^2$. Consider the map
\begin{equation*}
\Pi=(x_1,x_2,x_3,y_1,y_2,y_3):\R^7\longra\R^6.
\end{equation*}
This is $\U(1)$-invariant, and its fibres are exactly the $\U(1)$-orbits in $\R^7$. Hence it descends to a homeomorphism $\Pi:\R^7/\U(1)\ra\R^6$. The $\U(1)$-fixed locus $\R^3\subset\R^7$ maps to the 3-plane $L=\R^3=\bigl\{(x_1,x_2,x_3,0,0,0):x_j\in\R\bigr\}$ in~$\R^6$.

Note that we should {\it not\/} think of $\R^7/\U(1)$ as a smooth manifold near the fixed locus $\R^3\subset\R^7$. The identification $\R^7/\U(1)\cong\R^6$ is only topological, not smooth, near $\R^3$, and we should expect singular behaviour near $\R^3\subset\R^6$.

The next proposition relates $\U(1)$-invariant associative 3-folds $N$ in $\R^7\sm\R^3$ to $J$-holomorphic curves $\Si$ in $\R^6\sm\R^3$, for a certain almost complex structure $J$ on $\R^6\sm\R^3$. It is similar to~\cite[Prop.~4.1]{Joyc10}.

\begin{prop} Let\/ $\R^6$ have coordinates $(x_1,x_2,x_3,y_1,y_2,y_3),$ and write $L=\R^3=\bigl\{(x_1,x_2,x_3,0,0,0):x_j\in\R\bigr\}\subset\R^6$. Define $u:\R^6\ra[0,\iy)$ by $u(x_1,x_2,x_3,y_1,y_2,y_3)=(y_1^2+y_2^2+y_3^2)^{1/2}$. Define an almost complex structure $J$ on $\R^6\sm\R^3$ to have matrix
\e
J=\begin{pmatrix} 0 & 0 & 0 & -\ha u^{-1/2} & 0 & 0 \\
0 & 0 & 0 & 0 & -\ha u^{-1/2} & 0 \\
0 & 0 & 0 & 0 & 0 & -\ha u^{-1/2} \\
2u^{1/2} & 0 & 0 & 0 & 0 & 0 \\
0 & 2u^{1/2} & 0 & 0 & 0 & 0 \\
0 & 0 & 2u^{1/2} & 0 & 0 & 0 
\end{pmatrix}
\label{ca6eq2}
\e
with respect to the basis of sections $\frac{\pd}{\pd x_1},\frac{\pd}{\pd x_2},\frac{\pd}{\pd x_3},\frac{\pd}{\pd y_1},\frac{\pd}{\pd y_2},\frac{\pd}{\pd y_3}$ of\/ $T(\R^6\sm\R^3)$.

Suppose $N$ is a $\U(1)$-invariant\/ $3$-submanifold in $\R^7\sm\R^3,$ so that\/ $\Si=N/\U(1)$ is a $2$-submanifold in $\R^6\sm\R^3\cong (\R^7\sm\R^3)/\U(1)$. Then $N$ is an associative $3$-fold in $\R^7\sm\R^3$ if and only if\/ $\Si$ is a $J$-holomorphic curve in $\R^6\sm\R^3$.
\label{ca6prop1}	
\end{prop}

Note that $J$ in \eq{ca6eq2} becomes singular when $u=0$, that is, on~$L=\R^3\subset\R^6$.

\begin{ex} Let $N$ be the associative 3-plane $\bigl\{(x_1,0,0,x_4,x_5,0,0):x_j\in\R\bigr\}$ in $\R^7$. Then $N$ is $\U(1)$-invariant, and $\Si=N/\U(1)$ is the half-plane
\begin{equation*}
\Si=\bigl\{(x_1,0,0,y_1,0,0):x_1\in\R,\;\> y_1\in[0,\iy)\bigr\}\cong \R\t[0,\iy),
\end{equation*}
which has boundary $\pd\Si\subset L\subset\R^6$. 
\label{ca6ex1}	
\end{ex}

This example illustrates the general principle that $J$-{\it holomorphic curves $\Si$ in $\R^6$ with boundary $\pd\Si$ in $L\subset\R^6$ lift to associative $3$-folds $N=\Pi^{-1}(\Si)$ without boundary in\/} $\R^7$. Note that $J$ is singular along $L$. One moral is that we should expect any theory `counting' associative 3-folds $N$ in a TA-$G_2$-manifold $(X,\vp,\psi)$ to look more like Lagrangian Floer cohomology \cite{Fuka2,FOOO} (built on counting $J$-holomorphic curves $\Si$ with boundary in $L$) than like Gromov--Witten theory \cite{FuOn,HWZ,McSa} (built on counting $J$-holomorphic curves $\Si$ without boundary).

Identify $\R^6$ with $\C^3$ with complex coordinates $(x_1+iy_1,x_2+iy_2,x_3+iy_3)$. This corresponds to the complex structure $J_0$, with matrix
\begin{equation*}
J_0=\begin{pmatrix} 0 & 0 & 0 & -1 & 0 & 0 \\
0 & 0 & 0 & 0 & -1 & 0 \\
0 & 0 & 0 & 0 & 0 & -1 \\
1 & 0 & 0 & 0 & 0 & 0 \\
0 & 1 & 0 & 0 & 0 & 0 \\
0 & 0 & 1 & 0 & 0 & 0 
\end{pmatrix}
\end{equation*}
with respect to the basis $\frac{\pd}{\pd x_1},\frac{\pd}{\pd x_2},\frac{\pd}{\pd x_3},\frac{\pd}{\pd y_1},\frac{\pd}{\pd y_2},\frac{\pd}{\pd y_3}$, so that $J$ in \eq{ca6eq2} becomes $J_0$ if we replace $2u^{1/2}$ by 1. This $J_0$ is compatible with the standard symplectic structure $\om_0=\d x_1\w\d y_1+\d x_2\w\d y_2+\d x_3\w\d y_3$ on $\R^6$, for which $L$ is a Lagrangian submanifold. The next conjecture is not very precise:

\begin{conj} $J$-holomorphic curves in $\R^6$ (with boundary in $L$) have essentially the same qualitative behaviour as ordinary $J_0$-holomorphic curves in $\R^6=\C^3$ (with boundary in $L$), which are already very well understood.
\label{ca6conj1}	
\end{conj}

In \cite{Joyc10,Joyc11,Joyc12,Joyc13} the author studied $\U(1)$-invariant SL 3-folds in $\C^3$, in terms of solutions of a singular nonlinear Cauchy--Riemann equation. These correspond to studying $J$-holomorphic curves in the $\R^6$ above lying in the $\R^4\subset\R^6$ defined by $x_1=0$, $y_1=a$. One moral of \cite{Joyc10,Joyc11,Joyc12,Joyc13} is that the singular nonlinear Cauchy--Riemann equation behaves exactly like the usual Cauchy--Riemann equation, for questions such as existence and uniqueness of solutions with prescribed boundary data. The author expects a similar picture for this more general class.

If we accept Conjecture \ref{ca6conj1} then we can give heuristic descriptions of a large class of singularities of associative 3-folds: every kind of singularity of $J_0$-holomorphic curves in $\C^3$, possibly with boundary in a Lagrangian $L$, should correspond to a kind of singularity of associative 3-folds. 

Both the associative singularities in \S\ref{ca4}--\S\ref{ca5} can be made invariant under \eq{ca5eq1}, and so interpreted in this framework, as the next two examples show.

\begin{ex} Consider the $J_0$-holomorphic curves $\Si^+_t,\Si_t^-$ and $\ti\Si_s$ in $\R^6$ with boundary in $L$, for $s\ge 0$ and $t\in\R$:
\begin{align*}
\Si^+_t&=\bigl\{(x_1,0,t,y_1,0,0):x_1\in\R,\; y_1\in[0,\iy)\bigr\},\\
\Si^-_t&=\bigl\{(0,x_2,-t,0,y_2,0):x_2\in\R,\; y_2\in[0,\iy)\bigr\},\\
\ti\Si_s&=\bigl\{(x_1,x_2,0,y_1,y_2,0):(x_1+iy_1)(x_2+iy_2)=-s,\;\> y_1,y_2\ge 0\bigr\}.
\end{align*}
Here $\Si^+_t,\Si^-_t$ do not intersect for $t\ne 0$, and when $t=0$ they intersect in one point $(0,\ldots,0)$ in their common boundary. Also $\ti\Si_0=\Si^+_0\cup\Si^-_0$, but $\ti\Si_s$ for $s>0$ is diffeomorphic to $[0,1]\t\R$. Write $N_t^+,N_t^-,\ti N_s$ for the preimages of $\Si^+_t,\Si^-_t,\ti\Si_s$ under $\Pi:\R^7\ra\R^6$. Then $N_t^+,N_t^-$ are affine associative 3-planes $\R^3\subset\R^7$, and $\ti N_s$ for $s>0$ is diffeomorphic to $\cS^2\t\R\cong\R^3\#\R^3$, and is a distorted version of the associative 3-fold $K_s^{V,V'}$ in~\S\ref{ca41}.

This is an approximate local model for the index one singularity of associative 3-folds described in \S\ref{ca4}: we have associative 3-folds $N_t^+,N_t^-$ in $(X,\vp_t,\psi_t)$, which are disjoint for $t\ne 0$, and intersect in one point $\{x\}$ when $t=0$. As $t$ passes through 0 we create a new associative 3-fold $\ti N_s$ diffeomorphic to~$N_t^+\# N_t^-$.
\label{ca6ex2}	
\end{ex}

\begin{ex} Consider the $J_0$-holomorphic curves $\Si_t$ and $\ti\Si_s$ in $\R^6$, where $\ti\Si_s$ has boundary in $L$, for $s\ge 0$ and $t\in\R$:
\begin{align*}
\Si_t&=\bigl\{(x_1,x_2,0,x_2,-x_1,t):x_1\in\R,\; y_1\in[0,\iy)\bigr\},\\
\ti\Si_s&=\bigl\{(x_1,x_2,0,y_1,y_2,0):(x_1+iy_1)^2+(x_2+iy_2)^2=s,\;\> x_2y_1-x_1y_2\ge 0\bigr\}.
\end{align*}
Then $\Si_t\cong\R^2$, which does not intersect $L$ when $t\ne 0$, and intersects $L$ in one point $(0,\ldots,0)$ when $t=0$. Also $\ti\Si_0=\Si_0$, and $\ti\Si_s$ for $s>0$ is diffeomorphic to $\cS^1\t[0,\iy)$, with boundary the circle $\bigl\{(x_1,x_2,0,0,0,0):x_1^2+x_2^2=s\bigr\}$ in~$L$.

Write $N_t,\ti N_s$ for the preimages of $\Si_t,\ti\Si_s$ under $\Pi:\R^7\ra\R^6$. Then $N_0=\ti N_0$ is a $T^2$-cone in $\R^7$, and $N_t,\ti N_s$ for $s,t\ne 0$ are diffeomorphic to $\cS^1\t\R^2$. In fact $N_t$ for $t<0$ and $N_t$ for $t>0$ differ by a Dehn twist around $\cS^1\subset N_t$. So we should regard $N_t$, $t<0$ and $N_t$, $t>0$ and $\ti N_s$, $s>0$ as three different families of 3-manifolds $\cS^1\t\R^2$ desingularizing the $T^2$-cone $N_0=\ti N_0$. These are distorted versions of the associative $T^2$-cone $L_0$ and $\cS^1\t\R^2$'s $L^1_s,L^2_s,L^3_s$ in~\S\ref{ca51}.
\label{ca6ex3}	
\end{ex}

\subsection{Associative 3-folds with boundary in coassociatives}
\label{ca62}

Next we use the ideas of \S\ref{ca61} to discuss associative 3-folds with boundary in a coassociative 4-fold, as in \S\ref{ca27}. Let $C$ be the coassociative 4-plane
\begin{equation*}
C=\bigl\{(0,x_2,x_3,x_4,x_5,0,0):x_j\in\R\bigr\}\subset\R^7,	
\end{equation*}
which is invariant under the $\U(1)$-action \eq{ca6eq1}. Then
\begin{equation*}
M=C/\U(1)=\bigl\{(0,x_2,x_3,y_1,0,0):x_2,x_3\in\R,\;\> y_1\in[0,\iy)\bigr\}\cong [0,\iy)\t\R^2.	
\end{equation*}
We think of $M$ as a Lagrangian half-plane in $\R^6\cong\C^3$ with boundary in $L=\R^3\subset\C^3$. In the language of \S\ref{ca23}, $L$ is special Lagrangian with phase 1, and $M$ is special Lagrangian with phase~$i$.

Suppose now that $N$ is a $\U(1)$-invariant associative 3-fold in $\R^7$ with $\pd N\subset C$. Then $\Si=N/\U(1)$ is a (possibly singular) $J$-holomorphic curve in $\R^6$, which can have boundary $\pd\Si$ of two kinds. As in Example \ref{ca6ex1}, the fixed locus of $\U(1)$ in $N$ (which may lie in the interior $N^\ci$) gives a boundary component $\pd_L\Si$ of $\Si$ in $L$. And $\pd N/\U(1)$ gives a boundary component $\pd_M\Si$ of $\Si$ in $M$. Thus we expect that $\pd\Si=\pd_L\Si\cup\pd_M\Si\subset L\cup M$, where $\Si$ may have codimension 2 corners $\pd_L\Si\cap\pd_M\Si$ mapping to $L\cap M$. Thus we conclude:
\begin{quotation}
\noindent {\it Counting associative $3$-folds $N$ with boundary $\pd N\subset C$ in a coassociative $4$-fold $C$ in a TA-$G_2$-manifold\/ $(X,\vp,\psi),$ is analogous to counting $J$-holomorphic curves $\Si$ in a symplectic manifold\/ $(Y,\om)$ with boundary $\pd\Si\subset L\cup M,$ where $L$ is a Lagrangian in $Y,$ and\/ $M$ is another Lagrangian in $Y$ with boundary $\pd M\subset L$.}
\end{quotation}

The author does not know of any symplectic theory involving counting $J$-holomorphic curves with boundary in $L\cup M$ in this way. If we assume Conjecture \ref{ca6conj1} we can give heuristic models for singularities of $\U(1)$-invariant associative 3-folds $N$ with boundary in $C$. Here is one with index one: 

\begin{ex} Let $s\ge 0$, and consider the $J_0$-holomorphic map
\begin{align*}
f_s&:\Si=\bigl\{a+ib\in\C:a,b\ge 0\bigr\}\longra\R^6=\C^3,\\
f_s&:a+ib\mapsto (x_1+iy_1,x_2+iy_2,x_3+iy_3)=(s(a+ib)-(a+ib)^3,(a+ib)^2,0).
\end{align*}
Then $f_s$ maps the boundary component $\bigl\{(a,0):a\in[0,\iy)\bigr\}$ of $\Si$ to $L\subset\R^6$, and the boundary component $\bigl\{(0,b):b\in[0,\iy)\bigr\}$ of $\Si$ to $M\subset\R^6$, so $f_s(\Si)$ is a $J_0$-holomorphic curve in $\R^6$ with boundary in $L\cup M$. If $s<0$ then $f_s$ does not map $(0,b)$ to $M$ for small $b>0$, which is why we restrict to~$s\ge 0$.

Let $N_s$ be the preimage of $f_s(\Si)$ under $\Pi:\R^7\ra\R^6$. Then $N_s$ for $s>0$ is a nonsingular 3-submanifold of $\R^7$ diffeomorphic to $[0,\iy)\t\R^2$, with boundary $\pd N_s\subset C$. One interior point of $N_s$, from $a+ib=\sqrt{s}$, maps to $C$.  Also $N_0$ is homeomorphic to $[0,\iy)\t\R^2$, but is not smooth at $(0,\ldots,0)$. These $N_s$ are not associative, since $f_s$ is holomorphic with respect to $J_0$ rather than $J$. But as in Conjecture \ref{ca6conj1}, we expect there to exist $J$-holomorphic maps $\ti f_s$ with essentially the same behaviour as $f_s$, yielding associative 3-folds $\ti N_s$ very like the $N_s$. 

Such $\ti N_s$, $s\ge 0$ should provide an example of an {\it index one singularity\/} of associative 3-folds $N$ with boundary in coassociative 4-folds $C$. That is, singularities of this type occur in codimension one in generic families of TA-$G_2$-manifolds, and so could cause numbers of associatives $N$ with $\pd N\subset C$ to change under deformation.
\label{ca6ex4}	
\end{ex}

Because of all this, the author expects that it is not possible to define an interesting Floer-type theory for coassociative 4-folds $C$ in $(X,\vp,\psi)$, suitably deformation-invariant in $\vp,\psi$, involving counting associatives $N$ with $\pd N\subset C$, following the analogy of Lagrangian Floer cohomology or Fukaya categories in symplectic geometry, say. But the author is not completely certain.

\section{A superpotential counting associative 3-folds}
\label{ca7}

\subsection{Set up of situation and notation}
\label{ca71}

In \S\ref{ca7} we will consider the following situation, and use the following notation. Let $X$ be a compact, oriented 7-manifold, and $\ga\in H^3_{\rm dR}(X;\R)$. Write $\cF_\ga$ for the set of closed 4-forms $\psi$ on $X$ such that there exists a closed 3-form $\vp$ on $X$ with $[\vp]=\ga\in H^3_{\rm dR}(X;\R)$, for which $(X,\vp,\psi)$ is a TA-$G_2$-manifold, with the given orientation on $X$. Suppose $\cF_\ga\ne\es$. Then $\cF_\ga$ is open in the vector space of closed 4-forms on $X$, and so is infinite-dimensional.

We will be discussing moduli spaces $\cM(\cN,\al,\psi)$ of compact associative 3-folds $N$ in such TA-$G_2$-manifolds $(X,\vp,\psi)$, but note as in \S\ref{ca25} that $\cM(\cN,\al,\psi)$ depends only on $\psi$ and the orientation on $X$, not on the choice of $\vp$. Given any $\psi$ or $\psi_t,$ $t\in[0,1]$ in $\cF_\ga$, we generally implicitly suppose we have chosen $\vp$ or $\vp_t,$ $t\in[0,1]$ to make TA-$G_2$-manifolds $(X,\vp,\psi)$ or $(X,\vp_t,\psi_t)$, but this is just for notational convenience, the choices of $\vp,\vp_t$ do not affect anything.

We often restrict to $\psi$ which is {\it generic in\/} $\cF_\ga$, as we expect this will simplify the singular behaviour of associatives considerably, as in Conjecture~\ref{ca2conj2}. 

Given generic $\psi_0,\psi_1$ in the same connected component of $\cF_\ga$, we can choose a smooth 1-parameter family $\psi_t,$ $t\in[0,1]$ in $\cF_\ga$ connecting $\psi_0,\psi_1$. We often restrict to a {\it generic\/ $1$-parameter family\/} $\psi_t,$ $t\in[0,1]$, that is, to a family which is generic amongst all smooth 1-parameter families with fixed end-points $\psi_0,\psi_1$. We expect that this will simplify the singular behaviour of associative 3-folds in $(X,\vp_t,\psi_t)$ for $t\in(0,1)$ considerably.

Fix a flag structure $F$ on $X$, as in \S\ref{ca31}. Then as in \S\ref{ca32} we have orientations $\Or(N)=\pm 1$ on $\cM(\cN,\al,\psi)$ at $[N]$ for all compact, unobstructed associative 3-folds $N$ in $(X,\vp,\psi)$.

Let $\F$ be the field $\Q,\R$ or $\C$. As in \S\ref{ca1}, write $\La$ for the Novikov ring over $\F$:
\e
\La=\bigl\{\ts\sum_{i=1}^\iy c_iq^{\al_i}: \text{$c_i\in\F,$ $\al_i\in\R,$ $\al_i\ra \iy$ as $i\ra\iy$}\bigr\}, 
\label{ca7eq1}
\e
with $q$ a formal variable. Then $\La$ is a commutative $\F$-algebra. Define $v:\La\ra\R\amalg\{\iy\}$ by $v(\la)$ is the least $\al\in\R$ with the coefficient of $q^\al$ in $\la$ nonzero for $\la\in\La\sm\{0\}$, and $v(0)=\iy$. Write $\La_{\ge 0}\subset\La$ for the subalgebra of $\la\in\La$ with $v(\la)\ge 0$, and $\La_{>0}\subset\La_{\ge 0}$ for the ideal of $\la\in\La$ with $v(\la)>0$. 

Then $\La$ is a {\it complete non-Archimedean field\/} in the sense of Bosch, G\"untzer and Remmert \cite[\S A]{BGR}, with valuation $\nm{\la}=2^{-v(\la)}$, so we can consider {\it rigid analytic spaces\/} over $\La$ as in \cite[\S C]{BGR}. These are like schemes over $\La$, except that polynomial functions on schemes are replaced by convergent power series.

Consider $1+\La_{>0}\subset\La$ as a group under multiplication in $\La$. Write
\begin{equation*}
\cU=\Hom\bigl(H_3(X;\Z),1+\La_{>0}\bigr)
\end{equation*}
for the set of group morphisms $\th:H_3(X;\Z)\ra 1+\La_{>0}$. By choosing a basis $e_1,\ldots,e_n$ for $H_3(X;\Z)/$torsion, where $n=b_3(X)$, we can identify $\cU\cong\La_{>0}^n$ by $\th\cong (\la_1,\cdots,\la_n)$ if $\th(e_i)=1+\la_i$ for $i=1,\ldots,n,$ where $\La_{>0}$ is the open unit ball in $\La$ in the norm $\nm{\,.\,}$. We regard $\cU$ as a {\it smooth rigid analytic space\/} over~$\La$.

A map $\Up:\cU\ra\cU$ will be called a {\it quasi-identity morphism\/} if:
\begin{itemize}
\setlength{\itemsep}{0pt}
\setlength{\parsep}{0pt}
\item[(i)] Writing $\Up(\th)=(\Up_1(\la_1,\ldots,\la_n),\ldots,\Up_n(\la_1,\ldots,\la_n))$ under $\cU\cong\La_{>0}^n$, each $\Up_i$ is given by a power series in $\la_1,\ldots,\la_n$ convergent in $\La_{>0}$.  
\item[(ii)] There exists $\ep>0$ such that if $(\la_1,\ldots,\la_n),(\la'_1,\ldots,\la'_n)\in\La_{>0}^n$ and $\de>0$ with $\la_i'-\la_i\in q^\de\cdot\La_{\ge 0}$ for $i=1,\ldots,n$ then
\begin{equation*}
\Up_j(\la'_1,\ldots,\la'_n)-\la'_j-\Up_j(\la_1,\ldots,\la_n)+\la_j\in 
q^{\de+\ep}\cdot\La_{\ge 0}\;\>\text{for $j=1,\ldots,n$.}
\end{equation*}
\end{itemize}
Here (i) implies that $\Up$ is a morphism of rigid analytic varieties. Using (ii) we can show that $\Up:\cU\ra\cU$ is a bijection, and $\Up^{-1}$ is also a quasi-identity morphism, so that $\Up$ is an isomorphism of rigid analytic varieties. Quasi-identity morphisms are closed under composition, and form a group.

\subsection{Six kinds of wall-crossing behaviour}
\label{ca72}

Suppose now that $\psi_0,\psi_1\in\cF_\ga$ are generic, and $\psi_t,$ $t\in[0,1]$ is a generic 1-parameter family joining $\psi_0,\psi_1$. As in \S\ref{ca26}, fix $\cN\in\cD$ and $\al\in H_3(N;\Z)$. We want to know how the moduli spaces $\cM(\cN,\al,\psi_t)$ can change over $t\in[0,1]$. We briefly sketch six conjectural ways in which this can happen, labelled (A)--(F), where (A) comes from \S\ref{ca3}, (B)--(D) from \S\ref{ca4}, and (E) from \S\ref{ca5}. All of (A)--(F) can also happen in reverse, that is, we can replace $\psi_t$ by $\psi_{1-t}$. 

When we say `associative 3-folds of interest', we just mean the family of associative 3-folds in $(X,\vp_t,\psi_t)$ whose behaviour we are describing. There may of course be many other associative 3-folds in $(X,\vp_t,\psi_t)$ as well.

\subsubsection*{\ref{ca72}(A) Cancelling non-singular associatives with opposite signs}

As explained in Example \ref{ca3ex1}, we expect the following can happen in generic families $\psi_t,$ $t\in[0,1]$, for some $t_0\in(0,1)$:
\begin{itemize}
\setlength{\itemsep}{0pt}
\setlength{\parsep}{0pt}
\item For $t\in[0,t_0)$ there are no associative 3-folds of interest in $(X,\vp_t,\psi_t)$.
\item There is a single compact, nonsingular associative 3-fold $N_{t_0}$ of interest in $(X,\vp_{t_0},\psi_{t_0})$. It is obstructed, with $\O_{N_{t_0}}\cong\R$. 
\item For $t\in(t_0,1]$ there are two compact, nonsingular, unobstructed associative 3-folds $N^+_t,N^-_t$ of interest in $(X,\vp_t,\psi_t)$, with $\lim_{t\ra t_0-}N^+_t=\lim_{t\ra t_0-}N^+_t=N_{t_0}$. They are diffeomorphic to $N_{t_0}$ and in the same homology class in $\al\in H_3(N;\Z)$, and have $\Or(N^+_t)=1$ and $\Or(N^-_t)=-1$. The canonical flags of $N^+_t,N^-_t$ differ by 1, in a suitable sense.
\end{itemize}

Provided we count unobstructed associatives $[N]\in\cM(\cN,\al,\psi)$ weighted by $\Or(N)$ (possibly multiplied by some 3-manifold invariant $I(N)$), the count does not change over $t\in[0,1]$ under this transition.

\subsubsection*{\ref{ca72}(B) Intersecting associatives $N_{t_0}^\pm$ give a connect sum $N_{t_0}^+\# N_{t_0}^-$}

As explained in Definition \ref{ca4def1}(a) and Conjecture \ref{ca4conj1}, we expect the following can happen in generic families $\psi_t,$ $t\in[0,1]$, for some $t_0\in(0,1)$:
\begin{itemize}
\setlength{\itemsep}{0pt}
\setlength{\parsep}{0pt}
\item For all $t\in[0,1]$ there are compact, connected, unobstructed associatives $N_t^+,N_t^-$ in $(X,\vp_t,\psi_t)$, depending smoothly on $t$. For $t\ne t_0$ we have $N_t^+\cap N_t^-=\es$, but $N_{t_0}^+\cap N_{t_0}^-=\{x\}$, and $N_t^+,N_t^-$ cross transversely at $x$ with nonzero speed as $t$ increases through $t_0$.
\item For $t\in(t_0,1]$ there is a compact, unobstructed associative 3-fold $\ti N_t$ in $(X,\vp_t,\psi_t)$, depending smoothly on $t$. It is diffeomorphic to $N_t^+\# N_t^-$, with $[\ti N_t]=[N_t^+]+[N_t^-]$ in $H_3(X;\Z)$, with $\lim_{t\ra t_0-}\ti N_t=N_{t_0}^+\cup N_{t_0}^-$. No such associative of interest exists for $t\in[0,t_0]$. We have $\Or(\ti N_t)=\Or(N_t^+)\cdot\Or(N_t^-)\cdot\ep$, where $\ep=\pm 1$ according to whether $N_t^+$ crosses $N_t^-$ with positive or negative intersection number in $X$.
\end{itemize}

\subsubsection*{\ref{ca72}(C) Self-intersecting $N_{t_0}$ gives a connect sum $N_{t_0}\#(\cS^1\t\cS^2)$}

As explained in Definition \ref{ca4def1}(b) and Conjecture \ref{ca4conj1}, we expect the following can happen in generic families $\psi_t,$ $t\in[0,1]$, for some $t_0\in(0,1)$:
\begin{itemize}
\setlength{\itemsep}{0pt}
\setlength{\parsep}{0pt}
\item For all $t\in[0,1]$ there is a compact, connected, unobstructed associative $N_t$ in $(X,\vp_t,\psi_t)$, depending smoothly on $t$. Here $N_{t_0}$ is immersed, with a self-intersection point $x\in X$, the image of distinct points $x^+,x^-$ in $N_{t_0}$. The two sheets of $N_t$ near $x^+,x^-$ cross transversely at $x$ with nonzero speed as $t$ increases through $t_0$.
\item For $t\in(t_0,1]$ there is a compact, unobstructed associative 3-fold $\ti N_t$ in $(X,\vp_t,\psi_t)$, depending smoothly on $t$. It is the self-connect-sum of $N_{t_0}$ at $x^+,x^-$, diffeomorphic to $N_t\#(\cS^1\t\cS^2)$. It has $[\ti N_t]=[N_t]$ in $H_3(X;\Z)$, and $\lim_{t\ra t_0-}\ti N_t=N_{t_0}$. No such associative of interest exists for~$t\in[0,t_0]$.
\end{itemize}

Note that $\ti N_t\cong N_t\#(\cS^1\t\cS^2)$ has $b^1(\ti N_t)\ge 1$, so $\ti N_t$ is not a $\Q$-homology 3-sphere. Thus, if we count only associative $\Q$-homology 3-spheres, the count does not change over $t\in[0,1]$ under this transition. 

\subsubsection*{\ref{ca72}(D) Self-intersecting $N_{t_0}$ gives a connect sum $N_{t_0}\# N_{t_0}$}

Here is a combination of (B),(C) above:
\begin{itemize}
\setlength{\itemsep}{0pt}
\setlength{\parsep}{0pt}
\item For all $t\in[0,1]$ there is a compact, connected, unobstructed associative $N_t$ in $(X,\vp_t,\psi_t)$, depending smoothly on $t$. Here $N_{t_0}$ is immersed, with a self-intersection point $x\in X$, the image of distinct points $x^+,x^-$ in $N_{t_0}$. The two sheets of $N_t$ near $x^+,x^-$ cross transversely at $x$ with nonzero speed as $t$ increases through $t_0$.
\item For $t\in(t_0,1]$ there is a compact, unobstructed associative 3-fold $\ti N_t$ in $(X,\vp_t,\psi_t)$, depending smoothly on $t$. It is the connect sum of two copies of $N_{t_0}$ at $x^+,x^-$, diffeomorphic to $N_t\# N_t$. It has $[\ti N_t]=2[N_t]$ in $H_3(X;\Z)$, and $\lim_{t\ra t_0-}\ti N_t=2N_{t_0}$. No such associative of interest exists for $t\in[0,t_0]$. We have $\Or(\ti N_t)=\ep$, where $\ep=\pm 1$ according to whether $N_t$ near $x^+$ crosses $N_t$ near $x^-$ with positive or negative intersection number in~$X$.
\end{itemize}

\subsubsection*{\ref{ca72}(E) Three families $N_t^1,N_t^2,N_t^3$ from $N_{t_0}$ with $T^2$-cone singularity}

As explained in Definition \ref{ca5def2} and Conjectures \ref{ca5conj1} and \ref{ca5conj2}, we expect the following can happen in generic families $\psi_t,$ $t\in[0,1]$, for some $t_0\in(0,1)$:
\begin{itemize}
\setlength{\itemsep}{0pt}
\setlength{\parsep}{0pt}
\item For all $t\in[0,t_0)$ there is a compact, unobstructed associative $N_t^1$ in $(X,\vp_t,\psi_t)$, depending smoothly on $t$.
\item For all $t\in(t_0,1]$ there are compact, unobstructed associatives $N_t^2,N_t^3$ in $(X,\vp_t,\psi_t)$, depending smoothly on $t$.
\item There is a compact associative $N_{t_0}$ in $(X,\vp_{t_0},\psi_{t_0})$ with one singular point at $x\in X$ locally modelled on the associative $T^2$-cone $L_0\subset\R^7$ from \S\ref{ca51}. We have $\lim_{t\ra t_0-}N_t^1=\lim_{t\ra t_0+}N_t^2=\lim_{t\ra t_0+}N_t^3=N_{t_0}$, where $N_t^a$ is locally modelled near $x$ on $L_s^a\subset\R^7$ in \S\ref{ca51}, for $\md{t-t_0}$ and $s>0$ small.
\item Writing $I$ for the 3-manifold invariant in \eq{ca5eq12}, from \eq{ca5eq15} we have
\e
\Or(N^1_t)\cdot I(N^1_t)=\Or(N^2_t)\cdot I(N^2_t)+\Or(N^3_t)\cdot I(N^3_t).
\label{ca7eq2}
\e

\end{itemize}

If we count unobstructed associatives $[N]\in\cM(\cN,\al,\psi)$ weighted by $\Or(N)\cdot I(N)$, equation \eq{ca7eq2} implies that the count does not change over $t\in[0,1]$ under this transition. Note that $I(N)=0$ unless $N$ is a $\Q$-homology sphere, so this is consistent with counting only associative $\Q$-homology 3-spheres, as in~(C).

\subsubsection*{\ref{ca72}(F) Multiple cover phenomena}

This is one of the less satisfactory parts of this paper. 

The author expects that in generic families $\psi_t,$ $t\in[0,1]$, for some $t_0\in(0,1)$, it can happen that a family of associative 3-folds $\hat N_t$ in $(X,\vp_t,\psi_t)$ for $t\in(t_0,1]$ can converge as $t\ra t_0$ to a {\it branched multiple cover\/} of some associative $N_{t_0}$ in $(X,\vp_{t_0},\psi_{t_0})$, where $N_{t_0}$ may be obstructed, or immersed, or singular. There may be several ways in which this can happen. 

We illustrate this using (B) above. We expect the following can happen in generic families $\psi_t,$ $t\in[0,1]$, for some $t_0\in(0,1)$:
\begin{itemize}
\setlength{\itemsep}{0pt}
\setlength{\parsep}{0pt}
\item Let $N_t^\pm,t_0,x,\ti N_t$ be as in (B). Then for $t\in(t_0,1]$ there is a compact, unobstructed associative 3-fold $\hat N_t$ in $(X,\vp_t,\psi_t)$, depending smoothly on $t$. Topologically, we have $\hat N_t\cong k^+\,N_t^+\# k^-\,N_t^-\# l\,(\cS^1\t\cS^2)$, where $k^\pm\ge 1$ with $(k^+,k^-)\ne(1,1)$ and $l\ge 0$. That is, $\hat N_t$ is the connect sum of $k^+$ copies of $N_t^+$ and $k^-$ copies of $N_t^-$ at $k^++k^-+l-1$ pairs of points.
\item As $t\ra t_0$, $\hat N_t$ converges to a branched multiple cover of $N_{t_0}^+\cup N^-_{t_0}$, with multiplicity $k^+$ over $N_{t_0}^+$ and multiplicity $k^-$ over $N_{t_0}^+$. There is a 1-dimensional singular set $S\subset N_{t_0}^+\cup N^-_{t_0}$ with $x\in S$, probably a union of points $x'$ and curves $\ga$ with end-points. Over $N_{t_0}^+\sm S$ (or $N_{t_0}^-\sm S$), $k^+$ sheets (or $k^-$ sheets) of $\hat N_t$ converge smoothly to $N_{t_0}^\pm\sm S$. On the interiors $\ga^\ci$ of curves $\ga$ in $S$, $\hat N_t$ should look like a double cover of $N_{t_0}^\pm$ branched along $\ga^\ci$, as for branched covers of Riemann surfaces but one dimension higher. At points $x'$ or end-points of curves $\ga$ in $S$, the local models for how $\hat N_t$ converges to $N_{t_0}^+\cup N^-_{t_0}$ are more complicated. 
\end{itemize}

Using the ideas of \S\ref{ca6} we can write down heuristic $\U(1)$-invariant {\it local\/} models for how $\hat N_t$ can converge to $N_{t_0}^+\cup N^-_{t_0}$, based on branched-cover behaviour for families of $J_0$-holomorphic curves in $\C^3$ with boundary in~$L\subset\C^3$.

However, the author does not have a conjectural {\it global\/} description of how such multiple cover transitions happen, that is detailed enough to predict how many associatives $\hat N_t$ of each type $(\cN,\al)$ are created or destroyed in each such transition. Such a global description would necessarily be complicated. 

In the example above, suppose we have families $\hat N_t^1$ of type $(k_1^+,k_1^-,l_1)$ and $\hat N_t^2$ of type $(k_2^+,k_2^-,l_2)$ for $t\in(t_0,1]$. If we deform the geometry so that $\hat N_t^1$ crosses $\hat N_t^2$, then as in (B) above we create a new associative $\hat N_t^1\#\hat N_t^2$, which is another $\hat N_t$ of type $(k^+,k^-,l)=(k_1^++k_2^+,k_1^-+k_2^-,l_1+l_2)$. Because of this, the number of $\hat N_t$'s of type $(k^+,k^-,l)$ that appear or disappear as $t$ crosses $t_0$ will depend on all the other $\hat N'_t$ of type $(k^{\prime +},k^{\prime -},l')$ for $(k^{\prime +},k^{\prime -},l')<(k^+,k^-,l)$, and the canonical flags of these $\hat N'_t$, and their pairwise `linking numbers'.

We can see (D) as the simplest example of such a multiple cover transition.

Similar (but simpler) multiple cover phenomena occur for $J$-holomorphic curves in symplectic geometry, and do not spoil the deformation-invariance.

\subsection{Definition of the superpotential}
\label{ca73}

Work in the situation of \S\ref{ca71}, and assume Conjecture \ref{ca2conj2}. 

Let $\psi\in\cF_\ga$ be generic. We will define a superpotential $\Phi_\psi:\cU\ra \La_{>0}$, which is a generating function for Gromov--Witten type invariants $GW_{\psi,\al}$ counting associative $\Q$-homology spheres $N$ in $(X,\vp,\psi)$ with $[N]=\al\in H_3(X;\Z)$, depending on some arbitrary choices. 
 
\begin{dfn} For $i=0,\ldots,7$, choose elements $e^i_1,\ldots,e_{b_i(X)}^i$ in $H_i(X;\Z)$ such that $e^i_1,\ldots,e_{b_i(X)}^i$ is a basis for $H_i(X;\Q)$, with $e^7_1=[X]$. Choose compact, embedded, oriented, generic $i$-dimensional submanifolds $C^i_1,\ldots,C_{b_i(X)}^i$ in $X$ with $[C^i_j]=e^i_j$ in $H_i(X;\Z)$ for $j=1,\ldots,b_i(X)$, with~$C^7_1=X$.

By the K\"unneth theorem, $e^i_j\boxtimes e^{7-i}_k$ for $j=1,\ldots,b_i(X)$, $k=1,\ldots,b_{7-i}(X)$ is a basis for the homology group $H_7(X\t X;\Q)$, where $e^i_j\boxtimes e^{7-i}_k$ is represented by the compact, oriented submanifold $C^i_j\t C^{7-i}_k$ in $X\t X$. The diagonal map $\De_X:X\ra X\t X$, $\De_X:x\mapsto (x,x)$, gives a homology class $[\De_X(X)]$ in $H_7(X\t X;\Q)$. Hence for some coefficients $A^i_{jk}\in\Q$ we have
\begin{equation*}
[\De_X(X)]=\sum_{i=0}^7\sum_{j=1}^{b_i(X)}\sum_{k=1}^{b_{7-i}(X)} A^i_{jk}e^i_j\boxtimes e^{7-i}_k \quad\text{in $H_7(X\t X;\Q)$,}
\end{equation*}
with $(A^i_{jk})_{j,k=1}^{b_i(X)}$ the matrix of the intersection form $H_i(X;\Q)\t H_{7-i}(X;\Q)\ra\Q$.

Therefore we can choose an 8-chain $D$ in homology of $X\t X$ over $\Q$ with
\e
\pd D=\De_X(X)-\ts\sum_{i=0}^7\sum_{j=1}^{b_i(X)}\sum_{k=1}^{b_{7-i}(X)} A^i_{jk}\cdot C^i_j\t C^{7-i}_k.
\label{ca7eq3}
\e

As $\psi\in\cF_\ga$ is generic and we assume Conjecture \ref{ca2conj2}, for each $\al\in H_3(X;\Z)$ and $\cN\in\cD,$ the moduli space $\cM(\cN,\al,\psi)$ is finite and $N$ is finite-embedded and unobstructed for each $[N,i]\in\cM(\cN,\al,\psi)$. By genericness of $C^i_j$ we can suppose that for all such $N$ we have $N\cap C^i_j=\es$ for all $i=0,1,2,3$ and $j=1,\ldots,b_i(X)$.

Recall that a {\it tree\/} is a finite, undirected graph $\Ga$ which is connected and simply-connected. A tree $\Ga$ has a set $V$ of {\it vertices\/} $v$, and a set $E$ of {\it edges\/} $e$ joining two vertices $v,w$. In the next equation, a {\it labelled tree\/} $(\Ga,[N_v,i_v]_{v\in V})$ is a tree $\Ga$ together with an isomorphism class $[N_v,i_v]$ of compact, immersed associative $\Q$-homology spheres $i_v:N_v\ra X$ in $(X,\vp,\psi)$ for all $v\in V$, so that $[N_v,i_v]\in\cM(\cN,\al,\psi)$ for some $\cN\in\cDHS$ and $\al\in H_3(X;\Z)$.

Define a superpotential $\Phi_\psi:\cU\ra \La_{>0}$ by
\ea
\Phi_\psi(\th)&=\sum_{\begin{subarray}{l}\text{labelled trees}\\ (\Ga,[N_v,i_v]_{v\in V})\end{subarray}}\frac{1}{\md{\Aut(\Ga,[N_v,i_v]_{v\in V})}}\prod_{v\in V}
\frac{\Or(N_v)I(N_v)}{\md{\Iso([N_v,i_v])}}\cdot q^{\ga\cdot[N_v]}\th([N_v])
\nonumber\\
&\cdot\!\!\prod_{\begin{subarray}{l}\text{edges $\mathop{\bu}\limits^{\sst v}-\mathop{\bu}\limits^{\sst w}$ in $\Ga$: $N'_v,N'_w$ are small} \\ 
\text{perturbations of $N_v,N_w$ in directions $f_{N_v},f_{N_w}$}
\end{subarray}\!\!\!\!\!\!\!\!\!\!\!\!\!\!\!\!\!\!\!\!\!\!\!\!\!\!\!\!\!\!\!\!\!\!\!\!\!\!\!\!}\ha(N'_v\t N_w+N'_w\t N_v)\bu D
\label{ca7eq4}\\
&\text{+\,similar, but unknown, contributions from multiple covers.}
\nonumber
\ea
Here in the first line, $\Aut(\Ga,[N_v,i_v]_{v\in V})$ is the finite group of automorphisms of $\Ga$ preserving the assignment $v\mapsto[N_v,i_v]$. For each $v\in V$, $\Iso([N_v,i_v])$ is as in Definition \ref{ca2def5}, and $\Or(N_v)$ as in \S\ref{ca32}, and $I(N_v)$ as in~\eq{ca5eq12}.

In the second line, the associatives $N_v,N_w$ have canonical flags $f_{N_v},f_{N_w}$, as in \S\ref{ca32}. We choose representatives $s_{N_v}\in\Ga^\iy(\nu_{N_v})$, $s_{N_w}\in\Ga^\iy(\nu_{N_w})$ for $f_{N_v},f_{N_w}$, and take $N'_v,N'_w$ to be small perturbations of $N_v,N_w$ in normal directions $s_{N_v},s_{N_w}$.
Then $(N'_v\t N_w+N'_w\t N_v)\bu D$ in \eq{ca7eq4} is the intersection number in homology over $\Q$ of the 6-cycle $N'_v\t N_w+N'_w\t N_v$ and the 8-chain $D$. This is well defined provided $N'_v\t N_w+N'_w\t N_v$ does not intersect $\pd D$, which is given in \eq{ca7eq3}. As above $N_v,N_w$ do not intersect $C^i_j$ for $i=0,1,2,3$, so $N'_v,N'_w$ also do not intersect $C^i_j$ as they are close to $N_v,N_w$. Hence $N'_v\t N_w+N'_w\t N_v$ does not intersect $\sum_{i,j,k}A^i_{jk}\cdot C^i_j\t C^{7-i}_k$ in~\eq{ca7eq3}. 

To see that $N'_v\t N_w+N'_w\t N_v$ does not intersect $\De_X(X)$, as $\psi$ is generic we may divide into cases (i) $N_v\cap N_w=\es$, and (ii) $N_v$ and $N_w$ are finite covers of the same embedded $N\subset X$. In case (i) $N_v'\cap N_w=\es=N_w'\cap N_v$ as $N_v',N_w'$ are close to $N_v,N_w$. In case (ii) $N_v'\cap N_w=\es=N_w'\cap N_v$ since $N_v,N_w$ have the same image $N\subset X$. So in both cases $(N'_v\t N_w+N'_w\t N_v)\cap\De_X(X)=\es$, and $(N'_v\t N_w+N'_w\t N_v)\bu D$ is well defined.

Each edge $\mathop{\bu}\limits^{\sst v}-\mathop{\bu}\limits^{\sst w}$ in $\Ga$ appears only once in the product in \eq{ca7eq4}, that is, we do not distinguish $\mathop{\bu}\limits^{\sst v}-\mathop{\bu}\limits^{\sst w}$ and $\mathop{\bu}\limits^{\sst w}-\mathop{\bu}\limits^{\sst v}$. This makes sense as $(N'_v\t N_w+N'_w\t N_v)\bu D$ is symmetric in $v,w$. The sum \eq{ca7eq4} is generally infinite, but by the last part of Conjecture \ref{ca2conj2} there are only finitely many labelled trees $(\Ga,[N_v,i_v]_{v\in V})$ with $\sum_{v\in V}\ga\cdot[N_v]\le A$ for any $A>0$, so the term $\prod_{v\in V}q^{\ga\cdot[N_v]}$ in \eq{ca7eq4} implies that the sum converges in $\La_{>0}$, and thus the first two lines of \eq{ca7eq4} are well defined.

For the third line of \eq{ca7eq4}, the idea is to include correction terms which will ensure deformation-invariance of $\Phi_\psi$ under the multiple cover phenomena discussed in \S\ref{ca72}(F). As the author does not have a good conjectural description of these phenomena, we cannot yet write down the correction terms explicitly. We will mostly ignore this issue, and just hope things work out nicely.

We can also write \eq{ca7eq4} as 
\e
\Phi_\psi(\th)=\sum_{\al\in H_3(X;\Z):\ga\cdot\al>0}GW_{\psi,\al}\, q^{\ga\cdot\al}\,\th(\al), 
\label{ca7eq5}
\e
where $GW_{\psi,\al}\in\Q$ is defined by taking $GW_{\psi,\al}\, q^{\ga\cdot\al}\,\th(\al)$ to be the sum of all terms in \eq{ca7eq4} from $(\Ga,[N_v,i_v]_{v\in V})$ with $\sum_{v\in V}[N_v]=\al$ in $H_3(X;\Z)$. Then $GW_{\psi,\al}$ is a Gromov--Witten type invariant counting associative $\Q$-homology spheres in class $\al$ in $(X,\vp,\psi)$. Note however that the $GW_{\psi,\al}$ are {\it not independent of the choices of\/} $C^i_j,D$, and are {\it not invariant under deformations of\/ $\psi$ in\/} $\cF_\ga$. So they are not enumerative invariants in the usual sense.
\label{ca7def1}	
\end{dfn}

\begin{rem} We can interpret \eq{ca7eq4} as the sum of a `main term' $\Phi_\psi^{\rm main}(\th)$ coming from trees $\Ga$ with one vertex and no edges, and a series of increasingly complex `correction terms' coming from trees $\Ga$ with $n\ge 2$ vertices and $n-1$ edges, as $n\ra\iy$. The `main term' may be rewritten as
\begin{equation*}
\Phi_\psi^{\rm main}(\th)=\sum_{\begin{subarray}{l}\al\in H_3(X;\Z):\\ \quad \ga\cdot\al>0\end{subarray}}\sum_{\cN\in\cDHS}\sum_{[N,i]\in\cM(\cN,\al,\psi)}\frac{\Or(N)I(N)}{\md{\Iso([N,i])}}\cdot q^{\ga\cdot\al}\th(\al).
\end{equation*}
This is a straightforward weighted count of associative $\Q$-homology 3-spheres. Now $\Phi_\psi^{\rm main}(\th)$ is not deformation-invariant, because of the wall-crossing behaviour in \S\ref{ca72}(B),(D). The `correction terms' are designed to remedy this.
\label{ca7rem1}	
\end{rem}

\subsection{\texorpdfstring{How $\Phi_\psi$ depends on choices, and on $\psi$}{How Φᵩ depends on choices, and on ψ}}
\label{ca74}

We now consider how $\Phi_\psi$ in \S\ref{ca73} depends on the arbitrary choices $C^i_j,D$ in its definition, and how it varies under smooth deformations of $\psi$ in $\cF_\ga$. 

The next ``theorem'' depends on the conjectures in \S\ref{ca2}--\S\ref{ca5}, and we only sketch the proof. The hypotheses are rather limited and artificial. As in \S\ref{ca72}, we do not have a detailed conjecture for how multiple cover phenomena in \S\ref{ca72}(F) behave. So we exclude them, by just assuming that only wall-crossings of type \S\ref{ca72}(A)--(E) occur. However, the author actually expects that some \S\ref{ca72}(F) phenomena will occur simultaneously with \S\ref{ca72}(A)--(E), and \S\ref{ca72}(F) is needed to cancel interaction terms in \eq{ca7eq4} between pairs of associatives in \S\ref{ca72}(A)--(E). Part (a)(iii) ensures, just by assumption, that these interaction terms are zero. 

\begin{thm}{\bf(a)} Let\/ $\psi_0,\psi_1\in\cF_\ga$ be generic, and\/ $\psi_t,$ $t\in[0,1]$ be a generic smooth\/ $1$-parameter family in $\cF_\ga$ connecting $\psi_0,\psi_1$. Suppose that:
\begin{itemize}
\setlength{\itemsep}{0pt}
\setlength{\parsep}{0pt}
\item[{\bf(i)}] The only changes to moduli spaces $\cM(\cN,\al,\psi_t)$ as\/ $t$ increases through\/ $[0,1]$ are those of type {\rm\S\ref{ca72}(A)--(E)} (and \begin{bfseries}not\end{bfseries} those of type~{\rm\S\ref{ca72}(F)}).
\item[{\bf(ii)}] For any $A>0,$ only finitely many changes happen over $t\in[0,1]$ to all\/ $\cM(\cN,\al,\psi_t)$ with\/ $\ga\cdot\al\le A$. 
\item[{\bf(iii)}] If\/ $N^1_t,N^2_t$ are two distinct associatives in $(X,\vp_t,\psi_t)$ considered in one of\/ {\rm\S\ref{ca72}(A)--(E),} that do not both exist for all\/ $t\in[0,1]$ (this excludes $N^1_t,N^2_t$ being $N_t^+,N_t^-$ in {\rm\S\ref{ca72}(B)}), and\/ $N^{\prime 1}_t,N^{\prime 2}_t$ are small perturbations of\/ $N^1_t,N^2_t$ in directions $f_{N_t^1},f_{N_t^2},$ then $(N^{\prime 1}_t\t N^2_t+N^{\prime 2}_t\t N^1_t)\bu D=0$.
\item[{\bf(iv)}] $C^i_j,D$ in Definition\/ {\rm\ref{ca7def1}} are independent of\/ $t,$ and\/ $C^i_j\cap i_t(N_t)=\es$ for all\/ $i=0,\ldots,3,$ $j=1,\ldots,b_i(X),$ $t\in[0,1]$ and\/ $[N_t,i_t]\in\cM(\cN,\al,\psi_t)$. 
\end{itemize}
Define $\Phi_{\psi_t}$ as in {\rm\eq{ca7eq4},} but taking the `unknown multiple cover contributions' in the third line to be zero. Then~$\Phi_{\psi_1}=\Phi_{\psi_0}$.
\smallskip

\noindent{\bf(b)} Generalize {\bf(a)} by dropping {\bf(iv)}. Then there is a quasi-identity morphism $\Up:\cU\ra\cU$ in the sense of\/ {\rm\S\ref{ca71}} with\/ $\Phi_{\psi_1}=\Phi_{\psi_0}\ci\Up$.
\smallskip

\noindent{\bf(c)} Suppose $\ti\Phi_\psi$ is defined in \eq{ca7eq4} using alternative choices $\ti C^i_j,\ti D$ for $C^i_j,D$ in Definition\/ {\rm\ref{ca7def1}}. Then $\ti\Phi_\psi=\Phi_\psi\ci\Up$ for some quasi-identity morphism $\Up:\cU\ra\cU$.

\label{ca7thm1}	
\end{thm}

\begin{proof}[Sketch proof] For (a), $\Phi_{\psi_t}$ is defined for generic $t\in[0,1]$. We claim that $\Phi_{\psi_t}$ is constant in $t$, so that $\Phi_{\psi_0}=\Phi_{\psi_1}$. For $A>0$, consider the projection $\Phi_{\psi_t}+q^A\La_{>0}$ of $\Phi_{\psi_t}$ to $\La_{>0}/q^A\La_{>0}$. Part (ii) implies that $\Phi_{\psi_t}+q^A\La_{>0}$ undergoes at most finitely many changes in $t\in[0,1]$, each from a single transition in \S\ref{ca72}(A)--(E). We will show that $\Phi_{\psi_t}+q^A\La_{>0}$ is actually unchanged by each such transition. For (A),(C),(E) this follows from the discussion in \S\ref{ca72}, as \eq{ca7eq4} counts associative $\Q$-homology 3-spheres $N$ weighted by $\Or(N)I(N)$, together with part (iii), which ensures that interactions in \eq{ca7eq4} between pairs of associatives in (A),(C),(E) are all zero.

Let $N_t^\pm$ for $t\in[0,1]$, $x\in X$, $t_0\in(0,1)$, $\ti N_t$ for $t\in(t_0,1]$, and $\ep=\pm 1$ be as in \S\ref{ca72}(B). Then the sum \eq{ca7eq4} changes as $t$ crosses $t_0$ in two ways:
\begin{itemize}
\setlength{\itemsep}{0pt}
\setlength{\parsep}{0pt}
\item[$(\dag)$] When $t>t_0$ we can have terms in \eq{ca7eq4} from $(\Ga,[N_v,i_v]_{v\in V})$ with $N_v=\ti N_t$ for some $v\in V$. This does not happen for $t<t_0$.
\item[$(\ddag)$] Consider terms in \eq{ca7eq4} from $(\hat\Ga,[\hat N_v,\hat\imath_v]_{v\in\hat V})$ in which $\hat\Ga$ contains an edge $\mathop{\bu}\limits^{\sst v}-\mathop{\bu}\limits^{\sst w}$ with $\hat N_v=N_t^+$ and $\hat N_w=N_t^-$. Then the second line of \eq{ca7eq4} includes a factor $\ha(N^{+\prime}_t\t N_t^-+N^{-\prime}_t\t N^+_t)\bu D$. This factor (which (iii) does not require to be zero) changes by the addition of $-\ep$ as $t$ increases through $t_0$, because of extra intersection points of $N^{+\prime}_t\t N_t^-$ and $N^{-\prime}_t\t N^+_t$ with $D$ near $(x,x)$ in~$X\t X$.
\end{itemize} 

There is a map from trees $(\hat\Ga,[\hat N_v,\hat\imath_v]_{v\in\hat V})$ in $(\ddag)$ to trees $(\Ga,[N_v,i_v]_{v\in V})$ in $(\dag)$, in which we contract edges $\mathop{\bu}\limits^{\sst v}-\mathop{\bu}\limits^{\sst w}$ in $\hat\Ga$ with $\hat N_v=N_t^+$ and $\hat N_w=N_t^-$ to a vertex $v'$ in $\Ga$ with $N_{v'}=\ti N_t$. Under this map, the changes to \eq{ca7eq4} cancel, because we have $\Or(\ti N_t)=\Or(N_t^+)\cdot\Or(N_t^-)\cdot\ep$ from \S\ref{ca72}(B), and $I(\ti N_t)=I(N_t^+)I(N_t^-)$ by \eq{ca5eq14} as $\ti N_t\cong N_t^+\# N_t^-$. Thus $\Phi_{\psi_t}+q^A\La_{>0}$ is unchanged under transitions of type~(B).

Now let $N_t$ for $t\in[0,1]$, $x\in X$, $t_0\in(0,1)$, $\ti N_t$ for $t\in(t_0,1]$, and $\ep=\pm 1$ be as in \S\ref{ca72}(D). Then the sum \eq{ca7eq4} changes as $t$ crosses $t_0$ in two ways:
\begin{itemize}
\setlength{\itemsep}{0pt}
\setlength{\parsep}{0pt}
\item[$(\dag)'$] When $t>t_0$ we can have terms in \eq{ca7eq4} from $(\Ga,[N_v,i_v]_{v\in V})$ with $N_v=\ti N_t$ for some $v\in V$. This does not happen for $t<t_0$.
\item[$(\ddag)'$] Consider terms in \eq{ca7eq4} from $(\hat\Ga,[\hat N_v,\hat\imath_v]_{v\in\hat V})$ in which $\hat\Ga$ contains an edge $\mathop{\bu}\limits^{\sst v}-\mathop{\bu}\limits^{\sst w}$ with $\hat N_v=\hat N_w=N_t$. Then the second line of \eq{ca7eq4} includes a factor $(N'_t\t N_t)\bu D$. This factor (which (iii) does not require to be zero) changes by the addition of $-2\ep$ as $t$ increases through $t_0$, because of two extra intersection points of $N'_t,N_t$ with $D$ near $(x,x)$ in~$X\t X$.
\end{itemize} 

Again, there is a map from $(\hat\Ga,[\hat N_v,\hat\imath_v]_{v\in\hat V})$ in $(\ddag)'$ to $(\Ga,[N_v,i_v]_{v\in V})$ in $(\dag)'$, in which we contract edges $\mathop{\bu}\limits^{\sst v}-\mathop{\bu}\limits^{\sst w}$ in $\hat\Ga$ with $\hat N_v=\hat N_w=N_t$ to a vertex $v'$ in $\Ga$ with $N_{v'}=\ti N_t$. Under this map, the changes to \eq{ca7eq4} cancel, because we have $\Or(\ti N_t)=\ep=\Or(N_t)^2\ep$ from \S\ref{ca72}(D), and $I(\ti N_t)=I(N_t)^2$ as~$\ti N_t\cong N_t\# N_t$. 

The factor 2 in $-2\ep$ in $(\ddag)'$ is dealt with by the comparison between factors $1/\md{\Aut(\Ga,[N_v,i_v]_{v\in V})}$ and $1/\md{\Aut(\hat\Ga,[\hat N_v,\hat\imath_v]_{\smash{v\in\hat V}})}$ in \eq{ca7eq4}. For example, in the simplest case in which $\Ga=\bu$ and $\hat\Ga=\mathop{\bu}\limits^{\sst v}-\mathop{\bu}\limits^{\sst w}$ we have $1/\md{\Aut(\Ga,[N_v,i_v]_{v\in V})}=1$ and $1/\md{\Aut(\hat\Ga,[\hat N_v,\hat\imath_v]_{v\in\hat V})}=\ha$, where the $\ha$ cancels the 2 in $-2\ep$. Thus $\Phi_{\psi_t}+q^A\La_{>0}$ is unchanged under transitions of type (D). Hence $\Phi_{\psi_t}+q^A\La_{>0}$ is independent of $t$ for all $A>0$, so $\Phi_{\psi_1}=\Phi_{\psi_0}$, proving~(a).

For (b), the difference with (a) is that as (iv) does not hold, we now must allow associatives $N_t$ in $(X,\vp_t,\psi_t)$ with $C^i_j\cap N_t\ne\es$ for some $i=0,1,2,3$ and $j$. In fact, as $C^i_j$ is generic and there are only countably many smooth families of 3-folds $N_t$, $t\in[0,1]$ in $X$, it is automatic that $C^i_j\cap N_t=\es$ for $i=0,1,2$ for dimensional reasons, so we need only consider $i=3$, and then the only possibility is that $C^3_j\cap N_{t_0}=\{x\}$ for some $t_0\in(0,1)$, where $N_t$ crosses $C^3_j$ transversely as $t$ increases through~$t_0$.

First we consider the effect of just one such transition. So suppose that we have just one family $[\ti N_t,\ti\imath_t]\in\cM(\ti\cN,\ti\al,\psi_t)$ depending smoothly on $t\in[0,1]$, with $C^3_{\ti\jmath}\cap\ti N_{t_0}=\{x\}$ for $t_0\in(0,1)$, and $C^3_{\ti\jmath}\cap\ti N_t=\es$ for $t\ne t_0$, and $\ti N_t$ crosses $C^3_{\ti\jmath}$ transversely as $t$ increases through $t_0$ with intersection number $\ep=\pm 1$, and that $C^i_j\cap i_t(N_t)=\es$ for all $i=0,\ldots,3,$ $j=1,\ldots,b_i(X),$ $t\in[0,1]$ and $[N_t,i_t]\in\cM(\cN,\al,\psi_t)$ unless $i=3$, $j=\ti\jmath$, $t=t_0$ and $[N_t,i_t]=[\ti N_{t_0},\ti\imath_{t_0}]$. 

Define $\de=\ep\cdot\sum_{k=1}^{b_4(X)}A^i_{\ti\jmath k}e^4_k$ in $H_4(X;\Q)$. Then the effect of this change on \eq{ca7eq4} is that for each labelled tree $(\Ga,[N_v,i_v]_{v\in V})$ including an edge $\mathop{\bu}\limits^{\sst v}-\mathop{\bu}\limits^{\sst w}$, then:
\begin{itemize}
\setlength{\itemsep}{0pt}
\setlength{\parsep}{0pt}
\item[$(*)$] $\ha(N'_v\t N_w+N'_w\t N_v)\bu D$ in \eq{ca7eq4} increases by $[N_w]\bu\de$ as $t$ increases through $t_0$ if $[N_v,i_v]=[\ti N_t,\ti\imath_t]$ and $[N_w,i_w]\ne[\ti N_t,\ti\imath_t]$.
\item[$(**)$] $\ha(N'_v\t N_w+N'_w\t N_v)\bu D$ in \eq{ca7eq4} increases by $2[\ti N_t]\bu\de$ as $t$ increases through $t_0$ if $[N_v,i_v]=[N_w,i_w]=[\ti N_t,\ti\imath_t]$.
\end{itemize}
Here $\bu:H_3(X;\Q)\t H_4(X;\Q)\ra\Q$ is the intersection form. The reason for $(*)$ is that as $\ti N_t$ crosses $C^3_{\ti\jmath}$ in $X$ with intersection number $\ep$, $\ti N_t\t N_w$ (and also $\ti N'_t\t N_w$) crosses $C^3_{\ti\jmath}\t C^4_k$ in $X\t X$ with intersection number $\ep\cdot[N_w]\bu e^4_k$. Thus by \eq{ca7eq3}, the change in $(\ti N'_t\t N_w)\bu D$ as $t$ increases through $t_0$ is
\begin{equation*}
\ep\cdot\ts\sum_{k=1}^{b_4(X)}A^i_{\ti\jmath k}[N_w]\bu e^4_k=[N_w]\bu\de.
\end{equation*}
The change in $(N_w'\t\ti N_t)\bu D$ is the same. For $(**)$ we use a similar argument.

From $(*)$ and $(**)$ above we can show that
\ea
\Phi_{\psi_1}(\th)&=\sum_{\begin{subarray}{l}\text{labelled trees}\\ (\Ga,[N_v,i_v]_{v\in V})\\ \text{for $(X,\vp_0,\psi_0)$}\end{subarray}}\,\,\sum_{\begin{subarray}{l}\text{$S$ set of directed edges $\mathop{\bu}\limits^{\sst v}\ra\mathop{\bu}\limits^{\sst w}$}\\
\text{in $\Ga$ with $[N_v,i_v]=[\ti N_0,\ti\imath_0]$}\end{subarray}}\frac{1}{\md{\Aut(\Ga,[N_v,i_v]_{v\in V})}}
\nonumber\\
&\cdot \biggl(\prod_{v\in V}
\frac{\Or(N_v)I(N_v)}{\md{\Iso([N_v,i_v])}}\cdot q^{\ga\cdot[N_v]}\th([N_v])\biggr)\cdot\!\!\prod_{\text{edges $\mathop{\bu}\limits^{\sst v}\ra\mathop{\bu}\limits^{\sst w}$ in $S$}\!\!\!\!\!\!\!\!\!\!\!\!\!\!\!\!\!\!\!\!\!\!\!\!\!\!\!}[N_w]\bu\de
\nonumber\\
&\cdot\!\!\prod_{\begin{subarray}{l}\text{edges $\mathop{\bu}\limits^{\sst v}-\mathop{\bu}\limits^{\sst w}$ in $\Ga$ but not in $S$: $N'_v,N'_w$ are small} \\ 
\text{perturbations of $N_v,N_w$ in directions $f_{N_v},f_{N_w}$}
\end{subarray}\!\!\!\!\!\!\!\!\!\!\!\!\!\!\!\!\!\!\!\!\!\!\!\!\!\!\!\!\!\!\!\!\!\!\!\!\!\!\!\!\!\!\!\!\!\!\!\!\!\!\!\!\!\!\!\!\!\!\!\!\!\!\!}\ha(N'_v\t N_w+N'_w\t N_v)\bu D.
\label{ca7eq6}
\ea
Here the labelled trees $(\Ga,[N_v,i_v]_{v\in V})$ are as in \eq{ca7eq4} for $(X,\vp_0,\psi_0)$. On the first line we choose a subset $S$ of edges $\mathop{\bu}\limits^{\sst v}-\mathop{\bu}\limits^{\sst w}$ in $\Ga$, to each of which we assign a direction, written $\mathop{\bu}\limits^{\sst v}\ra\mathop{\bu}\limits^{\sst w}$, where we must have $[N_v,i_v]=[\ti N_0,\ti\imath_0]$. For any fixed $(\Ga,[N_v,i_v]_{v\in V})$, taking the sum in \eq{ca7eq6} over all $S$ is equivalent to replacing the factor $\ha(N'_v\t N_w+N'_w\t N_v)\bu D$ in \eq{ca7eq4} by $\ha(N'_v\t N_w+N'_w\t N_v)\bu D+[N_w]\bu\de$ for each edge $(*)$ (when the direction $\mathop{\bu}\limits^{\sst v}\ra\mathop{\bu}\limits^{\sst w}$ is fixed uniquely), and by $\ha(N'_v\t N_w+N'_w\t N_v)\bu D+2[\ti N_t]\bu\de$ for each edge $(**)$ (when both directions $\mathop{\bu}\limits^{\sst v}\ra\mathop{\bu}\limits^{\sst w}$ and $\mathop{\bu}\limits^{\sst w}\ra\mathop{\bu}\limits^{\sst v}$ are permitted), as we want.

We will not construct a quasi-identity map $\Up:\cU\ra\cU$ with $\Phi_{\psi_1}=\Phi_{\psi_0}\ci\Up$, but we will give a first approximation. Define $\Up_0:\cU\ra\cU$ by
\e
\Up_0(\th):\al\longmapsto \th(\al)\cdot\exp\biggl[\frac{\Or(\ti N_0)I(\ti N_0)}{\md{\Iso([\ti N_0,\ti\imath_0])}}\cdot q^{\ga\cdot[\ti N_0]}\th([\ti N_0])\cdot \al\bu\de\biggr].
\label{ca7eq7}
\e
This is a quasi-identity map. Substitute \eq{ca7eq7} into \eq{ca7eq4} for $\psi_0$ to give an expression for $\Phi_{\psi_0}\ci\Up_0$. Then each term $\th([N_w])$ in \eq{ca7eq4} is replaced by
\begin{equation*}
\Up_0(\th)([N_w])=\th([N_w])\cdot\sum_{k=0}^\iy\frac{1}{k!}\biggl[\frac{\Or(\ti N_0)I(\ti N_0)}{\md{\Iso([\ti N_0,\ti\imath_0])}}\cdot q^{\ga\cdot[\ti N_0]}\th([\ti N_0])\cdot [N_w]\bu\de\biggr]^k.
\end{equation*}
Rewrite this expression as a sum over graphs by adding $k$ new vertices $v_1,\ldots,v_k$ with $N_{v_i}=\ti N_0$ and edges $\mathop{\bu}\limits^{\sst v_i}-\mathop{\bu}\limits^{\sst w}$ to $\Ga$ in \eq{ca7eq4}. Then compare the result to \eq{ca7eq6}, where the new edges $\mathop{\bu}\limits^{\sst v_i}-\mathop{\bu}\limits^{\sst w}$ with $N_{v_i}=\ti N_0$ become the directed edges $\mathop{\bu}\limits^{\sst v_i}\ra\mathop{\bu}\limits^{\sst w}$ in $S$. What we find is that $\Phi_{\psi_0}\ci\Up_0$ agrees with the sum of all terms in \eq{ca7eq6} such that for each edge $\mathop{\bu}\limits^{\sst v}\ra\mathop{\bu}\limits^{\sst w}$ in $S$, there are no other edges $\mathop{\bu}\limits^{\sst u}-\mathop{\bu}\limits^{\sst v}$ in $\Ga$. So $\Phi_{\psi_0}\ci\Up_0$ is a kind of leading-order approximation to~$\Phi_{\psi_1}$. 

The author expects that there is a formula for $\Up:\cU\ra\cU$ which generalizes \eq{ca7eq7}, and yields $\Phi_{\psi_1}=\Phi_{\psi_0}\ci\Up$ by comparison with \eq{ca7eq6}. This formula should look like \eq{ca7eq7} with $[\cdots]$ replaced by a graph sum similar to \eq{ca7eq6}, but over labelled {\it rooted\/} trees $(\Ga,[N_v,i_v]_{v\in V}),r$ with a distinguished `root vertex' $r\in V$ with $N_r=\ti N_0$, and including some combinatorial coefficients $C(\Ga,r,S)\in\Q$. Equation \eq{ca7eq7} gives the term when $\Ga=\mathop{\bu}\limits^{\sst r}$ has one vertex $r$ and no edges.

The case in which finitely many $N_t$ cross finitely many $C^3_j$ follows by composing the corresponding morphisms $\Up$ for each transition in order. Then we prove the general case by reducing the target $\cU$ modulo $q^A$ for $A>0$, so that only finitely many transitions are relevant for any fixed $A$, and letting $A\ra\iy$, as in part (a). This concludes our sketch proof of~(b).

For (c), let $\psi\in\cF_\ga$ be generic, and let $C^i_j,D$ and $\ti C^i_j,\ti D$ be alternative choices in Definition \ref{ca7def1}, yielding superpotentials $\Phi_\psi$ and $\ti\Phi_\psi$. First suppose that there are smooth, generic families $\hat C^i_j(t)$, $\hat D(t)$ for $t\in[0,1]$ with $\hat C^i_j(0)=C^i_j$, $\hat D(0)=D$, $\hat C^i_j(1)=\ti C^i_j$, $\hat D(1)=\ti D$. In \eq{ca7eq4} replace $D$ by $\hat D(t)$, and consider how the sum changes as $t$ increases through $[0,1]$. By a similar argument to (b),  this happens only when $N_v$ or $N_w$ intersect $\hat C^3_j(t_0)$ for some $j$ and~$t_0\in(0,1)$.

Now fixing the associative $N_v$ and deforming $C^3_j$ over $t\in[0,1]$ so that $N_v$ and $C^3_j$ intersect at $t=t_0$, is basically the same as fixing $C^3_j$ and deforming the associative $N_v$ over $t\in[0,1]$ so that $N_v$ and $C^3_j$ intersect at $t=t_0$, which is what we did in (b), and it has the same effect on the sum  \eq{ca7eq4}. Hence by (b), we see that $\ti\Phi_\psi=\Phi_\psi\ci\Up$ for some $\Up:\cU\ra\cU$ as in (b) in this case.

By a slightly more general argument, we can change the $C^i_j$ not by smooth deformation $C^i_j(t),$ $t\in[0,1]$ but by smooth bordism in $X$, which allows us to link any two choices $C^i_j,\ti C^i_j$, and we can also allow any choices of~$D,\ti D$.
\end{proof}

\subsection{Our main conjecture}
\label{ca75}

The next conjecture is the one of the main points of this paper.

\begin{conj} Let\/ $X$ be a compact, oriented\/ $7$-manifold, and\/ $\ga\!\in\! H^3_{\rm dR}(X;\R)$. Write $\cF_\ga$ for the set of closed\/ $4$-forms $\psi$ on $X$ such that there exists a closed\/ $3$-form $\vp$ on $X$ with\/ $[\vp]=\ga$ in $H^3_{\rm dR}(X;\R),$ for which $(X,\vp,\psi)$ is a TA-$G_2$-manifold, with the given orientation on $X$. 

Assuming Conjecture {\rm\ref{ca2conj2},} and making some arbitrary choices, and supposing we can find a good definition for the `unknown multiple cover contributions' in \eq{ca7eq4} to compensate for the singular behaviour in {\rm\S\ref{ca72}(F),} Definition\/ {\rm\ref{ca7def1}} gives a superpotential\/ $\Phi_\psi:\cU\ra \La_{>0}$ for each generic $\psi\in\cF_\ga,$ where $\cU=\Hom(H_3(X;\Z),1+\La_{>0}),$ as a smooth rigid analytic space over $\La$.

We conjecture that if different arbitrary choices yield\/ $\ti\Phi_\psi:\cU\ra \La_{>0}$ then $\ti\Phi_\psi=\Phi_\psi\ci\Up$ for $\Up:\cU\ra\cU$ a quasi-identity morphism, as in {\rm\S\ref{ca71}}.

We also conjecture that if\/ $\psi_0,\psi_1$ are generic elements in the same connected component of\/ $\cF_\ga,$ then $\Phi_{\psi_1}=\Phi_{\psi_0}\ci\Up$ for $\Up:\cU\ra\cU$ a quasi-identity morphism.

\label{ca7conj1}	
\end{conj}

Some support for this is provided by Theorem \ref{ca7thm1}, and its sketch proof. 

Conjecture \ref{ca7conj1} implies that any information we can extract from the superpotential $\Phi_\psi$, which is unchanged under reparametrizations $\Phi_\psi\mapsto\Phi_\psi\ci\Up$ for quasi-identity morphisms $\Up:\cU\ra\cU$, is unchanged under deformations of $\psi$ in $\cF_\ga$. As a shorthand we say that such information {\it depends only on $\Phi_\psi$ modulo quasi-identity morphisms}. Here are some examples:
\begin{itemize}
\setlength{\itemsep}{0pt}
\setlength{\parsep}{0pt}
\item[(i)] For $GW_{\psi,\al}$ as in \eq{ca7eq5}, let $A>0$ be least such that $GW_{\psi,\al}\ne 0$ for some $\al\in H_3(X;\Z)$ with $\ga\cdot\al=A$, or $A=\iy$ if $GW_{\psi,\al}=0$ for all $\al$. Then $A$ depends only on $\Phi_\psi$ modulo quasi-identity morphisms. Also, the values of $GW_{\psi,\al}$ for any $\al\in H_3(X;\Z)$ with $\ga\cdot\al=A$ depend only on $\Phi_\psi$ modulo quasi-identity morphisms.

Roughly, this says that {\it the numbers of associative $\Q$-homology spheres with least area $A$ in $X$ are deformation-invariant}. There could exist associatives with area less than $A$, but their signed weighted count is zero.
\item[(ii)] Whether or not $\Phi_\psi$ has a critical point in $\cU$ depends only on $\Phi_\psi$ modulo quasi-identity morphisms. Also, the set of critical points $\Crit(\Phi_\psi)$, as a set up to bijection rather than as a subset of $\cU$, depends only on $\Phi_\psi$ modulo quasi-identity morphisms, since if $\Up:\cU\ra\cU$ is a quasi-identity morphism then $\Up\vert_{\Crit(\Phi_\psi\ci\Up)}$ is a bijection $\Crit(\Phi_\psi\ci\Up)\ra\Crit(\Phi_\psi)$.
\end{itemize}
We develop (ii) further in our discussion of $G_2$ {\it quantum cohomology\/} in~\S\ref{ca76}.

For a TA-$G_2$-manifold $(X,\vp,\psi)$, the moduli spaces $\cM(\cN,\al,\psi)$ depend only on the 4-form $\psi$, and the superpotential $\Phi_\psi$ depends only on $\psi$ and the cohomology class $[\vp]=\ga$ of $\psi$ in $H^3_{\rm dR}(X;\R)$. 

Conjecture \ref{ca7conj1} allows us to switch the focus back to the 3-form $\vp$. By Proposition \ref{ca2prop2}(b), the set of $\psi$ compatible with a fixed good 3-form $\vp$ is a convex cone, and so is connected. Therefore by Conjecture \ref{ca7conj1}, $\Phi_\psi$ modulo quasi-identity morphisms depends only on $(X,\vp)$, and in fact only on $\vp$ up to deformations in a fixed cohomology class $\ga\in H^3_{\rm dR}(X;\R)$. As in Remark \ref{ca2rem3}, we think of $(X,\vp)$ as the analogue of a symplectic manifold $(Y,\om)$, and $\psi$ as the analogue of an almost complex structure $J$ on $Y$ compatible with $\om$. So $\Phi_\psi$ modulo quasi-identity morphisms is the analogue of a symplectic invariant.

\subsection{\texorpdfstring{$G_2$ quantum cohomology}{G₂ quantum cohomology}}
\label{ca76}

This section is motivated by some areas of Symplectic Geometry: quantum cohomology, as in McDuff and Salamon \cite{McSa}, Lagrangian Floer cohomology, as in Fukaya, Oh, Ohta and Ono \cite{Fuka2,FOOO}, and work of Fukaya \cite{Fuka1} on counting $J$-holomorphic discs with boundary in Lagrangians in a Calabi--Yau 3-fold.

The {\it quantum cohomology\/} $QH^*(Y;\La)$ of a compact symplectic manifold $(Y,\om)$ is isomorphic to the ordinary cohomology $H^*(Y;\La)$ over a Novikov ring $\La$, but it has a deformed cup product $*$ depending on the genus zero three-point Gromov--Witten invariants $GW_\al(\be_1,\be_2,\be_3)$ of $(Y,\om)$. 

If $L$ is a compact, oriented, relatively spin Lagrangian in $(Y,\om)$, there is a notion of {\it bounding cochain\/} $b$ for $L$ \cite{Fuka2,FOOO}, which is an object in the homological algebra of $L$ satisfying an equation involving counts of $J$-holomorphic discs in $Y$ with boundary in $L$. If a bounding cochain $b$ exists, we say $L$ has {\it unobstructed Lagrangian Floer cohomology}. We can form the Lagrangian Floer cohomology ring $HF^*((L,b),(L,b))$, which is a deformed version of $H^*(L;\La)$. In contrast to quantum cohomology, we need not have~$HF^*((L,b),(L,b))\cong H^*(L;\La)$.

When $(Y,\om)$ is a symplectic Calabi--Yau 3-fold and $L\subset Y$ is a graded Lagrangian, and $J$ a generic almost complex structure on $Y$ compatible with $\om$, we can reinterpret and extend work of Fukaya \cite{Fuka1} as follows, though Fukaya does not write things in this form. One should define a superpotential
\begin{equation*}
\Phi_J:\cU=\Hom\bigl(H_1(L;\Z),1+\La_{>0}\bigr)\longra\La_{>0}	
\end{equation*}
which counts $J$-holomorphic discs in $(Y,\om)$ with boundary in $L$. This $\Phi_J$ depends on some choices, and has some wall-crossing behaviour under deformation of $J$, as for $\Phi_\psi$ in \S\ref{ca73}--\S\ref{ca75}. Critical points of $\Phi_J$ correspond exactly to (equivalence classes of) bounding cochains $b$ for $L$.

As in \S\ref{ca61}, there is a strong analogy between counting $J$-holomorphic curves $\Si$ in a symplectic Calabi--Yau 3-fold $(Y,\om)$ with boundary $\pd\Si$ in a graded Lagrangian $L$, and counting associative 3-folds $N$ without boundary in a TA-$G_2$-manifold $(X,\vp,\psi)$. Following this analogy, we might hope that critical points $\th$ of $\Phi_\psi$ should be `bounding cochains' needed to define some kind of `$G_2$ quantum cohomology' $QH^*_\th(X;\La)$ deforming $H^*(X;\La)$, analogous to~$HF^*((L,b),(L,b))$.

\begin{dfn} Work in the situation of \S\ref{ca71}--\S\ref{ca73}, with $\psi\in\cF_\ga$ generic. Use the formula \eq{ca7eq5} for the superpotential $\Phi_\psi$. We call $(X,\vp,\psi)$ {\it obstructed\/} if $\Phi_\psi$ has no critical points in $\cU$, and {\it unobstructed\/} otherwise. 

Suppose $(X,\vp,\psi)$ is unobstructed, and choose a critical point $\th$ of $\Phi_\psi$. Define a $\La_{\ge 0}$-linear map $\d:H^3(X;\La_{\ge 0})\ra H^4(X;\La_{\ge 0})$ by
\e
\d(\be)=\sum_{\al\in H_3(X;\Z):\ga\cdot\al>0}GW_{\psi,\al}\, q^{\ga\cdot\al}\,\th(\al)\cdot\be(\al)\cdot \Pd(\al).
\label{ca7eq8}
\e
Here $\be(\al)$ comes from the pairing $H^3(X;\La_{\ge 0})\t H_3(X;\Z)\ra\La_{\ge 0}$ and $\Pd(\al)$ from the Poincar\'e duality isomorphism $\Pd:H_3(X;\Z)\ra H^4(X;\Z)$, and the sum in \eq{ca7eq8} converges in the topology on $H^4(X;\La_{\ge 0})$ induced by that on $\La_{\ge 0}$.

We can interpret $\d$ as contraction with the Hessian $\Hess_\th(\Phi_\psi)$ of $\Phi_\psi$ at $\th$.

Now define the $G_2$-{\it quantum cohomology groups\/} $QH^k_\th(X;\La_{\ge 0})$ for $k\ge 0$ by
\begin{equation*}
QH^k_\th(X;\La_{\ge 0})=\begin{cases} H^k(X;\La_{\ge 0}), & k\ne 3,4,\\
\Ker\bigl[\d:H^3(X;\La_{\ge 0})\ra H^4(X;\La_{\ge 0})\bigr], & k=3, \\
\Coker\bigl[\d:H^3(X;\La_{\ge 0})\ra H^4(X;\La_{\ge 0})\bigr], & k=4. \end{cases}
\end{equation*}
Define a product $*:QH^k_\th(X;\La_{\ge 0})\t QH^l_\th(X;\La_{\ge 0})\ra QH^{k+l}_\th(X;\La_{\ge 0})$, written $\de*\ep\in QH^{k+l}_\th(X;\La_{\ge 0})$ for $\de\in QH^k_\th(X;\La_{\ge 0})$ and $\ep\in QH^l_\th(X;\La_{\ge 0})$, by:
\begin{itemize}
\setlength{\itemsep}{0pt}
\setlength{\parsep}{0pt}
\item[(i)] If $(k,l)$ are one of
\begin{align*}
&(0,0),(0,1),(0,2),(0,5),(0,6),(0,7),(1,0),(1,1),(1,5),(1,6),(2,0),\\
&(2,3),(2,5),(3,2),(3,3),(5,0),(5,1),(5,2),(6,0),(6,1),(7,0),	
\end{align*}
then $\de*\ep=\de\cup\ep$, as in these cases either $QH^*_\th(X;\La_{\ge 0})=H^*(X;\La_{\ge 0})$ in degrees $k,l,k+l$, or $QH^3_\th(X;\La_{\ge 0})\subseteq H^3(X;\La_{\ge 0})$ for $k=3$ or $l=3$.
\item[(ii)] If $(k,l)=(0,3)$ then $\de*\ep=\de\cup\ep$, where $\ep\in\Ker\d\subseteq H^3(X;\La_{\ge 0})$ implies that $\de\cup\ep\in\Ker\d$. Similarly for~$(k,l)=(3,0)$.
\item[(iii)] If $(k,l)=(0,4)$ then $\de*(\ep+\Im\d)=(\de\cup\ep)+\Im\d$, where $\ep\in H^4(X;\La_{\ge 0})$. Similarly for~$(k,l)=(4,0)$.
\item[(iv)] If $(k,l)=(1,2)$ then $\de*\ep=\de\cup\ep$. To show this is well defined we must prove that $\de\cup\ep\in\Ker\d\subseteq H^3(X;\La_{\ge 0})$ for all $\de\in H^1(X;\La_{\ge 0})$ and $\ep\in H^2(X;\La_{\ge 0})$. Now if $i:N\ra X$ is an immersed associative $\Q$-homology sphere with $[N]=\al\in H_3(X;\Z)$ then $(\de\cup\ep)\cdot\al=(i^*(\de)\cup i^*(\ep))\cdot[N]=0$, since $H^1(N;\Q)=H^2(N;\Q)=0$ as $N$ is a $\Q$-homology 3-sphere, and $i^*(\de)\in H^1(N;\Q)$, $i^*(\ep)\in H^2(N;\Q)$. Since $GW_{\psi,\al}$ counts associative $\Q$-homology 3-spheres in class $\al$, we have $(\de\cup\ep)\cdot\al=0$ if $GW_{\psi,\al}\ne 0$. Hence from \eq{ca7eq8} we see that $\de\cup\ep\in\Ker\d$. Similarly for~$(k,l)=(2,1)$.
\item[(v)] If $(k,l)$ is (1,3), (2,2) or (3,1) then $\de*\ep=\de\cup\ep+\Im\d$.
\item[(vi)] If $(k,l)=(1,4)$ or (2,4) then $\de*(\ep+\Im\d)=\de\cup\ep$. To show this is well-defined we must show that if $\ep+\Im\d=\ep'+\Im\d$ then $\de\cup\ep=\de\cup\ep'$. As $\ep'=\ep+\d\ze$ for $\ze\in H^3(X;\La_{\ge 0})$, it is enough to show that $\de\cup\d\ze=0$. From \eq{ca7eq8}, $\d\ze$ is a linear combination of classes $\Pd(\al)$ for $\al\in H_3(X;\Z)$ with $GW_{\psi,\al}\ne 0$. As in (iv), we have $\de\cup\Pd(\al)=0$ if $\de\in H^1(X;\La_{\ge 0})$ or $\de\in H^2(X;\La_{\ge 0})$, since $\al$ is represented by a $\Q$-homology 3-sphere, so $\de\cup\d\ze=0$. Similarly for $(k,l)=(4,1)$ or (4,2).
\item[(vii)] If $(k,l)=(3,4)$ then $\de*(\ep+\Im\d)=\de\cup\ep$ for $\de\in\Ker\d\subseteq H^3(X;\La_{\ge 0})$ and $\ep\in H^4(X;\La_{\ge 0})$. As in (vi), to show this is well-defined we must show that $\de\cup\d\ze=0$ for $\ze\in H^3(X;\La_{\ge 0})$. But from \eq{ca7eq8} we can prove that $\eta\cup\d\ze=\ze\cup\d\eta$ for any $\eta,\ze\in H^3(X;\La_{\ge 0})$, because $\Hess_\th(\Phi_\psi)$ is a symmetric form. Thus $\de\cup\d\ze=0$ as $\d\de=0$. Similarly for~$(k,l)=(4,3)$.
\item[(viii)] If $k+l>7$ then $\de*\ep=0$ automatically.
\end{itemize}
Since $\cup$ is associative and supercommutative, we see that $*$ is too.
\label{ca7def2}	
\end{dfn}

If we assume Conjecture \ref{ca7conj1}, then $G_2$ quantum cohomology $QH^*_\th(X;\La_{\ge 0})$ will be unchanged under deformations of $\psi$, in the same sense in which Lagrangian Floer cohomology $HF^*((L,b),(L,b))$ is independent of $J$. If $\psi_0,\psi_1$ are generic in the same connected component of $\cF_\ga$, Conjecture \ref{ca7conj1} gives $\Up:\cU\ra\cU$ with $\Phi_{\psi_1}=\Phi_{\psi_0}\ci\Up$. Then $\Up$ maps critical points $\th_1$ of $\Phi_{\psi_1}$ bijectively to critical points $\th_0$ of $\Phi_{\psi_1}$, and using the derivative $\d_{\th_1}\Up$ of $\Up$ at $\th_1$ we can define a $\La_{\ge 0}$-algebra isomorphism~$QH^*_{\th_1}(X;\La_{\ge 0})\ra QH^*_{\th_0}(X;\La_{\ge 0})$.

There should also be a way to define an $A_\iy$-algebra whose cohomology is $QH^*_\th(X;\La_{\ge 0})$, deforming the cochain cdga for $H^*(X;\La_{\ge 0})$, using similar ideas to Fukaya et al.\ \cite{Fuka2,FOOO}. In this definition we should use the fact that we count only associative $\Q$-homology 3-spheres $N\subset X$ in the following way. Consider the 6-cycle in $N\t N\t N$
\begin{equation*}
C=\bigl\{(x,x',x'):x,x'\in N\bigr\}\!+\!\bigl\{(x',x,x'):x,x'\in N\bigr\}\!+\!\bigl\{(x',x',x):x,x'\in N\bigr\}.
\end{equation*}
Since $N$ is a $\Q$-homology 3-sphere we have $[C]=0$ in $H_6(N\t N\t N;\Q)$, so there is a 7-cycle $D$ on $N\t N\t N$ with $\pd D=C$. The cochain-level version of multiplication $*$ should involve choosing such a 7-cycle $D$ for each associative $\Q$-homology sphere $N$ in the count.

The author does not know whether this $G_2$ quantum cohomology is actually interesting. It seems likely to play some r\^ole in M-theory, at least.

\subsection{Generalizations}
\label{ca77}

Here are some ways in which the picture of \S\ref{ca71}--\S\ref{ca76} can be extended.
\smallskip

\noindent{\bf Including a C-field.} Take the field $\F$ used to define $\La$ in \S\ref{ca71} to be $\F=\C$. Choose $C\in H^3(X;\R)/2\pi H^3(X;\Z)$. Then we can generalize the formulae \eq{ca7eq4}--\eq{ca7eq5} defining $\Phi_\psi$ by replacing $q^{\ga\cdot[N_v]}$ by $q^{\ga\cdot[N_v]}e^{i C\cdot [N_v]}$, so that \eq{ca7eq5} becomes
\begin{equation*}
\Phi_\psi(\th)=\ts\sum_{\al\in H_3(X;\Z):\ga\cdot\al>0}GW_{\psi,\al}\, q^{\ga\cdot\al}\,e^{iC\cdot\al}\,\th(\al).
\end{equation*}
Here as $C\in H^3(X;\R)/2\pi H^3(X;\Z)$ and $\al\in H_3(X;\Z)$, the product $C\cdot\al$ lies in $\R/2\pi\Z$, so that $e^{iC\cdot\al}$ is well defined. `C-fields' $C$ of this kind are natural in the M-theory of $G_2$-manifolds, and have the effect of complexifying the moduli space of $G_2$-manifolds, with $[\vp]+iC$ in the complex manifold~$H^3(X;\C/2\pi i\Z)$. 

\smallskip

\noindent{\bf Varying the cohomology class $[\vp]$.} So far we have worked with TA-$G_2$-manifolds $(X,\vp,\psi)$ for which the $[\vp]=\ga\in H^3_{\rm dR}(X;\R)$ is fixed. Here is a way to allow $[\vp]$ to vary. Let us regard the 4-form $\psi$ as fixed. Then Proposition \ref{ca2prop2}(a) gives an open convex cone $\cK_{X,\psi}$ in $H^3_{\rm dR}(X;\R)$, of cohomology classes $[\vp]$ of 3-forms $\vp$ such that $(X,\vp,\psi)$ is a TA-$G_2$-manifold. 

We can then then extend $\Phi_\psi$ in \eq{ca7eq7} to a map
\begin{equation*}
\hat\Phi_\psi:\cK_{X,\psi}\t\cU\longra\La_{>0},
\end{equation*}
which maps $(\ga,\th)$ in $\cK_{X,\psi}\t\cU$ to $\Phi_\psi$ in \eq{ca7eq7} computed using $[\vp]=\ga$. Over $\F=\R$, we can regard $\cK_{X,\psi}\t\cU$ as a rigid analytic space; it may be possible to glue the charts $\cK_{X,\psi_0}\t\cU$, $\cK_{X,\psi_1}\t\cU$ over $\cK_{X,\psi_0}\cap\cK_{X,\psi_1}$ for different $\psi_0,\psi_1$, using the morphisms $\Psi:\cU\ra\cU$ in Conjecture \ref{ca7conj1}, to get a $\cU$-bundle over a larger open subset of $H^3(X;\R)$, upon which a superpotential $\hat\Phi$ is defined.
\smallskip

\noindent{\bf Noncompact $G_2$-manifolds.} We can consider TA-$G_2$-manifolds $(X,\vp,\psi)$ with $X$ noncompact, if we have some control on the noncompact ends of $X$ -- some kind of convexity at infinity -- which prevents associative 3-folds from escaping to infinity in $X$, and so changing the numbers of associatives.
\smallskip

\noindent{\bf Counting associatives $N$ with $b^1(N)>0$.} It is tempting to try and modify \eq{ca7eq4} to count `higher genus' associative 3-folds $N$ with $g=b^1(N)>0$. The author does not know a way to do this in general, which is invariant under transitions of type \S\ref{ca72}(C). One possibility might be to try and count associatives $i:N\ra X$ where $N$ is not a $\Q$-homology 3-sphere, but $i_*:H_2(N;\Q)\ra H_2(X;\Q)$ is injective, as such $N$ are not affected by transitions~\S\ref{ca72}(C).

\section{\texorpdfstring{Remarks on counting $G_2$-instantons}{Remarks on counting G₂-instantons}}
\label{ca8}

We discussed $G_2$-instantons on TA-$G_2$-manifolds $(X,\vp,\psi)$ in \S\ref{ca24}--\S\ref{ca25} above. Donaldson and Segal \cite[\S 6.2]{DoSe} proposed a conjectural programme to define invariants counting $G_2$-instantons, which would hopefully be unchanged under deformations of $(\vp,\psi)$, and would be analogues of Donaldson--Thomas invariants of Calabi--Yau 3-folds \cite{Joyc22,JoSo}. The programme is currently under investigation by Menet, Nordstr\"om, S\'a Earp, Walpuski, and others~\cite{MNS,SaEa,SEWa,Walp1,Walp2,Walp3,Walp4}. 

As in \cite[\S 6]{DoSe}, to complete the Donaldson--Segal programme and define invariants of $(X,\vp,\psi)$ unchanged under deformations of $\psi$ will require the inclusion of `compensation terms' counting solutions of some equation on associative 3-folds $N$ in $X$, to compensate for bubbling of $G_2$-instantons on associative 3-folds. So counting $G_2$-instantons, and counting associative 3-folds, are intimately linked. 

We now discuss several aspects of this programme, drawing on the ideas of \S\ref{ca3}--\S\ref{ca7}. Section \ref{ca82} makes a proposal for how to define canonical orientations for $G_2$-instanton moduli spaces, based on the ideas in \S\ref{ca3} on orienting associative moduli spaces. Section \ref{ca84} gives two `thought-experiments' describing ways in which Donaldson--Segal's proposed invariants could change under deformations of $(\vp,\psi)$. Finally, \S\ref{ca85} suggests a way (not yet complete) to modify the Donaldson--Segal programme to (hopefully) fix these problems.

\subsection{The Donaldson--Segal programme}
\label{ca81}

Suppose $X$ is a compact 7-manifold, and $(\vp,\psi)$ a generic TA-$G_2$-structure on $X$. Let $G$ be a compact Lie group, and $\pi:P\ra X$ a principal $G$-bundle. Consider the moduli space $\cM(P,\psi)$ of $G_2$-instantons on $X$, as in \S\ref{ca24}--\S\ref{ca25}.

By analogy with Donaldson invariants of oriented 4-manifolds $M$ \cite{DoKr}, which count moduli spaces of instantons on $M$, and with Donaldson--Thomas invariants of Calabi--Yau 3-folds $Y$ \cite{Joyc22,JoSo}, which can be heuristically understood as counting Hermitian--Yang--Mills connections on $Y$, Donaldson and Segal \cite[\S 6]{DoSe} want to define invariants of $(X,\vp,\psi)$ by counting moduli spaces~$\cM(P,\psi)$.

Donaldson and Segal expect \cite[\S 4.1]{DoSe} that when $\psi$ is generic $\cM(P,\psi)$ will be a compact 0-manifold, that is, a finite set, and one can define an orientation on the moduli space $\Or:\cM(P,\psi)\ra\{\pm 1\}$ (compare \S\ref{ca3}), though they do not give details. Then a first approximation to the invariants they want is
\e
DS_0(P,\psi)=\ts\sum_{[A]\in \cM(P,\psi)}\Or([A])\in\Z.	
\label{ca8eq1}
\e

They explain \cite[\S 6.1]{DoSe} that $DS_0(P,\psi)$ should in general not be unchanged under deformations of $\psi$, as there are index one singularities of $G_2$-instantons which can change the moduli spaces $\cM(P,\psi)$. They expect that the typical way moduli spaces can change under deformations is as follows:

\begin{ex} Let $(\vp_t,\psi_t)$, $t\in[0,1]$ be a generic 1-parameter of TA-$G_2$-structures on $X$. Suppose that for some $t_0\in(0,1)$ there exists a connection $A_t$ on $P$ for $t\in[0,t_0)$ which is an unobstructed $G_2$-instanton on $(X,\vp_t,\psi_t)$, and depends smoothly on $t$. As $t\ra t_0$, the $G_2$-instanton $A_t$ approaches a singular limit, in which the curvature $F_{A_t}$ of $A_t$ concentrates around a compact associative 3-fold $N_{t_0}$ in~$(X,\vp_{t_0},\psi_{t_0})$. 

This singularity should be `removable'. That is, there is another principal $G$-bundle $P'\ra X$ with a $G_2$-instanton connection $A_{t_0}'$ on $(X,\vp_{t_0},\psi_{t_0})$, such that there is an isomorphism of principal $G$-bundles $P\vert_{X\sm N}\cong P'\vert_{X\sm N}$ on $X\sm N$, and up to gauge transformations, $A_t\vert_{X\sm N}$ converges to $A_{t_0}'\vert_{X\sm N}$ as $t\ra t_0$ on any compact subset of $X\sm N$. As $t$ converges to $t_0$, the connection $A_t$ near $N$ should resemble a family of instantons with group $G$ and charge $c_2=k$ on the $\R^4$ normal spaces $\nu_x$ to $N$ in $X$ at $x\in N$, concentrated near 0 in $\nu_x$. When $G=\SU(2)$, the second Chern classes $c_2(P),c_2(P')$ are related by
\begin{equation*}
c_2(P)=c_2(P')+k\cdot \Pd([N])\in H^4(X;\Z).
\end{equation*}

Now the moduli spaces of instantons on $\R^4$ are well understood, and can be described by the ADHM construction. Donaldson and Segal \cite[\S 6.1]{DoSe} define a bundle $\ul{M\!}\,\ra N$ whose fibre at $x\in N$ is the moduli space $\cM^G(\nu_x,k)$ of instantons on $\nu_x$ with group $G$ and charge $k$, with framing at infinity in $\nu_x$ depending on $P'\vert_N$. Using results of Haydys, they define an equation on smooth sections $\ul s:N\ra\ul{M\!}\,$ which they call the {\it Fueter equation}, which depends on $A'\vert_N$, and explain that the local model near $N$ for $A_t$ as $t\ra t_0$ should be written in terms of a solution $\ul s$ of the Fueter equation.  

They conjecture that given a $G_2$-instanton $(P',A')$ on $(X,\vp_{t_0},\psi_{t_0})$, a compact associative $N$ in $(X,\vp_{t_0},\psi_{t_0})$, and a solution $\ul s:N\ra\ul{M\!}\,$ of the Fueter equation constructed from $(P',A')\vert_N$ for charge $k$, it should be possible to find a smooth 1-parameter family of TA-$G_2$-manifolds $(X,\vp_t,\psi_t)$, $t\in[0,1]$ including $(X,\vp_{t_0},\psi_{t_0})$, and a smooth family of $G_2$-instantons $(P,A_t)$ on $(X,\vp_t,\psi_t)$ for $t\in[0,t_0)$, which bubble on $N$ as $t\ra t_0$ to recover $(P',A'),\ul s$ as above. This conjecture has now been proved by Walpuski~\cite{Walp2}.
\label{ca8ex1}
\end{ex}

When $G=\SU(2)$ and $k=1$, Donaldson and Segal \cite[\S 6.1]{DoSe} describe the bundle $\ul{M\!}\,\ra N$ and the Fueter equation for sections $\ul s:N\ra\ul{M\!}\,$ more explicitly:

\begin{ex} Continue in Example \ref{ca8ex1}, but fix $G=\SU(2)$ and the charge $k$ of instantons bubbling at $N$ as $t\ra t_0$ to be $k=1$. Also suppose that the associative 3-fold $N$ in $(X,\vp_{t_0},\psi_{t_0})$ is unobstructed, in the sense of~\S\ref{ca26}.

The moduli space of instantons on $\R^4$ with group $\SU(2)$ and charge 1 is $\cM^{\SU(2)}(\R^4,k)\cong [(\R^4\sm\{0\})/\{\pm 1\}]\t\R^4$. The corresponding bundle $\ul{M\!}\,\ra N$ is
\begin{equation*}
\ul{M\!}\,\cong [(\bS_{P'}\sm\{0\})/\{\pm 1\}]\t_N\nu.
\end{equation*}
Here we choose some spin structure $\si$ on $N$ and write $\bS\ra N$ for the spin bundle over $N$ associated to $\si$, which has fibre $\H\cong\R^4$. Then $\bS_{P'}=(\bS\t_N {P'\vert_N})/\SU(2)$ is the spin bundle on $N$ twisted by $P'\vert_N$, and $\bS_{P'}\sm\{0\}$ is the complement of the zero section in $\bS_{P'}$, so that $\bS_{P'}$, $\bS_{P'}\sm\{0\}$ and $(\bS_{P'}\sm\{0\})/\{\pm 1\}$ are bundles over $N$ with fibres $\R^4,\R^4\sm\{0\}$ and~$(\R^4\sm\{0\})/\{\pm 1\}$. 

Dividing by $\{\pm 1\}$ means that $(\bS_{P'}\sm\{0\})/\{\pm 1\}$ is independent of the choice of spin structure $\si$ on $N$. However, any section of $(\bS_{P'}\sm\{0\})/\{\pm 1\}$ lifts to a section of $\bS_{P'}\sm\{0\}$ for $\bS_{P'}$ defined using a unique spin structure $\si$. Thus, sections $\ul s:N\ra\ul{M\!}\,$ correspond to triples $(\si,\{\pm \ul s_1\},\ul s_2)$ of a spin structure $\si$ on $N$, a nonvanishing section $\ul s_1$ of the twisted spin bundle $\bS_{P'}\ra N$ defined using $\si$ and $P'\vert_N$, and a section $\ul s_2$ of~$\nu\ra N$.

The Fueter equation on $\ul s$ is then equivalent to $\bD_{P',A'}\ul s_1=0$, $\bD\ul s_2=0$, where $\bD_{P',A'}:\Ga^\iy(\bS_{P'})\ra\Ga^\iy(\bS_{P'})$ is the twisted Dirac operator for $(P'\vert_N,A'\vert_N)$, and $\bD:\Ga^\iy(\nu)\ra\Ga^\iy(\nu)$ is as in Theorem \ref{ca2thm2}. But by assumption $N$ is unobstructed, so $\Ker\bD=0$, and $\ul s_2=0$. Therefore, the conclusion is that solutions $\ul s$ of the Fueter equation correspond to pairs $(\si,\ul s_1)$, where $\si$ is a spin structure on $N$, and $\ul s_1$ is a non-vanishing solution of the twisted Dirac equation $\bD_{P',A'}\ul s_1=0$ for the $\SU(2)$-connection $(P'\vert_N,A'\vert_N)$ on $N$ with spin structure $\si$, where $\ul s_1$ only matters up to sign~$\pm\ul s_1$.
\label{ca8ex2}	
\end{ex}

Donaldson and Segal's proposal \cite[\S 6.2]{DoSe} is to try to modify \eq{ca8eq1} to define invariants, for TA-$G_2$-manifolds $(X,\vp,\psi)$ with $\psi$ generic:
\e
DS(P,\psi)=\sum_{[A]\in \cM(P,\psi)\!\!\!\!\!\!\!\!\!\!\!\!\!\!\!\!\!}\Or([A])+\sum_{\begin{subarray}{l} \text{$(P',A'),N,k$: $(P',A')$ $G_2$-instanton on $(X,\vp,\psi)$}\\
\text{with group $G$, up to gauge equivalence,} \\ \text{$N\ne\es$ compact associative in $(X,\vp,\psi)$,}\\ \text{$k\ge 1$, $P=P'+$charge $k$ modification along $N$} 
\end{subarray}\!\!\!\!\!\!\!\!\!\!\!\!\!\!\!\!\!\!\!\!\!\!\!\!\!\!\!\!\!\!\!\!\!\!\!\!\!\!\!\!\!\!\!\!\!\!\!\!\!\!\!\!\!\!\!\!\!\!\!\!\!\!\!\!\!\!\!\!\!\!\!\!}w\bigl((P',A'),N,k\bigr).
\label{ca8eq2}
\e
Here $w((P',A'),N,k)$ is some `compensation term' which they do not define, but the crucial point is that it must be chosen so that $DS(P,\psi)$ is unchanged under deformations of $(X,\vp,\psi)$ in 1-parameter families $(X,\vp_t,\psi_t)$, $t\in[0,1]$. So in Example \ref{ca8ex1}, the first term of \eq{ca8eq2} changes by $\pm 1$ as $t$ crosses $t_0$ and $[A_t]$ disappears from $\cM(P,\psi)$, and we expect $w((P',A'),N,k)$ for $(P,A'),N,k$ as in Example \ref{ca8ex1} to change by $\mp 1$ as $t$ crosses $t_0$ to compensate.

When $G=\SU(2)$ and $k=1$ Donaldson and Segal \cite[\S 6.2]{DoSe} suggest taking $w((P',A'),N,1)=\pm\ha$, where the sign is defined by using spectral flow as in \S\ref{ca32}. This is explained by Walpuski \cite[\S 6.2]{Walp3}. Haydys and Walpuski \cite[\S 1]{HaWa} give a different proposal for $w((P',A'),N,1)$, which we discuss in~\S\ref{ca85}.

\subsection{\texorpdfstring{Canonical orientations for moduli of $G_2$-instantons}{Canonical orientations for moduli of G₂-instantons}}
\label{ca82}

As in \S\ref{ca81}, there are close connections between moduli spaces of $G_2$-instantons and of associative 3-folds in $(X,\vp,\psi)$. So our method in \S\ref{ca32} for defining canonical orientations on associative moduli spaces $\cM(\cN,\al,\psi)$ in $(X,\vp,\psi)$, having chosen a flag structure $F$ on $X$, might have an analogue for defining canonical orientations on $G_2$-instanton moduli spaces.

\begin{conj} Let\/ $(X,\vp,\psi)$ be a compact TA-$G_2$-manifold and\/ $\pi:P\ra X$ a principal\/ $\SU(2)$-bundle, and write $\cM(P,\psi)$ for the moduli space of irreducible\/ $G_2$-instanton connections $A$ on $(X,\vp,\psi)$ up to gauge equivalence. We expect\/ $\cM(P,\psi)$ to be a smooth\/ $0$-manifold if\/ $\psi$ is generic, and an m-Kuranishi space of virtual dimension $0$ in general, as for Conjectures\/ {\rm\ref{ca2conj1}} and\/~{\rm\ref{ca2conj2}}.

Choose a flag structure $F$ for $X,$ as in {\rm\S\ref{ca31}}. Then there should be a way to define canonical orientations for the moduli spaces $\cM(P,\psi),$ as manifolds or m-Kuranishi spaces, which are well behaved under deformations of\/~$(X,\vp,\psi)$. 

If\/ $F,F'$ are flag structures on $X$ then Proposition\/ {\rm\ref{ca3prop3}(b)} gives a morphism $\ep:H_3(X;\Z)\ra\{\pm 1\}$ satisfying \eq{ca3eq8}. Let $\ep':H^4(X;\Z)\ra\{\pm 1\}$ correspond to $\ep$ under the Poincar\'e duality isomorphism $H_3(X;\Z)\cong H^4(X;\Z)$. Then the orientations on $\cM(P,\psi)$ coming from $F$ and\/ $F'$ differ by a factor\/~$\ep'\ci c_2(P)$.

\label{ca8conj1}	
\end{conj}

Here is how the author expects a proof of Conjecture \ref{ca8conj1} to go. We follow the method of Donaldson and Kronheimer \cite[\S 5.4 \& \S 7.1.6]{DoKr} for constructing orientations on moduli spaces $\cM(P,g)$ of instanton connections on a principal $\SU(2)$-bundle $P\ra M$ over a compact, oriented, generic Riemannian 4-manifold $(M,g)$. There are three main steps in their method:
\begin{itemize}
\setlength{\itemsep}{0pt}
\setlength{\parsep}{0pt}
\item[(a)] They define the orientation as a structure on the infinite-dimensional family $\cB$ of all connections $A$ on $P$, modulo gauge, not just on the finite-dimensional submanifold $\cM(P,g)\subset\cB$. Here $\cB$ is connected, and can be described using homotopy theory.
\item[(b)] In \cite[\S 5.4]{DoKr}, by considering loops $\cS^1$ in $\cB$, they show that $\cB$ is orientable. There are then two possible orientations on $\cB$, as $\cB$ is connected.
\item[(c)] In \cite[\S 7.1.6]{DoKr}, when $c_2(P)=k\ge 0$ in $H^4(M;\Z)\cong\Z$, they fix the orientation on $\cB$ by defining it near a connection $A$ on $P$ which is trivial away from $p_1,\ldots,p_k$ in $M$, and which near each $p_i$ approximates a standard $\SU(2)$-instanton on $\R^4$ with $c_2=1$, with curvature concentrated near~0. 
\end{itemize}

Orientations for moduli spaces $\cM(P,\psi)$ of $G_2$-instantons on $(X,\vp,\psi)$ are discussed by Donaldson and Segal \cite[\S 4.1]{DoSe}, and in more detail by Walpuski \cite[\S 6.1]{Walp3}. Walpuski does the analogues of (a),(b) above, where for (b) he shows \cite[Prop.~6.3]{Walp3} that $\cB$ is orientable for moduli spaces of $G_2$-instantons with gauge group $\SU(r)$ for $r\ge 2$. But he does not carry out step (c), instead choosing one of the two orientations on $\cB$ arbitrarily.

We propose that our ideas using flag structures may be used to complete step (c). The idea would be that given a principal $\SU(2)$-bundle $P\ra X$ with $c_2(P)=\be\in H^4(X;\Z)$, we would let $\al\in H_3(X;\Z)$ correspond to $\be$ under Poincar\'e duality, and choose a compact, oriented, embedded 3-submanifold $N$ in $X$ with $[N]=\al\in H_3(X;\Z)$. Here $N$ is not required to be associative. Then we should consider a connection $A$ on $P$ which is trivial away from $N$, and near $N$ approximates a family of small standard $\SU(2)$-instantons with $c_2=1$ on the $\R^4$ fibres of the normal bundle $\nu\ra N$, as in \cite[\S 6.1]{DoSe} for $N$ associative.

The orientation for $\cB$ should then be determined by giving $A$ the orientation $(-1)^{\SF(L_t:t\in[0,1])}F(N,f)$, where $F$ is the flag structure on $X$, and $\SF(L_t:t\in[0,1])$ is the spectral flow between an elliptic operator $L_0$ which depends on a choice of flag $f$ for $N$ at $t=0$, and the linearization $L_1$ of the $G_2$-instanton equation at $A$ at $t=1$, where we suppose $L_1$ is an isomorphism.

\subsection{\texorpdfstring{$P'$-flags, and canonical $P'$-flags}{P'-flags, and canonical P'-flag structures}}
\label{ca83}

\begin{dfn} Let $(X,\vp,\psi)$ be a compact TA-$G_2$-manifold, and $(P',A')$ a $G_2$-instanton on $X$ with structure group $\SU(2)$, and $N$ a compact, oriented 3-dimensional submanifold in $X$ (usually associative), and $\si$ a spin structure on $N$. Then as in \S\ref{ca8ex2} we define the twisted spin bundle $\bS_{P'}\ra N$ and the twisted Dirac operator $\bD_{P',A'}:\Ga^\iy(\bS_{P'})\ra\Ga^\iy(\bS_{P'})$ using $\si$ and $(P'\vert_N,A'\vert_N)$.

We now repeat parts of \S\ref{ca31}--\S\ref{ca32} with $\bS_{P'}\ra N$ in place of $\nu\ra N$. As in Definition \ref{ca3def1}, let $s,s'\in\Ga^\iy(\bS_{P'})$ be nonvanishing sections. Write
$0:N\ra \bS_{P'}$ for the zero section, and $\ga:[0,1]\t N\ra\in\Ga^\iy(\bS_{P'})$ for the map $\ga:(t,x)\mapsto (1-t)s(x)+ts'(x)$. Define $d(s,s')=0(N)\bu\ga([0,1]\t N)\in\Z$ .

Define a {\it $P'$-flag on\/} $N$ to be an equivalence class $[s]$ of nonvanishing $s\in\Ga^\iy(\bS_{P'})$, where $s,s'$ are equivalent if $d(s,s')=0$. Write $\Flag_{P'}(N)$ for the set of all $P'$-flags $[s]$ on $N$. For $[s],[s']\in\Flag_{P'}(N)$ we define $d([s],[s'])=d(s,s')\in\Z$ for any representatives $s,s'$ for $[s],[s']$. For any $[s]\in\Flag_{P'}(N)$ and any $k\in\Z$, there is a unique $[s']\in\Flag_{P'}(N)$ with $d([s],[s'])=k$, and we write $[s']=[s]+k$. This gives a natural action of $\Z$ on $\Flag_{P'}(N)$, making $\Flag_{P'}(N)$ into a $\Z$-torsor.

Following Definition \ref{ca3def5}, let $[s]$ be a $P'$-flag, and choose a representative $s$ of unit length. There is then a unique isomorphism $\bS_{P'}\cong \La^0T^*N\op\La^2 T^*N$ which identifies $s$ with $1\op 0$ in $\Ga^\iy(\La^0T^*N\op\La^2 T^*N)$, and identifies the symbols of $\bD_{P',A'}$ and $\d*+*\d$. Thus as in \eq{ca3eq10} we have $\bD_{P',A'}\cong \d*+*\d+B$, for $B$ of degree 0 as in \eq{ca3eq11}. Define a family of first order operators $A_t$, $t\in[0,1]$ as in \eq{ca3eq12} by $A_t=\d*+*\d+tB$. Then $A_0=\d*+*\d$ in \eq{ca3eq9}, and $A_1\cong\bD_{P',A'}$ under our isomorphism $\La^0T^*N\op\La^2 T^*N\cong\bS_{P'}$. Thus as in Definition \ref{ca3def4} we have the spectral flow~$\SF(A_t:t\in[0,1])\in\Z$.

As in Definition \ref{ca3def5}, there is a unique $P'$-flag $f^{P'}_N$ or $f^{P',A'}_N$ on $N$, called the {\it canonical\/ $P'$-flag\/} of $N$, such that $\SF(A_t:t\in[0,1])=0$ for $A_t:t\in[0,1]$ constructed using $s\in f_N^{P'}$. It has the property that for any $P'$-flag $[s]$ for $N$ and family $A_t:t\in[0,1]$ constructed from $s\in[s]$ as above, we have
\begin{equation*}
f_N^{P'}=[s]+\SF(A_t:t\in[0,1]).
\end{equation*}

\label{ca8def1}	
\end{dfn}

Canonical $P'$-flags $f_N^{P'}$ are related to the problem of defining the weight function $w((P',A'),N,k)$ in \eq{ca8eq2} when $G=\SU(2)$ and $k=1$, so that we can use Example \ref{ca8ex2}. Suppose we are given a generic 1-parameter family of TA-$G_2$-manifolds $(X,\vp_t,\psi_t)$, $t\in[0,1]$, and corresponding 1-parameter families $(P',A'_t)$, $t\in[0,1]$ of unobstructed $G_2$-instantons in $(X,\vp_t,\psi_t)$, and $N_t$, $t\in[0,1]$ of unobstructed associative 3-folds in $(X,\vp_t,\psi_t)$. Then we have a 1-parameter family of twisted Dirac operators $\bD_{P',A'_t}$ for $t\in[0,1]$ on~$N_t$.

According to the Donaldson--Segal--Walpuski picture, for generic $t\in[0,1]$ we have $\Ker\bD_{P',A'_t}=0$, but for isolated $t_0\in[0,1]$ we may have $\Ker\bD_{P',A'_{t_0}}\ne 0$, and then we create or destroy a new $G_2$-instanton $(P,A_t)$ as $t$ increases through $t_0$ in $[0,1]$, as in Examples \ref{ca8ex1} and \ref{ca8ex2}. This happens when an eigenvalue of $\bD_{P',A'_t}$ passes through 0 at $t=t_0$, so that $\SF(A_t:t\in[0,1])$ jumps by 1, and so the canonical flag $f_{N_t}^{P'}$ jumps by 1 as $t$ passes through~$t_0$.

Thus the canonical flag $f_N^{P'}$ has the property we want of $w((P',A'),N,1)$: under deformations of $(X,\vp_t,\psi_t)$, $f_N^{P'}$ changes by addition of $k\in\Z$ exactly when $w((P',A'),N,1)$ should change by addition of $k\in\Z$. Unfortunately, $f_N^{P'}$ is not a number, as $w((P',A'),N,1)$ should be, but a geometric structure on~$N$.

\subsection{\texorpdfstring{Problems with counting $G_2$-instantons}{Problems with counting G₂-instantons}}
\label{ca84}

Based on the ideas and results of Donaldson--Segal and Walpuski described in \S\ref{ca81}, and the material on $P'$-flags in \S\ref{ca83}, the author expects that the following is a possible (or at least plausible) behaviour for moduli spaces of $G_2$-instantons and associative 3-folds under smooth deformations of TA-$G_2$-manifolds:

\begin{ex} Suppose we are given a smooth family of compact TA-$G_2$-manifolds $(X,\vp_t,\psi_t)$, $t\in[0,1]$, supporting $G_2$-instantons and associative 3-folds as follows:
\begin{itemize}
\setlength{\itemsep}{0pt}
\setlength{\parsep}{0pt}
\item[(a)] There is an unobstructed $G_2$-instanton $(P',A'_t)$ on $(X,\vp_t,\psi_t)$ with structure group $\SU(2)$ for $t\in[0,1]$, depending smoothly on $t$.
\item[(b)] For $t\in[0,\frac{1}{3})$, $t\in(\frac{2}{3},1]$ there are no associatives of interest in~$(X,\vp_t,\psi_t)$.
\item[(c)] For $t\in(\frac{1}{3},\frac{2}{3})$ there are two associatives $N_t^+,N_t^-$ in $(X,\vp_t,\psi_t)$, depending smoothly on $t$. They are unobstructed, in the same homology class, with orientations $\Or(N_t^+)=1$, $\Or(N_t^-)=-1$.
\item[(d)] There are associatives $N_{1/3}$ in $(X,\vp_{1/3},\psi_{1/3})$ and $N_{2/3}$ in $(X,\vp_{2/3},\psi_{2/3})$. They are obstructed, with $\O_{N_{1/3}}\cong\R\cong \O_{N_{2/3}}$. We have $N_t^\pm\ra N_{1/3}$ as $t\ra \frac{1}{3}_+$ and $N_t^\pm\ra N_{2/3}$ as $t\ra \frac{2}{3}_-$, as in~\S\ref{ca72}(A).
\item[(e)] All of $N_t^\pm,N_{1/3},N_{2/3}$ are diffeomorphic to a fixed compact, oriented 3-manifold $N$, such as $N=\cS^3$. For simplicity we suppose $H_1(N;\Z_2)=0$, so that $N$ has a unique spin structure.
\end{itemize}

Let us now ask: how many $G_2$-instantons $(P,A_t)$ with structure group $\SU(2)$ are created or destroyed by bubbling a 1-instanton along $N_t^\pm$ from $(P',A'_t)$, as $t$ increases over $[0,1]$, as described in Examples \ref{ca8ex1} and \ref{ca8ex2}?

Consider the oriented 4-manifold $M\cong N\t\cS^1$ (or a twisted product) made of the disjoint union of $N_t^\pm$, $t\in(\frac{1}{3},\frac{2}{3})$ and $N_{1/3},N_{2/3}$ glued together in the obvious way, with its natural map $M\ra X$ from the inclusions $N_t^\pm,N_{1/3},N_{2/3}\ra X$. On $M$ we have a rank 4 oriented vector bundle $E\ra M$ restricting to the twisted spin bundles $\bS_{P'}$ on each slice $N_t^\pm,N_{1/3},N_{2/3}$, where $\bS_{P'}$ is unique as the spin structures on $N_t^\pm,N_{1/3},N_{2/3}\cong N$ are unique. The number of zeroes of a generic section of $E\ra M$, counted with signs, is~$k:=\int_Mc_2(P')$. 

Suppose no $G_2$-instantons $(P,A_t)$ are created or destroyed over $t\in[0,1]$. Then the canonical $P'$-flags $f^{P'}_{N_t^\pm},f^{P'}_{N_{1/3}},f^{P'}_{N_{2/3}}$ do not jump, and vary continuously. Therefore we can choose nonvanishing sections $s_t^\pm,s_{1/3},s_{2/3}$ of $\bS_{P'}$ on $N_t^\pm,N_{1/3},N_{2/3}$ representing $f^{P'}_{N_t^\pm},f^{P'}_{N_{1/3}},f^{P'}_{N_{2/3}}$ and varying continuously with $t$, and these $s_t^\pm,s_{1/3},s_{2/3}$ make up a continuous, nonvanishing section of $E\ra M$, so that $k=0$. In general, $k$ counts the jumps of $f^{P'}_{N_t^\pm}$ as $t$ increases over $[0,1]$, so we create or destroy $k$ new $G_2$-instantons $(P,A_t)$ as $t$ increases from 0 to~1.

We expect that we can have $k\ne 0$ in $\Z$ in examples. Thus, we can have:
\begin{itemize}
\setlength{\itemsep}{0pt}
\setlength{\parsep}{0pt}
\item[(i)] In $(X,\vp_0,\psi_0)$ one $G_2$-instanton $(P',A'_0)$ and no $G_2$-instantons on $P$, where $P\ra X$ is the principal $\SU(2)$-bundle obtained from $P'$ by gluing in a 1-instanton along $N^+_t$, and there are no associative 3-folds of interest.
\item[(ii)] In $(X,\vp_1,\psi_1)$ one $G_2$-instanton $(P',A'_1)$, and $k\ne 0$  $G_2$-instantons on $P$ counted with signs, and no associative 3-folds of interest.
\end{itemize}
Hence, in \eq{ca8eq2} we have $DS(P,\psi_0)=0$ and $DS(P,\psi_1)=k\ne 0$, so $DS(P,\psi)$ is not deformation-invariant.
\label{ca8ex3}	
\end{ex}

If Example \ref{ca8ex3} is true to mathematical reality, it demonstrates a problem with the Donaldson--Segal proposal \cite[\S 6.2]{DoSe} for defining invariants $DS(P,\psi)$ in \eq{ca8eq2}. Note that the actual choice of `compensation terms' $w((P',A'),N,k)$ is {\it irrelevant}, since in our example there are no associatives in $(X,\vp_0,\psi_0)$ or in $(X,\vp_1,\psi_1)$, so the second sum in \eq{ca8eq2} is automatically zero. However, we can trace the failure to difficulties in defining $w((P',A'),N,1)$ compensating for $\SU(2)$-instantons with charge 1 bubbling along $N$ in the way Donaldson and Segal want. We discuss a possible solution to this problem in~\S\ref{ca85}.

Here is another thought-experiment similar to Example~\ref{ca8ex3}:

\begin{ex} Suppose we are given a smooth family of compact TA-$G_2$-manifolds $(X,\vp_t,\psi_t)$, $t\in[0,1]$, and a principal $\SU(2)$-bundle $P'\ra X$ with $c_2(P')\ne 0$ in $H^4(X;\Q)$, supporting $G_2$-instantons and associatives as follows:
\begin{itemize}
\setlength{\itemsep}{0pt}
\setlength{\parsep}{0pt}
\item[(a)] There is an unobstructed associative 3-fold $N_t$ in $(X,\vp_t,\psi_t)$ for $t\in[0,1]$, depending smoothly on $t$. For simplicity we suppose $N_t$ is connected with $H_1(N_t;\Z_2)=0$, say $N_t\cong\cS^3$, so that $N_t$ has a unique spin structure.
\item[(b)] For $t\in[0,\frac{1}{3})$, $t\in(\frac{2}{3},1]$ there are no $G_2$-instantons on $P'$ over $(X,\vp_t,\psi_t)$. 
\item[(c)] For $t\in(\frac{1}{3},\frac{2}{3})$ there are two gauge equivalence classes $[A^{\prime +}_t],[A^{\prime -}_t]$ of $G_2$-instantons on $P'$ over $(X,\vp_t,\psi_t)$, depending smoothly on $t$. They are unobstructed, with orientations $\Or([A^{\prime +}_t])=1$ and $\Or([A^{\prime -}_t])=-1$.
\item[(d)] There are gauge equivalence classes $[A'_{1/3}]$ and $[A'_{2/3}]$ of $G_2$-instantons on $P'$ over $(X,\vp_{1/3},\psi_{1/3})$ and $(X,\vp_{2/3},\psi_{2/3})$, respectively. They are both obstructed, with obstruction space $\R$. We have $[A^{\prime \pm}_t]\ra [A'_{1/3}]$ as $t\ra\frac{1}{3}_+$ and $[A^{\prime \pm}_t]\ra [A'_{2/3}]$ as $t\ra \frac{2}{3}_-$.
\end{itemize}

Consider the problem of lifting the gauge equivalence classes $[A^{\prime +}_t],\ab[A^{\prime -}_t],\ab[A'_{1/3}],\ab[A'_{2/3}]$ to connections $A^{\prime +}_t,A^{\prime -}_t,A'_{1/3},A'_{2/3}$ on $P$ depending continuously on $t$. As we are dealing with a loop of connections, there may be monodromy. That is, we can choose $A^{\prime +}_t,A^{\prime -}_t,A'_{1/3},A'_{2/3}$ such that $A^{\prime \pm}_t$ depend smoothly on $t\in(\frac{1}{3},\frac{2}{3})$, and $A^{\prime \pm}_t\ra A'_{1/3}$ as $t\ra\frac{1}{3}_+$, and $A^{\prime +}_t\ra A'_{2/3}$ as $t\ra \frac{2}{3}_-$. But we cannot also ensure that $A^{\prime -}_t\ra A'_{2/3}$ as $t\ra \frac{2}{3}_-$. Instead, we can only suppose that $A^{\prime -}_t\ra \ga\cdot A'_{2/3}$ for some smooth gauge transformation $\ga:X\ra\SU(2)$, which may induce a nontrivial map $\ga_*:H_3(X;\Z)\ra H_3(\SU(2);\Z)\cong\Z$. Write $(\ga\vert_N)_*:\Z\cong H_3(N;\Z)\ra H_3(\SU(2);\Z)\cong\Z$ as multiplication by $k\in\Z$. We expect that we can have $k\ne 0$ in $\Z$ in examples.

Let $P\ra X$ be the principal $\SU(2)$-bundle obtained from $P'$ by gluing in a family of instantons of charge 1 along $N_t$. The author expects that by a similar calculation to that in Example \ref{ca8ex3} one can show that $k$ $G_2$-instantons $(P,A_t)$ are created or destroyed by bubbling a 1-instanton along $N_t$ from $(P',A^{\prime\pm}_t)$, as $t$ increases over $[0,1]$, counted with signs. Thus, we can have:
\begin{itemize}
\setlength{\itemsep}{0pt}
\setlength{\parsep}{0pt}
\item[(i)] In $(X,\vp_0,\psi_0)$ there is one associative $N_0$, and no $G_2$-instantons of interest.
\item[(ii)] In $(X,\vp_1,\psi_1)$ there is one associative $N_1$, and $k\ne 0$ $G_2$-instantons on $P$, counted with signs, and no other $G_2$-instantons of interest.
\end{itemize}
In \eq{ca8eq2} we have $DS(P,\psi_0)=0$ and $DS(P,\psi_1)=k\ne 0$, so $DS(P,\psi)$ is not deformation-invariant. There are no contributions to $DS(P,\psi_0),DS(P,\psi_1)$ from $N_0,N_1$, as there are no $G_2$-instantons on $P'$ over $(X,\vp_0,\psi_0)$ or~$(X,\vp_1,\psi_1)$.

\label{ca8ex4}	
\end{ex}

Again, if Example \ref{ca8ex4} is true to mathematical reality, it demonstrates a problem with the Donaldson--Segal proposal \cite[\S 6.2]{DoSe}, which we discuss in~\S\ref{ca85}.

\subsection{A suggestion for how to modify Donaldson--Segal}
\label{ca85}

Examples \ref{ca8ex3} and \ref{ca8ex4} indicate that Donaldson and Segal's proposed invariants $DS(P,\psi)$ in \eq{ca8eq2} will not be deformation-invariant. However, all may not be lost. We now outline a way to modify the Donaldson--Segal programme to hopefully fix these problems. We summarize our main points as (i),(ii), \ldots. This is not a complete proposal, just the beginnings of a possible answer.

While counting $G_2$-instantons and counting associative 3-folds are linked, counting associative 3-folds is the more primitive problem, as one can count associatives on their own, but to count $G_2$-instantons with any hope of deformation-invariance, one must count associative 3-folds too. So we should really start with the problem of counting associative 3-folds. The author expects that it should only be possible to count $G_2$-instantons on $(X,\vp,\psi)$ if counting associative 3-folds on $(X,\vp,\psi)$ is well-behaved, by which we mean:
\begin{itemize}
\setlength{\itemsep}{0pt}
\setlength{\parsep}{0pt}
\item[(i)] The Donaldson--Segal programme for counting $G_2$-instantons on a TA-$G_2$-manifold $(X,\vp,\psi)$, giving an answer independent of deformations of $\psi$, should only work if $(X,\vp,\psi)$ is {\it unobstructed\/} in the sense of Definition~\ref{ca7def2}.
\end{itemize}

In Example \ref{ca8ex4}, $(X,\vp_t,\psi_t)$ is obstructed by the associative 3-fold $N_t$. The author expects that the change in invariants $DS(P,\psi_t)$ in Example \ref{ca8ex4} under deformations of $\vp_t,\psi_t$ is typical for deformations of obstructed TA-$G_2$-manifolds $(X,\vp,\psi)$. The author knows of no way to add compensation terms to restore deformation-invariance in the obstructed case.
\begin{itemize}
\setlength{\itemsep}{0pt}
\setlength{\parsep}{0pt}
\item[(ii)] If $(X,\vp,\psi)$ is unobstructed then $\Phi_\psi:\cU\ra\La_{>0}$ in \S\ref{ca73} has at least one critical point $\th\in\cU$, but this critical point may not be unique. To get deformation-invariant information from counting $G_2$-instantons on $(X,\vp,\psi)$, we should first make a {\it choice of critical point\/} $\th$ of $\Phi_\psi$, and whatever invariants $\,\,\widehat{\!\!DS\!}\,(\psi,\th)$ we define {\it should depend on this choice of\/}~$\th$.
\item[(iii)] Suppose we are given a smooth 1-parameter family $(X,\vp_t,\psi_t)$, $t\in[0,1]$ of TA-$G_2$-manifolds with $[\vp_t]$ constant in $H^3(X;\R)$. Then as in Conjecture \ref{ca7conj1}, there should exist a natural quasi-identity morphism $\Up:\cU\ra\cU$ with $\Phi_{\psi_1}=\Phi_{\psi_0}\ci\Up$. We think of $\Up$ as a kind of `parallel translation' of associative 3-fold counting data along the family $(X,\vp_t,\psi_t)$, $t\in[0,1]$.

Now $\Up$ gives a bijection $\Crit(\Phi_{\psi_1})\ra\Crit(\Phi_{\psi_0})$. The correct meaning of {\it deformation-invariance\/} for the Donaldson--Segal style invariants $\,\,\widehat{\!\!DS\!}\,(\psi,\th)$ in (ii) should be that $\,\,\widehat{\!\!DS\!}\,(\psi_0,\th_0)=\,\,\widehat{\!\!DS\!}\,(\psi_1,\th_1)$ whenever $\th_0\in\Crit(\Phi_{\psi_0})$ and $\th_1\in\Crit(\Phi_{\psi_1})$ with~$\Up(\th_1)=\th_0$.
\item[(iv)] If we follow (ii)--(iii), we generally cannot make invariants $DS(P,\psi)$ for each principal $\SU(2)$-bundle $P\ra X$, as in \eq{ca8eq2} (though see Remark \ref{ca8rem1} below). Instead, we should aim to make one invariant $\,\,\widehat{\!\!DS\!}\,(\psi,\th)$ in $\La_{>0}$, as a formal power series similar to \eq{ca7eq4}, roughly of the form
\e
\,\,\widehat{\!\!DS\!}\,(\psi,\th)=\sum_{\text{$P\ra X$ principal $\SU(2)$-bundle}\!\!\!\!\!\!\!\!\!\!\!\!\!\!\!\!\!\!\!\!\!\!\!\!\!\!\!\!\!\!\!\!\!\!\!\!\!\!\!\!\!\!\!\!\!\!\!\!\!\!\!\!\!\!\!\!\!\!\!\!\!}
DS(P,\psi)\, q^{-4\pi^2\int_X[\vp]\cup c_2(P)} +\text{correction terms.}
\label{ca8eq3}
\e

\end{itemize}

In (i)--(iv) the author is motivated by an analogy with the Lagrangian Floer theory of Fukaya, Oh, Ohta and Ono \cite{Fuka2,FOOO}. Here for a Lagrangian $L$ in a symplectic manifold $(S,\om)$, one needs to choose a `bounding cochain' $\th$ for $L$ in homological algebra over a Novikov ring $\La_{>0}$. Such $\th$ need not exist or be unique, and we call $L$ `unobstructed' if $\th$ exists. When $(S,\om)$ is a symplectic Calabi--Yau 3-fold, $\th$ corresponds to the critical point of a superpotential $\Phi_J:\cU\ra\La_{>0}$. There is a notion of `parallel translation' of bounding cochains $\th$ along smooth families $L_t$, $t\in[0,1]$ of Hamiltonian isotopic Lagrangians.

We can now explain how to deal with Example \ref{ca8ex3} in our modified proposal. In Example \ref{ca8ex3}, at least when $t\in[0,\frac{1}{3})$ and $t\in(\frac{2}{3},1]$, there are no associative 3-folds in $(X,\vp_t,\psi_t)$, so $\Phi_{\psi_t}\equiv 0$ and $\Crit(\Phi_{\psi_t})=\cU$, and the extra data $\th_t$ required in (ii)--(iii) is an arbitrary element of $\cU$. We could take $\th_0$ to be the constant function $1:H_3(X;\Z)\ra 1+\La_{>0}$, but there are many other choices.

In Example \ref{ca8ex3} there are no associatives at $t=0$ and at $t=1$, so you might think that nothing changes between $t=0$ and $t=1$ from the point of view of counting associatives. However, the map $\Up:\cU\ra\cU$ from `parallel translation' along $(X,\vp_t,\psi_t)$, $t\in[0,1]$ will in general not be the identity, but will depend on the (co)homology classes $[\vp_t]\in H^3(X;\R)$, $[N_t]\in H_3(X;\Z)$ and $[M]\in H_4(X;\Z)$, as $\Up_0$ in \eq{ca7eq7} depends on $\ga$, $[\ti N_0]$ and $\de$. So if $\th_0=1$, we may not have $\th_1=1$. The difference in $G_2$-instanton counting between $t=0$ and $t=1$ should be compensated for by the difference between $\th_0$ and~$\th_1$.

Our proposal for counting associative 3-folds in \S\ref{ca7} involves counting only associative $\Q$-homology spheres. However, in the Donaldson--Segal picture, $G_2$-instantons $(P,A)$ might bubble on any compact associative 3-fold $N$, not just $\Q$-homology 3-spheres, and in \eq{ca8eq2} we must allow $N$ to be a general associative 3-fold. So we should explain how to bridge the gap between associative $\Q$-homology 3-spheres, and general associative 3-folds. 

Haydys and Walpuski \cite[\S 1]{HaWa} briefly outline a method for defining the `compensation terms' $w((P',A'),N,1)$ required by Donaldson and Segal, as in \S\ref{ca81}. They fix a line bundle $L\ra N$, and a spin structure on $N$ with spin bundle $\bS$, and consider moduli spaces  $\cM_{(P,A'),N}$ of solutions $(B,\Psi)$ of the Seiberg--Witten type equations $\bD_{B\ot A'}\Psi=0$, $F_B=\mu(\Psi)$ on $N$, where $B$ is a $\U(1)$-connection on $L$ with curvature $F_B$, and $\Psi:\ad(P)\vert_N\ra\bS\ot L$ is a vector bundle morphism, and $\bD_{B\ot A'}$ is a twisted Dirac operator, and $\mu$ is a natural quadratic bundle map. Then $\cM_{(P,A'),N}$ has virtual dimension 0, and they wish to define $w((P',A'),N,1)$ to be the virtual count~$[\cM_{(P,A'),N}]_{\rm virt}\in\Z$.

We need to understand how $w((P',A'),N,1)=[\cM_{(P,A'),N}]_{\rm virt}$ can change under deformations of $(X,\vp_t,\psi_t)$, as a result of noncompactness or singularities in the moduli spaces $\cM_{(P,A'),N}$. There are two ways in which this can happen:
\begin{itemize}
\setlength{\itemsep}{0pt}
\setlength{\parsep}{0pt}
\item[(A)] There may be a family of solutions $(B_t,\Psi_t)$ with $\nm{\Psi_t}_{L^2}\ra\iy$ as $t\ra t_0$. The main result of \cite{HaWa} is that a rescaled limit of the $(B_t,\Psi_t)$ converges to a solution of the Fueter equation which controls bubbling of $G_2$-instantons along $N$, as in Examples \ref{ca8ex1}--\ref{ca8ex2}. Thus, Haydys and Walpuski hope that changes in $w((P',A'),N,1)$ resulting from such limits will exactly cancel changes in $G_2$-instanton counting, so that \eq{ca8eq2} is unchanged.
\item[(B)] There may be a family of solutions $(B_t,\Psi_t)$ with $\nm{\Psi_t}_{L^2}\ra 0$ as $t\ra t_0$. While this does not cause noncompactness in $\cM_{(P,A'),N}$, there is a problem in defining the virtual count $[\cM_{(P,A'),N}]_{\rm virt}$ near solutions $(B,\Psi)$ with $\Psi=0$, as $(B,0)$ has stabilizer group $\U(1)$, so $[\cM_{(P,A'),N}]_{\rm virt}$ may change.

When $\Psi=0$ the equation $F_B=\mu(\Psi)$ becomes $F_B=0$, so $(L,B)$ is a flat $\U(1)$-line bundle on $N$. It turns out that $[\cM_{(P,A'),N}]_{\rm virt}$ only changes under such transitions if the moduli space of such $(L,B)$ has dimension 0, that is, if $b^1(N)=0$, so that $N$ is a $\Q$-homology 3-sphere.
\end{itemize}

Our conclusion is that the Haydys--Walpuski proposal for $w((P',A'),N,1)$ in \eq{ca8eq2} has problems for associative 3-folds $N$ which are $\Q$-{\it homology $3$-spheres}, and these problems also involve {\it flat $\U(1)$-line bundles on\/} $N$. Observe that this looks very similar to the programme of \S\ref{ca7}, which involves counting associative $\Q$-homology 3-spheres $N$ with weight $I(N)=\md{H_1(N;\Z)}$ in \eq{ca5eq12}, which is the number of flat $\U(1)$-line bundles on~$N$.

The author's rough idea is to add some extra terms to \eq{ca8eq2}, involving the choice of critical point $\th$ of $\Phi_\psi$ in (ii) above, whose changes under deformations would cancel the changes of type (B) to the Haydys--Walpuski terms, making the sum deformation-invariant. The author does not yet know how to do this.
\begin{itemize}
\setlength{\itemsep}{0pt}
\setlength{\parsep}{0pt}
\item[(v)] The invariant $\widehat{\!\!DS\!}\,(\psi,\th)$ envisaged in (iv) should roughly be the sum of products of three kinds of terms: (a) terms counting $G_2$-instantons, as for $\sum_{[A]\in \cM(P,\psi)}\Or([A])$ in \eq{ca8eq1}; (b) Haydys--Walpuski style compensation terms \cite{HaWa}; and (c) terms involving the chosen critical point $\th$ of $\Phi_\psi$.
\end{itemize}
This concludes our outline of modifications to the Donaldson--Segal programme.

\begin{rem}{\bf(a)} From \S\ref{ca71} we have been working in the ideal $\La_{>0}$ in the Novikov ring $\La$ in \eq{ca7eq1}. So for instance, setting $\th=1$ in \eq{ca7eq5} gives
\e
\Phi_\psi(1)=\sum_{\al\in H_3(X;\Z):[\vp]\cdot\al>0}GW_{\psi,\al}\, q^{[\vp]\cdot\al}\qquad\text{in $\La_{>0}$.}
\label{ca8eq4}
\e
The coefficient of $q^a$ in $\Phi_\psi(1)$ is $\sum_{\al:[\vp]\cdot\al=a}GW_{\psi,\al}$, that is, it `counts' associative 3-folds $N$ in $(X,\vp,\psi)$ with area $a$. Similarly, the coefficient of $q^a$ in $\,\,\widehat{\!\!DS\!}\,(\psi,\th)$ in \eq{ca8eq3} morally `counts' $G_2$-instantons $(P,A)$ with energy~$-4\pi^2\int_X[\vp]\cup c_2(P)=a$. 

The effect of working in $\La_{>0}$ like this is that we only get one counting invariant for each area or energy $a>0$, so homology classes $\al$ with the same area, or principal bundles $P$ with the same energy, get lumped together.
\smallskip

\noindent{\bf(b)} If $[\vp]$ is generic in $H^3(X;\R)$ then $[\vp]\cdot: H_3(X;\Z)/\text{torsion}\ra\R$ is injective, so invariants in $\La_{>0}$ of the form \eq{ca8eq4} give an invariant for each class $\al$ in $H_3(X;\Z)/\text{torsion}$ or $c_2(P)$ in $H^4(X;\Z)/\text{torsion}$, which is not far from the system of invariants hoped for in the Donaldson--Segal proposal in~\S\ref{ca81}.

However, there is a catch. If $[\vp]$ is generic, and the superpotential $\Phi_\psi$ in \eq{ca74}--\eq{ca7eq5} is not identically zero, then one can show that $\d\Phi_\psi(\th)\ne 0$ for all $\th\in\cU$, as the term in $\d\Phi_\psi(\th)$ from $\al\in H_3(X;\Z)$ with $GW_{\psi,\al}\ne 0$ and $[\vp]\cdot\al$ least dominates all others. So $\Phi_\psi$ has no critical points, and $(X,\vp,\psi)$ is obstructed.

If $\Phi_\psi\not\equiv 0$ then $\Phi_\psi$ can only have critical points if there exist one or more pairs $\al_1,\al_2$ in $H_3(X;\Z)$ with $GW_{\psi,\al_1}\ne 0$, $GW_{\psi,\al_2}\ne 0$, $\al_1\ne\al_2$ and $[\vp]\cdot\al_1=[\vp]\cdot\al_2$, so that the obstructions from $\al_1,\al_2$ cancel out. Then $\al_1-\al_2$ lies in the kernel of $[\vp]\cdot: H_3(X;\Z)/\text{torsion}\ra\R$, and principal $\SU(2)$-bundles $P,P'\ra X$ such that $c_2(P)-c_2(P')$ lies in the subspace of $H^4(X;\Q)$ spanned by $\mathop{\rm Pd}(\al_1-\al_2)$ for all such pairs $\al_1,\al_2$ contribute to the same $G_2$-instanton counting invariant.

\smallskip

\noindent{\bf(c)} As in {\bf(b)}, for a TA-$G_2$-manifold $(X,\vp,\psi)$ we have a dichotomy: either
\begin{itemize}
\setlength{\itemsep}{0pt}
\setlength{\parsep}{0pt}
\item[(i)] $\Phi_\psi\equiv 0$. Then all associative 3-fold counting invariants are trivial. We can take $[\vp]$ generic in $H^3(X;\R)$, and hope to define $G_2$-instanton counting invariants $DS(\al,\psi,\th)\in\F$ for all $\al\in H^4(X;\Z)/\text{torsion}$, depending on a choice of $\th\in\cU$.
\item[(ii)] $\Phi_\psi\not\equiv 0$. Then we must choose a critical point $\th$ of $\Phi_\psi$, which can only exist if $[\vp]$ lies in some proper vector subspace $V$ of $H^3(X;\R)$, and hope to define Donaldson--Segal style counting invariants $DS(\al,\psi,\th)\in\F$ parametrized by $\al$ in $H^4(X;\Z)/W$ for $W=\Ker([\vp]\cup-)\subseteq H^4(X;\Z)$ with~$\mathop{\rm rank} W>0$.
\end{itemize}

Here is an interesting special case of (i). Take $X=Y\t\cS^1$ for $Y$ a Calabi--Yau 3-fold, and initially take $(\vp,\psi)$ to be an $\cS^1$-invariant TA-$G_2$-structure on $X$, e.g.\ the torsion-free $G_2$-structure induced by a Calabi--Yau structure on~$Y$. 

Let $N$ be a compact associative 3-fold in $(X,\vp,\psi)$. If $N$ is $\cS^1$-invariant then $N\cong\cS^1\t\Si$ for some surface $\Si\subset Y$, so $N$ is not a $\Q$-homology 3-sphere, and it contributes zero to the superpotential $\Phi_\psi$ in \S\ref{ca73}. If $N$ is not $\cS^1$-invariant then it lies in an $\cS^1$-family $\bigl\{e^{i\th}\cdot N:e^{i\th}\in\U(1)\bigr\}$ of associative 3-folds in $X$, and this family also contributes zero to $\Phi_\psi$, as $\chi(\cS^1)=0$. Thus $\Phi_\psi\equiv 0$, as in~(i).

By $\cS^1$-localization we expect that counting $G_2$-instantons on $(X,\vp,\psi)$ gives the same answer as counting $\cS^1$-invariant $G_2$-instantons on $(X,\vp,\psi)$, which is equivalent to counting solutions of a gauge theoretic equation on $Y$, essentially the `Donaldson--Thomas instantons' considered by Tanaka \cite{Tana}. The invariants may be an analytic version of some form of the algebro-geometric Donaldson--Thomas invariants of $Y$, as in Thomas \cite{Thom} and Joyce and Song~\cite{Joyc21,JoSo}.
\label{ca8rem1}
\end{rem}

\medskip

\noindent{\small\sc The Mathematical Institute, Radcliffe
Observatory Quarter, Woodstock Road, Oxford, OX2 6GG, U.K.

\noindent E-mail: {\tt joyce@maths.ox.ac.uk.}}

\end{document}